\documentclass[12pt]{article}
\usepackage{setspace}

\usepackage{inputenc}
\usepackage{graphicx,graphics}
\usepackage{supertabular}
\usepackage{amsmath,amscd,amsfonts,amssymb}
\usepackage{fancyhdr}
\usepackage{color}
\usepackage{verbatim}

\usepackage{array,supertabular}
\usepackage{float}
\usepackage{rotating}
\usepackage{lscape}

\usepackage{booktabs,caption,fixltx2e}
\usepackage[flushleft]{threeparttable}
\usepackage{amsthm}

\usepackage{dcolumn}
\newcolumntype{L}{D{.}{.}{2,5}}

\usepackage{chngpage} 

\newtheorem{theorem}{Theorem}[section]
\newtheorem{lemma}[theorem]{Lemma}

\newtheorem{proposition}[theorem]{Proposition}
\newtheorem{corollary}[theorem]{Corollary}

\newtheorem{assumption}{Assumption}

\newenvironment{definition}[1][Definition.]{\begin{trivlist}
\item[\hskip \labelsep {\bfseries #1}]}{\end{trivlist}}

\newenvironment{remark}[1][Remark.]{\begin{trivlist}
\item[\hskip \labelsep {\bfseries #1}]}{\end{trivlist}}

\def\mds{\medskip}
\def\Rb{{\mathbb R}}
\def\Pc{{\mathcal P}}
\def\Fc{{\mathcal P}}
\def\Nc{{\mathcal N}}

\long\def\symbolfootnote[#1]#2{\begingroup%
\def\thefootnote{\fnsymbol{footnote}}\footnotetext[#1]{#2}\footnotemark[#1]\endgroup}

\makeatletter
\@addtoreset{equation}{section}

\makeatother

\newcounter{Fig}[figure]

\newcounter{Tab}[table]

  {\refstepcounter{Tab}
   \addtocounter{table}{1}
   \samepage\vspace{0.2cm}
   \centerline{#1} \nobreak
   \begin{verse}{Table \thesection.\arabic{Tab}:}
  }%
  {\end{verse} \vspace{0.2cm}}

\relax

\newcommand{\ZZ}{\mbox{\boldmath $Z$}}

\newcommand{\BB}{\mbox{\boldmath $B$}}

\newcommand{\MM}{\mbox{\boldmath $M$}}
\newcommand{\xx}{\mbox{\boldmath $x$}}

\newcommand{\cc}{\mbox{\boldmath $c$}}

\newcommand{\pp}{\mbox{\boldmath $p$}}
\newcommand{\uu}{\mbox{\boldmath $u$}}

\newcommand{\ww}{\mbox{\boldmath $w$}}

\newcommand{\zz}{\mbox{\boldmath $z$}}

\def \Cb{{\mathbb C}}

\def \Eb{{\mathbb E}}
\def \Fb{{\mathbb F}}
\def \Gb{{\mathbb G}}
\def \Hb{{\mathbb H}}
\def \Kb{{\mathbb K}}

\def \Mb{{\mathbb M}}

\def \Pb{{\mathbb P}}
\def \Rb{{\mathbb R}}

\def \Vb{{\mathbb V}}

\def \Ac{{\mathcal A}}
\def \Rc{{\mathcal R}}

\def \Bc{{\mathcal B}}
\def \Gc{{\mathcal G}}

\def \Uc{{\mathcal U}}

\def \Fc{{\mathcal F}}
\def \Cc{{\mathcal C}}
\def \Sc{{\mathcal S}}

\def \Mc{{\mathcal M}}
\def \Nc{{\mathcal N}}
\def \Pc{{\mathcal P}}

\newcommand{\bqa}{\begin{eqnarray*}}
\newcommand{\eqa}{\end{eqnarray*}}
\newcommand{\bqan}{\begin{eqnarray}}
\newcommand{\eqan}{\end{eqnarray}}
\newcommand{\bqt}{\begin{quote}}
\newcommand{\eqt}{\end{quote}}
\newcommand{\bt}{\begin{tabbing}}
\newcommand{\et}{\end{tabbing}}
\newcommand{\bit}{\begin{itemize}}
\newcommand{\eit}{\end{itemize}}
\newcommand{\ben}{\begin{enumerate}}
\newcommand{\een}{\end{enumerate}}
\newcommand{\beq}{\begin{equation}}
\newcommand{\eeq}{\end{equation}}
\newcommand{\beqw}{\begin{equation*}}
\newcommand{\eeqw}{\end{equation*}}
\newcommand{\bdefi}{\begin{definition}}
\newcommand{\edefi}{\end{definition}}
\newcommand{\bpro}{\begin{proposition}}
\newcommand{\epro}{\end{proposition}}
\newcommand{\blem}{\begin{lemma}}
\newcommand{\elem}{\end{lemma}}
\newcommand{\bco}{\begin{corollary}}
\newcommand{\eco}{\end{corollary}}
\newcommand{\bdes}{\begin{description}}
\newcommand{\edes}{\end{description}}
\newcommand{\bre}{\begin{remark}}
\newcommand{\ere}{\end{remark}}

\newcommand{\eps}{\epsilon}

\def\mds{\medskip}
\def\1{{\mathbf 1}}

\begin{document}

\title{Asymptotic Theory of the Sparse Group LASSO}
\author{Benjamin Poignard\footnote{ENSAE - CREST \& PSL - Paris Dauphine (CEREMADE). E-mail: benjamin.poignard@ensae.fr}}

\date{\today}
\maketitle

\abstract{This paper proposes a general framework for penalized convex empirical criteria and a new version of the Sparse-Group LASSO (SGL, Simon and al., 2013), called the adaptive SGL, where both penalties of the SGL are weighted by preliminary random coefficients. We explore extensively its asymptotic properties and prove that this estimator satisfies the so-called oracle property (Fan and Li, 2001), that is the sparsity based estimator recovers the true underlying sparse model and is asymptotically normally distributed. Then we study its asymptotic properties in a double-asymptotic framework, where the number of parameters diverges with the sample size. We show by simulations that the adaptive SGL outperforms other oracle-like methods in terms of estimation precision and variable selection.}

\mds

\noindent {\it Keywords}: Asymptotic Normality, Consistency, Model Selection, Oracle Property.

\newpage

\section{Introduction}

Model complexity is an obstacle when one models richly parameterized dynamics such as multivariate nonlinear dynamic systems. For instance, dynamic variance correlation processes of size $N$ have an $O(N^2)$ complexity as in the dynamic conditional correlation parametrization (DCC, Ding and Engle, 2001). Another issue arises when the sample size, say $T$, is comparable to $N$, which may reduce the estimation performances. This is typically a high-dimensional statistical framework.

\mds

A significant literature developed on model penalization, which consists of reducing the number of parameters and performing variable selection. For instance, the Akaike's or Bayesian information criteria aim at selecting the size of a model. However, these methods are unstable, computationally complex and their sampling properties are difficult to study as Fan and Li (2001) pointed out mainly because they are stepwise and subset selection procedures.

\mds

The LASSO procedure of Tibshirani (1996) overcomes these drawbacks as it simultaneously performs variable selection and model estimation. It then fosters sparsity and allows for continuity of the selected models. Other penalties were proposed such as the smoothly clipped absolute deviation (SCAD) of Fan, which modifies the LASSO to shrink large coefficients less severely. The elastic net regularization procedure of Zou and Hastie (2005) was developed to overcome the collinearity between the variables, which hampers the LASSO to perform well. Their idea consists of mixing a $l^1$ penalty, which performs variable selection, with a $l^2$ penalty, which stabilizes the solution paths. The Group LASSO of Yuan and Lin (2006) fosters sparsity and variable selection in a group of variables. Simon and al. (2013) designed the Sparse-Group LASSO (SGL) to foster sparsity both at a a group level and within a group. Their penalization involves a $l^1$ LASSO type penalty and a mixed $l^1/l^2$ penalty for group selection. All these procedures, together with the algorithms designed for performing selection and estimation, were developed within a linear framework. The penalized Ordinary Least Square (OLS) loss function is typically used for linear models as it is convex, which makes the computation easier, and allows for closed form solutions, such as the soft-thresholding operator for the LASSO penalty. Furthermore, linear modeling allows for deriving non asymptotic oracle inequalities straightforwardly: see B\"uhlmann and van de Geer (2011) on this non-asymptotic framework.

\mds

Knight and Fu (2000) explored the asymptotic properties of the LASSO penalty for OLS loss functions. Fan and Li (2001) proposed a penalization framework for general convex functions and studied the asymptotic properties of the SCAD penalty. They proved that the SCAD estimator satisfies the oracle property, that is the sparsity based estimator recovers the true underlying sparse model and is asymptotically normally distributed. The LASSO as proposed by Tibshirani cannot enjoy the oracle property. To fix this drawback, Zou (2006) proposed the adaptive LASSO within an OLS framework, where adaptive weights are used to penalize different coefficients in the penalty. Nardi and Rinaldo (2008) applied the same methodology for the Group LASSO estimator within an OLS framework and studied its oracle property.

\mds

These theoretical studies were developed for fixed dimensional models with i.i.d. data, a case where $N$ does not depend on the sample size, and for least square type loss functions, except Fan and Li (2001). Fan and Peng (2004) considered the general penalized convex likelihood framework when the number of parameters grows with the sample size and focused on the oracle property for general penalties. Zou and Zhang (2009) also focused on the oracle property of the adaptive elastic-net within the double-asymptotic framework. Their work highlights that adaptive weights penalizing different coefficients are key quantities to enjoy the oracle property as one can modify the convergence rate of the penalty terms. Nardi and Rinaldo (2008) also proposed within the double-asymptotic setting selection consistency results, which states that asymptotically the right set of relevant variables is selected.

\mds

In this paper, we develop the asymptotic theory of penalized M-estimators for convex criteria and dependent variables and consider the asymptotic properties of the Sparse-Group LASSO estimator. This penalty is relevant for problems where one would like to foster sparsity for selecting active groups, that is a group for which some of the corresponding coefficients are non zero, and active coefficients within an active group, a situation where a coefficient is non zero within an active group. Hence this is somehow a two step approach as first the active groups are selected, and then the active variables within an active group are selected. We prove that the SGL as proposed by Simon and al. (2013) does not enjoy the oracle property. Then we propose a new version of the SGL, the adaptive SGL using the same methodology of Zou (2006), which consists of penalizing different coefficients and group of coefficients using random weights that are positive functions of a first step estimator. This enables to alter the rate of convergence of the penalties such that the adaptive SGL satisfies the oracle property. Our work is influenced by Fan and Peng (2004) concerning the oracle property for general penalized convex loss functions and by Zou and Zhang (2009) regarding the modeling of random weights penalizing the coefficients differently. They both considered a double-asymptotic framework. We also prove that the adaptive SGL enjoys the oracle property in a double-asymptotic framework, a situation where the model complexity grows with the sample size.

\mds

The rest of the paper is organized as follows. In section 2, we describe our general framework for penalized convex empirical criteria and the SGL penalty. In section 3, we derive the optimality conditions of the statistical criterion. In section 4, we derive the asymptotic properties of both the SGL and adaptive SGL when the number of parameters is fixed. In section 5, we prove the oracle property of the adaptive SGL in a double-asymptotic setting.

\section{Framework and notations}

We consider a dynamic system in which the criterion is written as an empirical criterion, that is
\beq \label{emp_criteria_unpen}
\theta \mapsto \Gb_T l(\theta) = \cfrac{1}{T}\overset{T}{\underset{t=1}{\sum}} l(\eps_t;\theta),
\eeq
such that $l(.)$ is "a general" known loss function on the sample space such that for any process $(\eps_t)$, $\theta \mapsto l(\eps_t;\theta)$ is convex. This framework encompasses for instance the maximum likelihood method, where the $l(.)$ function corresponds to $l(\eps_t;\theta) = -\log f(\eps_t;\theta)$, where $f(\eps_t;\theta)$ is the density of the observation $(\eps_t)$ under $\Pb_{\theta}$. Alternatively, a linear model would imply $l(\eps_t;\theta) = \|\eps^{(1)}_t - \theta' \eps^{(2)}_t\|_p$, where $(\eps^{(1)}_t,\eps^{(2)}_t) = \eps_t$. We denote the empirical score and Hessian of the empirical criterion respectively as
\beqw
\dot{\Gb}_T l(\theta) = \cfrac{1}{T}\overset{T}{\underset{t=1}{\sum}} \nabla_{\theta} l(\eps_t;\theta), \; \ddot{\Gb}_T l(\theta) = \cfrac{1}{T}\overset{T}{\underset{t=1}{\sum}} \nabla^2_{\theta \theta'} l(\eps_t;\theta).
\eeqw
The dependent nature of our framework requires the use of particular probabilistic tools to study the asymptotic properties of M-estimators. We extensively use the ergodic theorem and central limit theorem (Billingsley, 1961, 1995) to obtain convergence in probability of empirical quantities to their theoretical counterparts and central limit theorems. To do so, we assume the stationarity and the ergodicity of the underlying process $(\eps_t)$: see assumption \ref{H1} in section \ref{Section4}.

\mds

In this setting, $\epsilon_t \in \Rb^N$ and $\theta \in \Rb^d$, a vector that can be split into $m$ groups $\Gc_k, k = 1,\cdots,m$, such that $\text{card}(\Gc_k) = \cc_k$ and $\overset{m}{\underset{k=1}{\sum}} \cc_k = d$. We suppose no overlap between these groups. We use the notation $\theta^{(l)}$ as the subvector of $\theta$, that is the set $\{\theta_k: k \in \Gc_l\}$. Hence the vector $\theta = (\theta_j,j=1,\cdots,d)$ can be written as $\theta = (\theta^{(k)}_i, k \in \{1,\cdots,m\}, i = 1,\cdots,\cc_k)$ \footnote{Formally, there is a one-to-one mapping between two ways for writing $\theta$:
\beqw
\begin{array}{llll}
\psi: \{1,\cdots,d\} \rightarrow \{(k,i), k= 1,\dots,m; i = 1,\cdots,\cc_k\}, &&\\
j \mapsto \psi(j) = (k_j,i_j).&&
\end{array}
\eeqw
In the rest of this paper, this mapping is implicit such that we allow such writings as $j = (k,i)$ or $j = i_k$ where $k$ is clear.
}. We denote by $\theta_0$ the true parameter vector of interest. Moreover, $\theta \rightarrow \Eb[l(\eps_t;\theta)]$ is supposed to be a one-to-one mapping and is minimized uniquely at $\theta = \theta_0$.

\mds

We denote by $\Sc:=\{k: \theta^{(k)} \neq 0\}$ the set of indices for which the groups are active. Let $\Ac := \{j: \theta_{0,j} \neq 0\}$ be the true subset model, which can be decomposed into sub-groups of active sets as $l \in \Sc$, $\Ac_l = \{(l,i): \theta^{(l)}_{0,i} \neq 0\}$. Besides, there are inactive indices $ \Gc_l \setminus \Ac_l = \Ac^c_l = \{(l,i): \theta^{(l)}_{0,i} = 0 \}$. We have $\{l \notin \Sc\} \Leftrightarrow \{ \forall i = 1,\cdots,\cc_l, \theta^{(l)}_{0,i} = 0\}$. In this setting, $\Ac = \underset{l \in \Sc}{\cup}\Ac_l$ such that for $k \neq l, \, \Ac_k \cap \Ac_l = \emptyset$. Furthermore, $\Ac^c = \overset{m}{\underset{l = 1}{\cup}} \Ac^c_l$ such that for $k \neq l, \, \Ac^c_k \cap \Ac^c_l = \emptyset$.

\mds

Finally, we need the following notations: $\dot{\Gb}_T l(\theta)_{(k)} \in \Rb^{\cc_k}$ is the "score" vector of the empirical criterion taken over group $k$ of size $\cc_k$, $\dot{\Gb}_T l(\theta)_{(k),i} \in \Rb$ is the $i$-th component of this score, and $\dot{\Gb}_T l(\theta)_{\Ac} \in \Rb^{\text{card}(\Ac)}$ is the score over the set of active indices. $\ddot{\Gb}_T l(\theta)_{(k)(k)} \in \Mc_{\cc_k \times \cc_k}(\Rb)$ (resp. $\Hb_{(k)(k)}$) is the empirical (resp. theoretical) Hessian taken over the block representing group $k$, and $\ddot{\Gb}_T l(\theta)_{\Ac \Ac} \in \Mc_{\text{card}(\Ac) \times \text{card}(\Ac)}(\Rb)$ is the Hessian over the set of active indices.

\mds

The statistical problem consists of minimizing over the parameter space $\Theta$ a penalized criterion of the form
\beq \label{1stcrit}
\hat{\theta} = \underset{\theta \in \Theta}{\arg \, \min} \, \{\Gb_T \varphi(\theta)\},
\eeq
where
\beqw
\begin{array}{llll}
\theta \mapsto \Gb_T \varphi(\theta) & = & \cfrac{1}{T}\overset{T}{\underset{t=1}{\sum}} \{ l(\eps_t;\theta) + \pp_1(\lambda_T,\theta) +  \pp_2(\gamma_T,\theta)\} \\
& = & \Gb_T l(\theta) + \pp_1(\lambda_T,\theta) + \pp_2(\gamma_T,\theta).
\end{array}
\eeqw
and both penalties are specified as
\beqw
\left\{\begin{array}{llll}
\pp_1 : \Rb_+ \times \Rb^m_+ \times \Theta \rightarrow \Rb_+, & & \pp_2 : \Rb_+ \times \Rb^m_+ \times \Theta \rightarrow \Rb_+, \\
(\lambda_T,\alpha,\theta) \mapsto \pp_1(\lambda_T,\theta) = \lambda_T T^{-1} \overset{m}{\underset{k = 1}{\sum}} \alpha_k \|\theta^{(k)}\|_1, & & (\gamma_T,\xi,\theta) \mapsto \pp_2(\gamma_T,\theta) = \gamma_T T^{-1} \overset{m}{\underset{l = 1}{\sum}} \xi_l \|\theta^{(l)}\|_2.
\end{array} \right.
\eeqw
Both $\alpha_k$ and $\xi_l$ are non negative scalar quantities for each group and the tuning parameters $\lambda_T$ and $\gamma_T$ vary with $T$.

\mds

The estimator $\hat{\theta}$ obtained in (\ref{1stcrit}) is not the minimum of the empirical unpenalized criterion $\Gb_T l(.)$. Our main interest is to analyze the bias generated by the penalties and how the oracle property can be achieved in the sense of Fan and Li (2001). More precisely, the sparsity based estimator must satisfy
\beqw
\begin{array}{llll}
(i) \hat{\Ac} = \{i: \hat{\theta}_i \neq 0 \} = \Ac \; \text{asymptotically, that is "model selection consistency"}. & & \\
(ii) \sqrt{T}(\hat{\theta}_{\Ac} - \theta_{0,\Ac}) \overset{d}{\rightarrow} \Nc(0,\Vb_0) \; \text{with} \, \Vb_0 \, \text{a covariance matrix related to the criterion of interest}. &&
\end{array}
\eeqw

We highlight in Proposition \ref{proposition1}, section \ref{Section4} that actually the SGL as proposed by Simon and al. (2013) cannot perform the oracle property. Hence in section \ref{Section4}, we propose a new estimator based on the same idea as Zou (2006), the adaptive Sparse Group LASSO, for which the oracle property is obtained when the weights are randomized, as proved in Theorem \ref{oracle1}.

\mds

This framework can be adapted to a broad range of problem. For instance, one can penalize a subset of groups with a $l^1$ penalty only, and the other groups with a $l^1/l^2$ penalty only. This framework encompasses the SGL, the LASSO and the group LASSO for proper choices of $\alpha$'s and $\xi$'s.

\mds

Let us motivate the interets of the SGL approach and illustrate our notations through a simple linear example. In finance, finding the right set of explanatory variables to predict future asset returns is a significant issue. For instance, one may use Japanese companies indices, the Japanese GDP or the Japanese aggregated dividend-price ratio to explain the Nikkei index return through a linear projection. But one should also consider some foreign variables, such as the S\&P 500 index or the US yield curve. Consequently, some groups of variables naturally arise: group of financial companies, tech companies, and the like; group of foreign components such as American financial companies, and the like. Hence the set $\Gc_k$ may represent the $k$-th ($k \leq m$) group of Japanese financial companies, composed (as a shortcoming) with Nomura (index $1$), MUFG-Bank of Tokyo (index $2$) and Sumitomo (index 3) represented by the parameter vector $\theta^{(k)}=(\theta^{(k)}_1,\theta^{(k)}_2,\theta^{(k)}_3)$; then $k \in \Sc$ if the whole group has a statistically significant effect on the Nikkei index. Suppose the $l^1/l^2$ penalty selects this group as active. Then $\Ac_k$ represents the set of active components in $\Gc_k$ such that $\cc_{\Ac_k} = \text{card}(\Ac_k) \leq \text{card}(\Gc_k) = \cc_k$. The $l^1$ penalty fosters sparsity within this selected group. If Nomura is the only variable that is expelled, then $1 \in \Ac^c_k = \Gc_k \setminus \Ac_k$, whereas $\{2,3\} \in \Ac_k$ and $\cc_{\Ac_k} = 2$.

\section{Optimality conditions}

The statistical problem consists of solving (\ref{1stcrit}). Both $\Gb_T l(.)$, $\pp_1(\lambda_T,\alpha,.)$ and $\pp_2(\gamma_T,\xi,.)$ are convex functions and there are no inequality constraints. Consequently, by the Karush-Kuhn-Tucker optimality conditions, which are necessary and sufficient, the estimator $\hat{\theta}$ satisfies for a group $k$
\beq \label{opt1}
\dot{\Gb}_T l(\hat{\theta})_{(k)} + \lambda_T T^{-1} \alpha_k \hat{\ww}^{(k)} + \gamma_T T^{-1} \xi_k \hat{\zz}^{(k)} = 0,
\eeq
for some vectors $\ww^{(k)}$ and $\zz^{(k)}$ satisfying
\beqw
\hat{\ww}^{(k)} = \begin{cases}
\text{sgn}(\hat{\theta}^{(k)}_i) & \text{if} \, \hat{\theta}^{(k)}_i \neq 0, i = 1,\cdots,\cc_k, \\
\in \{\hat{\ww}^{(k)}_i : |\hat{\ww}^{(k)}_i| \leq 1\} & \text{if} \, \hat{\theta}^{(k)}_i = 0, i = 1,\cdots,\cc_k.
\end{cases}
\,
\hat{\zz}^{(k)} = \begin{cases}
\hat{\theta}^{(k)} / \|\hat{\theta}^{(k)}\|_2 & \text{if} \, \hat{\theta}^{(k)} \neq 0,\\
\in \{\hat{\zz}^{(k)} : \|\hat{\zz}^{(k)}\|_2 \leq 1\} & \text{if} \, \hat{\theta}^{(k)} = 0.
\end{cases}
\eeqw
If $\hat{\theta}^{(k)} = \bf{0}$, we have $\|\hat{\zz}^{(k)}\|_2 \leq 1$. Then, from (\ref{opt1}), we obtain for such a $k \notin \Sc$
\beqw
\overset{\cc_k}{\underset{i=1}{\sum}}(\dot{\Gb}_T l(\hat{\theta})_{(k),i} + \lambda_T T^{-1} \alpha_k \hat{\ww}^{(k)}_i )^2 = \overset{\cc_k}{\underset{i=1}{\sum}} (\gamma_T T^{-1} \xi_k \hat{\zz}^{(k)}_i)^2 \leq \gamma^2_T T^{-2} \xi^2_k \|\zz^{(k)} \|^2_2.
\eeqw
Consequently, if the subgradient equations are satisfied for $\hat{\theta}^{(k)}$, then $\hat{\theta}^{(k)} = \mathbf{0}$ if
\beqw
\|\dot{\Gb}_T l(\hat{\theta})_{(k)}+\lambda_T T^{-1}\alpha_k \hat{\ww}^{(k)}\|_2 \leq \gamma_T T^{-1} \xi_k.
\eeqw
On the contrary, if this condition is not satisfied, then $\hat{\theta}^{(k)} \neq \mathbf{0}$. In this case, sparsity is fostered by the $l^1$ penalty as follows: using the optimality condition of (\ref{opt1}), we have for $\hat{\theta}^{(k)} \neq \bf{0}$
\beqw
\forall i = 1,\cdots,\cc_k, -\dot{\Gb}_T l(\hat{\theta})_{(k),i} = \lambda_T T^{-1} \alpha_k \hat{\ww}^{(k)}_i + \gamma_T T^{-1} \xi_k \cfrac{\hat{\theta}^{(k)}_i}{\|\hat{\theta}^{(k)}\|_2}.
\eeqw
If $\hat{\theta}^{(k)}_i = 0$, then $|\hat{\ww}^{(k)}_i| \leq 1$ and we obtain straightforwardly
\beqw
|\dot{\Gb}_T l(\hat{\theta})_{(k),i}| \leq \lambda_T T^{-1} \alpha_k.
\eeqw
Bertsekas (1995) proposed the use of subdifferential calculus to characterize necessary and sufficient solutions for problems such as (\ref{1stcrit}). The conditions we derived are close to those of Simon and al. (2013) (obtained for a least square loss function). They will be extensively used in the rest of the paper.

\section{Asymptotic properties} \label{Section4}

To prove the asymptotic results, we make the following assumptions.
\begin{assumption}\label{H1}
$(\eps_t)$ is a strictly stationary, ergodic and nonanticipative process.
\end{assumption}
\begin{assumption}\label{H2}
The parameter set $\Theta \subset \Rb^d$ is convex and not necessarily compact.
\end{assumption}
\begin{assumption}\label{H3}
For any $(\eps_t)$, the function $\theta \mapsto l(\eps_t;\theta)$ is strictly convex and $C^{\infty}(\Rb,\Theta)$.
\end{assumption}
\begin{assumption} \label{H4}
The function $\theta \mapsto \Eb[l(\eps_t;\theta)]$ is uniquely minimized in $\theta_0$. Moreover, $(\nabla l(\eps_t;\theta_0))$ is a square integrable martingale difference.
\end{assumption}
\begin{assumption}\label{H5}
$\Hb := \Eb[ \nabla^2_{\theta \theta'} l(\epsilon_t;\theta_0)]$ and $\Mb := \Eb[ \nabla_{\theta} l(\epsilon_t;\theta_0) \nabla_{\theta'} l(\epsilon_t;\theta_0)]$ exist and are positive definite.
\end{assumption}
\begin{assumption}\label{third_or}
Let $\upsilon_t(C) = \underset{k,l,m=1,\cdots,d}{\sup} \{ \underset{\theta:\|\theta-\theta_0\|_2 \leq \nu_T C}{\sup} |\partial^3_{\theta_k \theta_l \theta_m} l(\eps_t;\theta_0)|\}$, where $C > 0$ is a fixed constant and $\nu_T \underset{T \rightarrow \infty}{\longrightarrow} 0$, a quantity that will be made explicit. Then
\beqw
\eta(C) := \cfrac{1}{T^2} \overset{T}{\underset{t,t'=1}{\sum}} \Eb[\upsilon_t(C) \upsilon_{t'}(C)] < \infty.
\eeqw
\end{assumption}
\bre Assumptions \ref{H1} and \ref{H4} allows for using the central limit theorem of Billingsley (1961). We remind this result stated as a corollary in Billingsley (1961).
\begin{corollary}\emph{(Billingsley, 1982)} \\
If $(x_t,\Fc_t)$ is a stationary and ergodic sequence of square integrable martingal increments such that $\sigma^2_x = \text{Var}(x_t) \neq 0$, then
\beqw
\cfrac{1}{\sqrt{T}} \overset{T}{\underset{t=1}{\sum}} x_t \overset{d}{\rightarrow} \Nc(0,\sigma^2_x).
\eeqw
\end{corollary}
Note that the square martingale difference condition can be relaxed by $\alpha$-mixing and moment conditions. For instance, Rio (2013) provides a central limit theorem for strongly mixing and stationary sequences.
\ere

\begin{theorem} \label{consistency}
Under assumptions \ref{H1}-\ref{H3}, if $\lambda_T/T \rightarrow \lambda_0 \geq 0$ and $\gamma_T/T \rightarrow \gamma_0 \geq 0$, then for any compact set $\BB \subset \Theta$ such that $\theta_0 \in \BB$,
\beqw
\hat{\theta} \overset{\Pb}{\longrightarrow} \underset{\xx \in \BB}{\arg \, \min} \{\Gb_{\infty}\varphi(\xx)\},
\eeqw
with
\beqw
\Gb_{\infty}\varphi(\xx) = \Gb_{\infty}l(\xx) + \lambda_0 \overset{m}{\underset{k = 1}{\sum}} \alpha_k \|\xx^{(k)}\|_1 + \gamma_0 \overset{m}{\underset{l = 1}{\sum}} \xi_l \|\xx^{(l)}\|_2,
\eeqw
where $\theta^*_0 = \underset{\xx \in \BB}{\arg \, \min} \{\Gb_{\infty}\varphi(\xx)\}$ is supposed to be a unique minimum, and $\Gb_{\infty}l(.)$ is the limit in probability of $\Gb_T l(.)$.
\end{theorem}

\mds

To prove this theorem, we remind of Theorem II.1 of Andersen and Gill (1982) which proves that pointwise convergence in probability of random concave functions implies uniform convergence on compact subspaces.
\begin{lemma}\label{andersen_gill_lemma}\emph{(Andersen and Gill, 1982)} \\
Let $E$ be an open convex subset of $\Rb^p$, and let $F_1, F_2,\ldots,$ be a sequence of random concave functions on $E$ such that $F_n(x) \overset{\Pb}{\underset{n \rightarrow \infty}{\longrightarrow}} f(x)$ for every $x \in E$ where $f$ is some real function on $E$. Then $f$ is also concave, and for all compact $A \subset E$,
\beqw
\underset{x \in A}{\sup} |F_n(x) - f(x)| \overset{\Pb}{\underset{n \rightarrow \infty}{\longrightarrow}} 0.
\eeqw
\end{lemma}
The proof of this theorem is based on a diagonal argument and theorem 10.8 of Rockafeller (1970), that is the pointwise convergence of concave random functions on a dense and countable subset of an open set implies uniform convergence on any compact subset of the open set. Then the following corollary is stated.
\begin{corollary}\label{corollary_Andersen_Gill}\emph{(Andersen and Gill, 1982)} \\
Assume $F_n(x) \overset{\Pb}{\underset{n \rightarrow \infty}{\longrightarrow}} f(x)$,for every $x \in E$, an open convex subset of $\Rb^p$. Suppose $f$ has a unique maximum at $x_0 \in E$. Let $\hat{X}_n$ maximize $F_n$. Then $\hat{X}_n \overset{\Pb}{\underset{n \rightarrow \infty}{\longrightarrow}} x_0$.
\end{corollary}
Newey and Powell (1987) use a similar theorem to prove the consistency of asymmetric least squares estimators without any compacity assumption on $\Theta$. We apply these results in our framework, where the parameter set $\Theta$ is supposed to be convex.

\begin{proof} {\it of Theorem \ref{consistency}.} \\
By definition, $\hat{\theta} = \underset{\theta \in \Theta}{\arg \, \min} \, \{\Gb_T \varphi(\theta)\}$. In a first step, we prove the uniform convergence of $\Gb_T \varphi(.)$ to the limit quantity $\Gb_{\infty}\varphi(.)$ on any compact set $\BB \subset \Theta$, idest
\beq \label{Anderson_Gill}
\underset{\xx \in \BB}{\sup} |\Gb_T \varphi(\xx) - \Gb_{\infty}\varphi(\xx) | \overset{\Pb}{\underset{T \rightarrow \infty}{\longrightarrow}} 0.
\eeq
We define $\Cc \subset \Theta$ an open convex set and pick $\xx \in \Cc$. Then by Assumption \ref{H1}, the law of large number implies
\beqw
\Gb_T l(\xx) \overset{\Pb}{\underset{T \rightarrow \infty}{\longrightarrow}} \Gb_{\infty} l(\xx).
\eeqw
Consequently, if $\lambda_T / T \rightarrow \lambda_0 \geq 0$ and $\gamma_T / T \rightarrow \gamma_0 \geq 0$, we obtain the pointwise convergence
\beqw
|\Gb_T \varphi(\xx) - \Gb_{\infty}\varphi(\xx)| \overset{\Pb}{\underset{T \rightarrow \infty}{\longrightarrow}} 0.
\eeqw
By Lemma \ref{andersen_gill_lemma} of Andersen and Gill, $\Gb_{\infty} \varphi(.)$ is a convex function and we deduce the desired uniform convergence over any compact subset of $\Theta$, that is (\ref{Anderson_Gill}).

\mds

Now we would like that $\arg \, \min \, \{\Gb_T \varphi(.)\}\overset{\Pb}{\longrightarrow} \arg \, \min \, \{\Gb_{\infty} \varphi(.)\}$. By assumption \ref{H3}, $\varphi(.)$ is strictly convex, which implies
\beqw
|\Gb_T \varphi(\theta)| \overset{\Pb}{\underset{\|\theta\| \rightarrow \infty}{\longrightarrow}} \infty.
\eeqw
Consequently, $\arg \, \min \{\Gb_T \varphi(\xx)\} = O(1)$, such that $\hat{\theta} \in \Bc_o(\theta_0,C)$ with probability approaching one for $C$ large enough, with $\Bc_o(\theta_0,C)$ an open ball centered at $\theta_0$ and of radius $C$. Furthermore, as $\Gb_{\infty} \varphi(.)$ is strictly convex, continuous, then $\underset{\xx \in B}{\arg \, \min} \, \{\Gb_{\infty} \varphi(\xx)\}$ exists and is unique. Then by Corollary \ref{corollary_Andersen_Gill} of Andersen and Gill, we obtain
\beqw
\underset{\xx \in \BB}{\arg \, \min} \{\Gb_T \varphi(\xx)\} \overset{\Pb}{\rightarrow} \underset{\xx \in \BB}{\arg \, \min} \{\Gb_{\infty} \varphi(\xx)\},
\eeqw
that is $\hat{\theta} \overset{\Pb}{\longrightarrow} \theta^*_0$.
\end{proof}

\begin{theorem}\label{bound_prob}
Under Assumptions \ref{H1}-\ref{H3} and \ref{third_or}, the sequence of penalized estimators $\hat{\theta}$ satisfies
\beqw
\|\hat{\theta} - \theta_0\| = O_p(T^{-1/2} + \lambda_T T^{-1} a_T + \gamma_T T^{-1} b_T),
\eeqw
when $\lambda_T = o(T)$ and $\gamma_T = o(T)$, and $a_T := \text{card}(\Ac).\{\underset{k}{\max} \; \alpha_k\}$ and $b_T := \text{card}(\Ac).\{\underset{l}{\max} \; \xi_l\}$ satisfy $\lambda_T T^{-1} a_T \rightarrow 0$ and $\gamma_T T^{-1} b_T \rightarrow 0$.
\end{theorem}

\begin{proof}
We denote $\nu_T = T^{-1/2} + \lambda_T T^{-1} a_T + \gamma_T T^{-1} b_T$, with $a_T = \text{card}(\Ac).\{\underset{k}{\max} \; \alpha_k\}$ and $b_T = \text{card}(\Ac).\{\underset{l}{\max} \; \xi_l\}$. We would like to prove that for any $\eps > 0$, there exists $C_{\eps} > 0$ such that
\beq \label{bound_prob_prove}
\Pb(\cfrac{1}{\nu_T}\|\hat{\theta} - \theta_0\| > C_{\eps}) < \eps.
\eeq
We have
\beqw
\Pb(\cfrac{1}{\nu_T} \|\hat{\theta} - \theta_0\| > C_{\eps}) \leq \Pb(\exists \uu \in \Rb^d, \|\uu\|_2 \geq C_{\eps}: \Gb_T \varphi(\theta_0 + \nu_T \uu) \leq \Gb_T \varphi(\theta_0)).
\eeqw
Furthermore, $\|\uu\|_2$ can potentially be large as it represents the discrepancy $\hat{\theta}-\theta_0$ normalized by $\nu_T$. Now based on the convexity of the objective function, we have
\beq \label{subset_G}
\{\exists \uu^*, \|\uu^*\|_2 \geq C_{\eps}, \Gb_T \varphi(\theta_0 + \nu_T \uu^*) \leq \Gb_T \varphi(\theta_0)\} \subset \{\exists \bar{\uu}, \|\bar{\uu}\|_2 = C_{\eps}, \Gb_T \varphi(\theta_0 + \nu_T \bar{\uu}) \leq \Gb_T \varphi(\theta_0)\},
\eeq
a relationship that allows us to work with a fixed $\|\uu\|_2$. Let us define $\theta_1 = \theta_0 + \nu_T \uu^*$ such that $\Gb_T \varphi(\theta_1) \leq \Gb_T \varphi(\theta_0)$. Let $\alpha \in (0,1)$ and $\theta= \alpha \theta_1 + (1-\alpha) \theta_0$. Then by convexity of $\Gb_T \varphi(.)$, we obtain
\beqw
\begin{array}{llll}
\Gb_T \varphi(\theta) & \leq & \alpha \Gb_T \varphi(\theta_1) + (1-\alpha) \Gb_T \varphi(\theta_0) \\
& \leq & \Gb_T \varphi(\theta_0).
\end{array}
\eeqw
We pick $\alpha$ such that $\|\bar{\uu}\| = C_{\eps}$ with $\bar{\uu} := \alpha \theta_1 + (1-\alpha) \theta_0$. Hence (\ref{subset_G}) holds, which implies
\beqw
\begin{array}{llll}
\Pb(\|\hat{\theta} - \theta_0 \| > C_{\eps} \nu_T) & \leq & \Pb(\exists \uu \in \Rb^d, \|\uu\|_2 \geq C_{\eps}: \Gb_T \varphi(\theta_0 + \nu_T \uu) \leq \Gb_T \varphi(\theta_0)) \\
& \leq & \Pb(\exists \uu, \|\uu\|_2 = C_{\eps}: \Gb_T \varphi(\theta_0 + \nu_T \bar{\uu}) \leq \Gb_T \varphi(\theta_0)).
\end{array}
\eeqw
Hence, we pick a $\uu$ such that $\|\uu\|_2 = C_{\eps}$. Using $\pp_1(\lambda_T,\alpha,0) = 0$ and $\pp_2(\gamma_T,\xi,0) = 0$, by a Taylor expansion to $\Gb_T l(\theta_0 + \nu_T \uu)$, under assumption \ref{H4}, we obtain
\beqw
\begin{array}{llll}
\Gb_T \varphi(\theta_0 + \nu_T \uu) - \Gb_T \varphi(\theta_0)& = & \nu_T \dot{\Gb}_T l(\theta_0) \uu + \cfrac{\nu^2_T}{2} \uu' \ddot{\Gb}_T l(\bar{\theta}) \uu + \cfrac{\nu^3_T}{6} \nabla'\{\uu' \ddot{\Gb}_T l(\bar{\theta}) \uu\} \uu \\
& + & \pp_1(\lambda_T,\alpha,\theta_T)-\pp_1(\lambda_T,\alpha,\theta_0) + \pp_2(\gamma_T,\xi,\theta_T)-\pp_2(\gamma_T,\xi,\theta_0),
\end{array}
\eeqw
where $\bar{\theta}$ is defined as $\|\bar{\theta} - \theta_0\| \leq \|\theta_T - \theta_0\|$.  We want to prove
\beq \label{boundprobaproof}
\begin{array}{llll}
\Pb(\exists \uu, \|\uu\|_2 = C_{\eps}: \dot{\Gb}_T l(\theta_0) \uu + \cfrac{\nu_T}{2} \Eb[\uu' \ddot{\Gb}_T l(\theta_0) \uu] + \cfrac{\nu_T}{2} \Rc_T(\theta_0) + \cfrac{\nu^2_T}{6} \nabla'\{\uu' \ddot{\Gb}_T l(\bar{\theta}) \uu\} \uu && \\
+ \nu^{-1}_T\{\pp_1(\lambda_T,\alpha,\theta_T)-\pp_1(\lambda_T,\alpha,\theta_0) + \pp_2(\gamma_T,\xi,\theta_T)-\pp_2(\gamma_T,\xi,\theta_0)\} \leq 0) < \eps, &&
\end{array}
\eeq
where $\Rc_T(\theta_0) = \overset{d}{\underset{k,l=1}{\sum}} \uu_k \uu_l \{\partial^2_{\theta_k \theta_l} \Gb_T l(\theta_0) - \Eb[\partial^2_{\theta_k \theta_l} \Gb_T l(\theta_0)]\}$. By assumption \ref{H1}, $(\epsilon_t)$ is a non anticipative stationary solution and is ergodic. As a square integrable martingale difference,
\beqw
\sqrt{T} \dot{\Gb}_T l(\theta_0) \uu \overset{d}{\longrightarrow} \Nc(0, \uu' \Mb \uu),
\eeqw
by the central limit theorem of Billingsley (1961), which implies $\dot{\Gb}_T l(\theta_0) \uu = O_p(T^{-1/2} \uu' \Mb \uu)$. By the ergodic theorem of Billingsley (1995), we have
\beqw
\ddot{\Gb}_T l(\theta_0) \overset{\Pb}{\underset{T \rightarrow \infty}{\longrightarrow}} \Hb.
\eeqw
This implies $\Rc_T(\theta_0) = o_p(1)$.
\mds

Furthermore, we have by the Markov inequality and for $b > 0$ that
\beqw
\begin{array}{llll}
\Pb(\exists \uu, \|\uu\|_2 = C_{\eps}: \underset{\bar{\theta}:\|\theta-\theta_0\|_2 \leq \nu_T C_{\eps}}{\sup}|\cfrac{\nu^2_T}{6} \nabla'\{\uu' \ddot{\Gb}_T l(\bar{\theta}) \uu\} \uu| > b) & \leq & \cfrac{\nu^4_T C^6_{\eps}}{36 b^2} \eta(C_{\eps}),
\end{array}
\eeqw
where $\eta(C_{\eps})$ is defined in assumption \ref{third_or}. We now focus on the penalty terms. As $\pp_1(\lambda_T,\alpha,0)=0$, for the $l^1$ norm penalty, we have
\bqan\label{BoundL1}
\pp_1(\lambda_T,\alpha,\theta_T) - \pp_1(\lambda_T,\alpha,\theta_0) & = & \lambda_T T^{-1} \underset{k \in \Sc}{\sum} \alpha_k \{\|\theta^{(k)}_0 + \nu_T \uu^{(k)}\|_1  - \|\theta^{(k)}_0\|_1 \}, \nonumber  \\
\text{and} \; |\pp_1(\lambda_T,\alpha,\theta_T) - \pp_1(\lambda_T,\alpha,\theta_0)| & \leq & \text{card}(\Sc) \{ \underset{k \in \Sc}{\max} \; \alpha_k \} \lambda_T T^{-1} \nu_T  \|\uu\|_1.
\eqan
As for the $l^1/l^2$ norm, we obtain
\bqan\label{BoundL2}
\pp_2(\gamma_T,\xi,\theta_T) - \pp_2(\gamma_T,\xi,\theta_0) & = & \gamma_T T^{-1} \underset{l \in \Sc}{\sum} \xi_l \{ \|\theta^{(l)}_T\|_2 - \|\theta^{(l)}_0\|_2\}, \nonumber \\
\text{and} \; |\pp_2(\gamma_T,\xi,\theta_T) - \pp_2(\gamma_T,\xi,\theta_0) |& \leq & \gamma_T T^{-1} \underset{l \in \Sc}{\sum} \xi_l \nu_T \|\uu^{(l)}\|_2 \nonumber \\
& \leq & \text{card}(\Sc) \{ \underset{l\in \Sc}{\max} \; \xi_l \} \gamma_T T^{-1} \nu_T \|\uu\|_2.
\eqan

\mds

Then denoting by $\delta_T = \lambda_{\min}(\Hb) C^2_{\eps} \nu_T$, and using $\cfrac{\nu_T}{2} \Eb[\uu' \ddot{\Gb}_T l(\theta_0) \uu] \geq \delta_T$, we deduce that (\ref{boundprobaproof}) can be bounded as
\beqw
\begin{array}{llll}
\Pb(\exists \uu, \|\uu\|_2 = C_{\eps}: \dot{\Gb}_T l(\theta_0) \uu + \cfrac{\nu_T}{2} \uu' \ddot{\Gb}_T l(\theta_0) \uu + \cfrac{\nu^2_T}{6} \nabla'\{\uu' \ddot{\Gb}_T l(\bar{\theta}) \uu\} \uu && \\
+ \nu^{-1}_T \{\pp_1(\lambda_T,\alpha,\theta_T)-\pp_1(\lambda_T,\alpha,\theta_0) + \pp_2(\gamma_T,\xi,\theta_T)-\pp_2(\gamma_T,\xi,\theta_0) \} \leq 0)  && \\
\leq \Pb(\exists \uu, \|\uu\|_2 = C_{\eps}: |\dot{\Gb}_T l(\theta_0) \uu| > \delta_T/8) + \Pb(\exists \uu, \|\uu\|_2 = C_{\eps}: \cfrac{\nu_T}{2} |\Rc_T(\theta_0)| > \delta_T/8) && \\
+ \Pb(\exists \uu, \|\uu\|_2 = C_{\eps}:| \cfrac{\nu^2_T}{6} \nabla'\{\uu' \ddot{\Gb}_T l(\bar{\theta}) \uu\} \uu |> \delta_T/8) && \\
+ \Pb(\exists \uu, \|\uu\|_2 = C_{\eps}: |\pp_1(\lambda_T,\alpha,\theta_T)-\pp_1(\lambda_T,\alpha,\theta_0)| > \nu_T \delta_T/8) && \\
+ \Pb(\exists \uu, \|\uu\|_2 = C_{\eps}: |\pp_2(\gamma_T,\xi,\theta_T)-\pp_2(\gamma_T,\xi,\theta_0)| > \nu_T \delta_T/8). &&
\end{array}
\eeqw
We also have for $C_{\eps}$ and $T$ large enough, and using norm equivalences that
\beqw
\begin{array}{llll}
\Pb(\exists \uu, \|\uu\|_2 = C_{\eps}: |\pp_1(\lambda_T,\alpha,\theta_T)-\pp_1(\lambda_T,\alpha,\theta_0)| > \nu_T \delta_T/8) && \\
\leq \Pb(\exists \uu, \|\uu\|_2 = C_{\eps}: \text{card}(\Sc) \{ \underset{k \in \Sc}{\max} \; \alpha_k \} \lambda_T T^{-1} \nu_T  \|\uu\|_1  > \nu_T \delta_T/8) < \eps/5, &&\\
\Pb(\exists \uu, \|\uu\|_2 = C_{\eps}: |\pp_2(\gamma_T,\xi,\theta_T)-\pp_2(\gamma_T,\xi,\theta_0)| > \nu_T \delta_T/8)  && \\
\leq \Pb(\exists \uu, \|\uu\|_2 = C_{\eps}: \text{card}(\Sc) \{ \underset{l\in \Sc}{\max} \; \xi_l \} \gamma_T T^{-1} \nu_T \|\uu\|_2 > \nu_T \delta_T/8) < \eps/5.&&
\end{array}
\eeqw

\mds

Moreover, if $\nu_T = T^{-1/2} + \lambda_T T^{-1} a_T + \gamma_T T^{-1} b_T$, then for $C_{\eps}$ large enough
\beqw
\begin{array}{llll}
\Pb(\exists \uu, \|\uu\|_2 = C_{\eps}: |\dot{\Gb}_T l(\theta_0) \uu| > \delta_T/8) & \leq & \cfrac{C^2_{\eps} C_{st}}{T \delta^2_T} \\
& \leq &  \cfrac{C_{st}}{C^4_{\eps}} < \eps/5.
\end{array}
\eeqw
Moreover
\beqw
\begin{array}{llll}
\Pb(\exists \uu, \|\uu\|_2 = C_{\eps}: \underset{\bar{\theta}:\|\bar{\theta}-\theta_0\|_2 < \nu_T C_{\eps}}{\sup}|\cfrac{\nu^2_T}{6} \nabla'\{\uu' \ddot{\Gb}_T l(\bar{\theta}) \uu\} \uu |> \delta_T/8) & \leq & \cfrac{C_{st} \nu^4_T \eta(C_{\eps})}{\delta^2_T} \\
& \leq & C_{st} \nu^2_T C^2_{\eps} \eta(C_{\eps})
\end{array}
\eeqw
where $C_{st} > 0$ is a generic constant. Consequently, we obtain, for $T$ and $C_{\eps}$ large enough, we obtain
\beqw
\begin{array}{llll}
\Pb(\exists \uu, \|\uu\|_2 = C_{\eps}: |\dot{\Gb}_T l(\theta_0) \uu| > \delta_T/8) + \Pb(\exists \uu, \|\uu\|_2 = C_{\eps}: \cfrac{\nu_T}{2} |\Rc_T(\theta_0)| > \delta_T/8) && \\
+ \Pb(\exists \uu, \|\uu\|_2 = C_{\eps}:| \cfrac{\nu^2_T}{6} \nabla'\{\uu' \ddot{\Gb}_T l(\bar{\theta}) \uu\} \uu |> \delta_T/8) &&\\
+ \Pb(\exists \uu, \|\uu\|_2 = C_{\eps}: |\pp_1(\lambda_T,\alpha,\theta_0)-\pp_1(\lambda_T,\alpha,\theta_T)| > \nu_T \delta_T/8) && \\
+ \Pb(\exists \uu, \|\uu\|_2 = C_{\eps}: |\pp_2(\gamma_T,\xi,\theta_0)-\pp_2(\gamma_T,\xi,\theta_T)| > \nu_T \delta_T/8) + 0 && \\
\leq \cfrac{C_{st}}{C^4_{\eps}} + \nu^2_T C^2_{\eps} \eta(C_{\eps}) C_{st} + 3 \eps/5 && \\
\leq \eps,
\end{array}
\eeqw
for $C_{\eps}$ sufficiently large, and $T$ large enough. We then deduce
\beqw
\|\hat{\theta} - \theta_0\| = O_p(\nu_T) = O_p(\lambda_T T^{-1} a_T + \gamma_T T^{-1} b_T + T^{-1/2}).
\eeqw
\end{proof}

\bre
We would like to highlight the use of the convexity property of $\Gb_T \varphi(.)$. It allowed us to obtain the upper bound (\ref{boundprobaproof}). Otherwise, the inequality would have been uniform over $\|\uu\|_2 \geq C_{\eps}$. A consequence is that $\|\uu\|_2$ can take significantly large values, which would have made the control of the random part in the Taylor expansion hard. This issue is overcome thanks to the convexity that allows for working with fixed $\|\uu\|_2$, as Fan and Li (2001), Fan and Peng (2004) or Nardi and Rinaldo (2008) do.
\ere

\mds

We now focus on the distribution of the SGL estimator. Deriving the asymptotic distribution for M-estimators is standard in the case the objective function is differentiable. It consists of characterizing the estimator by the orthogonality conditions and derive a linear representation by Taylor expansions of the estimator. But these techniques do not apply when the objective function is not differentiable. In our case, $\varphi(.)$ is not differentiable at $0$ due to the penalty terms. In some specific context, it may be possible to treat the non-differentiability of $\Gb_T \varphi(.)$ by applying the expectation operator $\Eb[.]$ to $\varphi(.)$, which then becomes differentiable in $\theta_0$. Then Taylor expansions are feasible and one obtains the distribution, provided some regularity conditions of the empirical criterion, such as stochastic equi-continuity: see Andrews (1994, a,b). This approach works for specific loss functions, such as the LAD. But in our setting, the expectation operator fails at regularizing $\varphi(.)$ due to the penalty functionals.

\mds

Another approach to obtain the asymptotic distribution relies on the convexity property of $\varphi(.)$, and hence of $\Gb_T \varphi(.)$, without assuming strong regularity conditions on $\varphi(.)$. The intuition behind this rather strong statement is as follows. Let $\Fb_T(\uu)$ and $\Fb_{\infty}(\uu)$, $\uu  \in \Rb^d$, be random convex functions such that their minimum are respectively $\uu_T$ and $\uu_{\infty}$. Then if $\Fb_T(.)$ converges in finite distribution to $\Fb_{\infty}(.)$, and $\uu_{\infty}$ is the unique minimum of $\Fb_{\infty}$ with probability one, then $\uu_T$ converges weakly to $\uu_{\infty}$. This method to prove the convergence of $\arg \, \min$ processes is called the \emph{convexity argument}. It was developed by Davis and al. (1992), Hjort, Pollard (1993), Geyer (1996a, 1996b) or Kato (2009). Chernozhukov and Huong (2004), Chernozhukhov (2005) use this convexity argument to obtain the asymptotic distribution of quantile regression type estimators. The convexity argument only requires the lower-semicontinuity and convexity of the empirical criterion. The convexity Lemma, as in Chernozhukov (2005), can be stated as follows.
\begin{lemma} \label{convex_lemma} ( Chernozhukov, 2005) \\
Suppose \\
(i) a sequence of convex lower-semicontinuous $\Fb_T: \Rb^d \rightarrow \bar{\Rb}$ marginally converges to $\Fb_{\infty}: \Rb^d \rightarrow \bar{\Rb}$ over a dense subset of $\Rb^d$; \\
(ii) $\Fb_{\infty}$ is finite over a nonempty open set $E \subset \Rb^d$; \\
(iii) $\Fb_{\infty}$ is uniquely minimized at a random vector $\uu_{\infty}$. \\
Then
\beqw
\underset{\zz \in \Rb^d}{\arg \, \min} \, \Fb_T(\zz) \overset{d}{\longrightarrow} \underset{\zz \in \Rb^d}{\arg \, \min} \, \Fb_{\infty}(\zz), \, \text{that is} \; \uu_T \overset{d}{\longrightarrow} \uu_{\infty}.
\eeqw
\end{lemma}

\begin{theorem}\label{asym_dist_Knight}
Under Assumptions \ref{H1}-\ref{third_or}, if $\lambda_T T^{-1/2} \rightarrow \lambda_0$ and $\gamma_T T^{-1/2} \rightarrow \gamma_0$, then
\beqw
\sqrt{T}(\hat{\theta} - \theta_0) \overset{d}{\longrightarrow} \underset{\uu \in \Rb^d}{\arg \; \min} \, \{ \Fb_{\infty}(\uu) \},
\eeqw
provided $\Fb_{\infty}$ is the random function in $\Rb^d$, where
\beqw
\begin{array}{llll}
\Fb_{\infty}(\uu) & = & \cfrac{1}{2}\uu' \Hb \uu + \uu' \ZZ + \lambda_0 \overset{m}{\underset{k = 1}{\sum}} \alpha_k \overset{\cc_k}{\underset{i = 1}{\sum}} \{ |\uu^{(k)}_i| \mathbf{1}_{\theta^{(k)}_{0,i} = 0} + \uu^{(k)}_i \text{sgn}(\theta^{(k)}_{0,i})\mathbf{1}_{\theta^{(k)}_{0,i} \neq 0} \}  \\
& + & \gamma_0 \overset{m}{\underset{l = 1}{\sum}} \xi_l \{\|\uu^{(l)}\|_2 \mathbf{1}_{\theta^{(l)}_{0} = \mathbf{0}} + \cfrac{\uu^{(l)'} \theta^{(l)}_0}{\|\theta^{(l)}_0\|_2} \mathbf{1}_{\theta^{(l)}_{0} \neq \mathbf{0}}\},
\end{array}
\eeqw
with $\Hb = \Hb(\theta_0) := \Eb[\nabla^2_{\theta \theta'} l(\eps_t;\theta_0)]$ and some random vector $\ZZ \sim \Nc(0,\Mb)$, $\Mb = \Mb(\theta_0):= \Eb[\nabla_{\theta} l(\eps_t;\theta_0) \nabla_{\theta'} l(\eps_t;\theta_0)]$.
\end{theorem}

\begin{proof}
Let $\uu \in \Rb^d$ such that $\theta = \theta_0 + \uu/T^{1/2}$ and we define the empirical criterion $\Fb_T(\uu) = T \Gb_T (\varphi(\theta_0 + \uu/T^{1/2}) - \varphi(\theta_0))$. First, we are going to prove the finite distributional convergence of $\Fb_T$ to $\Fb_{\infty}$. Then we use the convexity of $\Fb_T(.)$ to obtain the convergence in distribution of the $\arg \, \min$ empirical criterion to the $\arg \, \min$ process limit. To do so, let $\uu = \sqrt{T}(\theta - \theta_0)$. We have
\beqw
\begin{array}{llll}
\Fb_T(\uu) & = & T \{\Gb_T (l(\theta) - l(\theta_0)) + \pp_1(\lambda_T,\alpha,\theta) - \pp_1(\lambda_T,\alpha,\theta_0) + \pp_2(\gamma_T,\xi,\theta)-\pp_2(\gamma_T,\xi,\theta_0) \} \\
& = & T \Gb_T (l(\theta_0 + \uu/T^{1/2}) - l(\theta_0)) + \lambda_T \overset{m}{\underset{k = 1}{\sum}} \alpha_k [\|\theta^{(k)}_0 + \uu^{(k)}/\sqrt{T}\|_1 - \|\theta^{(k)}_0\|_1] \\
& + & \gamma_T \overset{m}{\underset{l = 1}{\sum}} \xi_l [\|\theta^{(l)}_0 + \uu^{(l)}/\sqrt{T}\|_2 - \|\theta^{(l)}_0\|_2],
\end{array}
\eeqw
where $\Fb_T(.)$ is convex and $C^0(\Rb^d)$. We now prove the finite dimensional distribution of $\Fb_T$ to $\Fb_{\infty}$ to apply Lemma \ref{convex_lemma}. For the $l^1$ penalty, for any group $k$, we have for $T$ sufficiently large
\beqw
\|\theta^{(k)}_0 + \uu^{(k)}/\sqrt{T}\|_1 - \|\theta^{(k)}_0\|_1 = T^{-1/2} \overset{\cc_k}{\underset{i = 1}{\sum}} \{ |\uu^{(k)}_i| \mathbf{1}_{\theta^{(k)}_{0,i} = 0} + \uu^{(k)}_i \text{sgn}(\theta^{(k)}_{0,i})\mathbf{1}_{\theta^{(k)}_{0,i} \neq 0} \},
\eeqw
which implies that
\beqw
\lambda_T \overset{m}{\underset{k = 1}{\sum}} \alpha_k [\|\theta^{(k)}_0 + \uu^{(k)}/\sqrt{T}\|_1 - \|\theta^{(k)}_0\|_1] \underset{T \rightarrow \infty}{\longrightarrow} \lambda_0 \overset{m}{\underset{k = 1}{\sum}} \alpha_k \overset{\cc_k}{\underset{i = 1}{\sum}} \{ |\uu^{(k)}_i| \mathbf{1}_{\theta^{(k)}_{0,i} = 0} + \uu^{(k)}_i \text{sgn}(\theta^{(k)}_{0,i})\mathbf{1}_{\theta^{(k)}_{0,i} \neq 0} \},
\eeqw
under the condition that $\lambda_T / \sqrt{T} \rightarrow \lambda_0$.

\mds

As for the $l^1/l^2$ quantity, for any group $l$, we have
\beqw
\|\theta^{(l)}_0 + u^{(l)}/\sqrt{T}\|_2 - \|\theta^{(l)}_0\|_2 = T^{-1/2} \{\|u^{(l)}\|_2 \mathbf{1}_{\theta^{(l)}_{0} = \mathbf{0}} + \cfrac{u^{(l)'} \theta^{(l)}_0}{\|\theta^{(l)}_0\|_2} \mathbf{1}_{\theta^{(l)}_{0} \neq \mathbf{0}}\} + o(T^{-1}).
\eeqw
Consequently, if $\gamma_T T^{-1/2} \rightarrow \gamma_0 \geq 0$, we obtain
\beqw
\gamma_T \overset{m}{\underset{l = 1}{\sum}} \xi_l [\|\theta^{(l)}_0 + u^{(l)}/\sqrt{T}\|_2 - \|\theta^{(l)}_0\|_2 ] = \gamma_0 \overset{m}{\underset{l = 1}{\sum}} \xi_l \{\|u^{(l)}\|_2 \mathbf{1}_{\theta^{(l)}_{0,k} = \mathbf{0}} + \cfrac{u^{(l)'} \theta^{(l)}_0}{\|\theta^{(l)}_0\|_2} \mathbf{1}_{\theta^{(l)}_0 \neq \mathbf{0}}\} + o(T^{-1}) \gamma_T.
\eeqw

Now for the unpenalized criterion $\Gb_T l(.)$, by a Taylor expansion, we have
\beqw
T \Gb_T (l(\theta_0 + \uu/T^{1/2}) - l(\theta_0)) = \uu' T^{1/2}\dot{\Gb}_T l(\theta_0) + \cfrac{1}{2} \uu' \ddot{\Gb}_T l(\theta_0) \uu + \frac{1}{6 T^{1/3}}\nabla'\{\uu' \ddot{\Gb}_T l(\bar{\theta}) \uu\} \uu,
\eeqw
where $\bar{\theta}$ is defined as $\|\bar{\theta} - \theta_0\| \leq \|\uu\|/\sqrt{T}$. Then by Assumption \ref{H4}, we have the central limit theorem of Billingsley (1961)
\beqw
\sqrt{T} \dot{\Gb}_T l(\theta_0) \overset{d}{\longrightarrow} \Nc(0,\Mb),
\eeqw
and by the ergodic theorem
\beqw
\ddot{\Gb}_T l(\theta_0) \overset{\Pb}{\underset{T \rightarrow \infty}{\longrightarrow}} \Hb.
\eeqw
Furthermore, we have by assumption \ref{third_or}
\beqw
\begin{array}{llll}
|\nabla'\{\uu' \ddot{\Gb}_T l(\bar{\theta}) \uu\} \uu |^2 & \leq & \cfrac{1}{T^2} \overset{T}{\underset{t,t'=1}{\sum}} \overset{d}{\underset{k_1,l_1,m_1}{\sum}}\overset{d}{\underset{k_2,l_2,m_2}{\sum}} \uu_{k_1} \uu_{l_1} \uu_{m_1} \uu_{k_2} \uu_{l_2} \uu_{m_2}  |\partial^3_{\theta_{k_1} \theta_{l_1} \theta_{m_1}} l(\eps_t;\bar{\theta}) . \partial^3_{\theta_{k_2} \theta_{l_2} \theta_{m_2}} l(\eps_{t'};\bar{\theta}) |\\
& \leq & \cfrac{1}{T^2} \overset{T}{\underset{t,t'=1}{\sum}} \overset{d}{\underset{k_1,l_1,m_1}{\sum}}\overset{d}{\underset{k_2,l_2,m_2}{\sum}} \uu_{k_1} \uu_{l_1} \uu_{m_1} \uu_{k_2} \uu_{l_2} \uu_{m_2} \upsilon_t(C) \upsilon_{t'}(C),
\end{array}
\eeqw
for $C$ large enough, such that the $\upsilon_t(C) = \underset{k,l,m=1,\cdots,d}{\sup} \{ \underset{\theta:\|\theta-\theta_0\|_2 \leq \nu_T C}{\sup} |\partial^3_{\theta_k \theta_l \theta_m} l(\eps_t;\theta)|\}$ with $\nu_T = T^{-1/2} + \lambda_T T^{-1} a_T + \gamma_T T^{-1} b_T$. We deduce
\beqw
\nabla'\{\uu' \ddot{\Gb}_T l(\bar{\theta}) \uu\} \uu  = O_p(\|\uu\|^3_2 \eta(C)).
\eeqw
Consequently, we obtain
\beqw
\frac{1}{6 T^{1/3}}\nabla'\{\uu' \ddot{\Gb}_T l(\bar{\theta}) \uu\} \uu \overset{\Pb}{\underset{T \rightarrow \infty}{\longrightarrow}} 0.
\eeqw
Then we proved that $\Fb_T(\uu) \overset{d}{\longrightarrow} \Fb_{\infty}(\uu)$, for a fixed $\uu$. Let us observe that
\beqw
\uu^*_T = \underset{\uu}{\arg \; \min} \, \{ \Fb_T(\uu) \},
\eeqw
and $\Fb_T(.)$ admits as a minimizer $\uu^*_T = \sqrt{T}(\hat{\theta} - \theta_0)$. As $\Fb_T$ is convex and $\Fb_{\infty}$ is continuous, convex and has a unique minimum, then by the convexity Lemma \ref{convex_lemma}, we obtain
\beqw
\sqrt{T}(\hat{\theta} - \theta_0) = \underset{\uu}{\arg \; \min} \{ \Fb_T \}  \overset{d}{\longrightarrow} \underset{\uu}{\arg \; \min} \{ \Fb_{\infty} \}.
\eeqw

\end{proof}

\begin{theorem}\label{asym_dist_bias}
Under assumptions \ref{H1}-\ref{third_or}, if $\gamma_T T^{-1} \rightarrow 0$ and $\gamma_T T^{-1/2} \rightarrow \infty$ such that $\lambda_T \gamma^{-1}_T \rightarrow \mu_0$, with $\mu_0 \geq 0$, then
\beqw
\cfrac{T}{\gamma_T} (\hat{\theta} - \theta_0) \overset{d}{\longrightarrow} \underset{\uu}{\arg \; \min} \, \{ \Kb_{\infty}(\uu) \},
\eeqw
provided $\Kb_{\infty}$ is a uniquely defined random function in $\Rb^d$, where
\beqw
\begin{array}{llll}
\Kb_{\infty}(\uu) & = & \cfrac{1}{2}\uu' \Hb \uu + \mu_0 \overset{m}{\underset{k = 1}{\sum}} \alpha_k \{ \|\uu^{(k)} \|_1\mathbf{1}_{\theta^{(k)}_0 = 0} + \uu^{(k) '} \text{sgn}(\theta^{(k)}_0)\mathbf{1}_{\theta^{(k)}_0 \neq 0} \}  \\
& + &  \overset{m}{\underset{l = 1}{\sum}} \xi_l \{\|\uu^{(l)}\|_2 \mathbf{1}_{\theta^{(l)}_{0} = \mathbf{0}} + \cfrac{\uu^{(l)'} \theta^{(l)}_0}{\|\theta^{(l)}_0\|_2} \mathbf{1}_{\theta^{(l)}_{0} \neq \mathbf{0}}\}.
\end{array}
\eeqw
The limit quantity $\Kb_{\infty}(.)$ is non-random, which implies that the convergence in distribution implies the convergence in probability $\cfrac{T}{\gamma_T} (\hat{\theta} - \theta_0) \overset{\Pb}{\longrightarrow} \underset{\uu}{\arg \; \min} \, \{ \Kb_{\infty}(\uu) \}$ by Shiryaev (ex 7, p 259, 1995).
\end{theorem}

\begin{proof}
To prove this convergence result, we proceed as in Theorem \ref{asym_dist_Knight}. To do so, we define $\theta = \theta_0 + \uu \gamma_T /T$ and we prove that $\tilde{\Fb}_T (\uu) = \Gb_T (\varphi(\theta_0 + \uu T/\gamma_T)-\varphi(\theta_0))$ converges in finite distribution to $\Kb_{\infty}(.)$. We have
\beqw
\begin{array}{llll}
\tilde{\Fb}_T (\uu) & = & T \{\Gb_T (l(\theta) - l(\theta_0)) + \pp_1(\lambda_T,\alpha,\theta) - \pp_1(\lambda_T,\alpha,\theta_0) + \pp_2(\gamma_T,\xi,\theta)-\pp_2(\gamma_T,\xi,\theta_0) \} \\
& = & T \Gb_T (l(\theta_0 + \uu \gamma_T /T) - l(\theta_0)) + \lambda_T \overset{m}{\underset{k = 1}{\sum}} \alpha_k [\|\theta^{(k)}_0 + \uu^{(k)} \gamma_T /T\|_1 - \|\theta^{(k)}_0\|_1] \\
& + & \gamma_T \overset{m}{\underset{l = 1}{\sum}} \xi_l [\|\theta^{(l)}_0 + \uu^{(l)} \gamma_T /T\|_2 - \|\theta^{(l)}_0\|_2].
\end{array}
\eeqw
For the unpenalized empirical criterion, we have the expansion
\beqw
T \Gb_T (l(\theta_0 + \uu \gamma_T /T) - l(\theta_0)) = \gamma_T \dot{\Gb}_T l(\theta_0) \uu + \cfrac{\gamma^2_T}{2T} \uu' \ddot{\Gb}_T l(\theta_0) \uu + \cfrac{\gamma^3_T}{6T^2} \nabla \{\uu' \ddot{\Gb}_T l(\bar{\theta}) \uu\} \uu,
\eeqw
where $\bar{\theta}$ lies between $\theta_0$ and $\theta_0 + \uu \gamma_T /T$. This implies $\tilde{\Fb}_T(\uu) = \frac{\gamma^2_T}{T} \Kb_T(\uu)$, where
\beqw
\begin{array}{llll}
\Kb_T(\uu) & = & \frac{\sqrt{T}}{\gamma_T}( \sqrt{T} \dot{\Gb}_T l(\theta_0) \uu) + \frac{1}{2}\uu' \ddot{\Gb}_T l(\bar{\theta}) \uu + \frac{\gamma_T}{6 T} \nabla'\{\uu' \ddot{\Gb}_T l(\bar{\theta}) \uu\} \uu\\
& + & \frac{T}{\gamma^2_T} \lambda_T \overset{m}{\underset{k = 1}{\sum}} \alpha_k [\|\theta^{(k)}_0 + \uu^{(k)} \gamma_T /T \|_1 - \|\theta^{(k)}_0\|_1 ] + \frac{T}{\gamma_T} \overset{m}{\underset{l = 1}{\sum}} \xi_l [ \|\theta^{(l)}_0 + u^{(l)} \gamma_T /T\|_2 - \|\theta^{(l)}_0\|_2 ].
\end{array}
\eeqw
We first focus on the penalty terms. For the $l^1$ part, for any group $k$, we have
\beqw
\|\theta^{(k)}_0 + \uu^{(k)} \gamma_T /T \|_1 - \|\theta^{(k)}_0\|_1 = \gamma_T T^{-1} \overset{\cc_k}{\underset{i = 1}{\sum}} \{ |\uu^{(k)}_i| \mathbf{1}_{\theta^{(k)}_{0,i} = 0} + \uu^{(k)}_i \text{sgn}(\theta^{(k)}_{0,i})\mathbf{1}_{\theta^{(k)}_{0,i} \neq 0} \}.
\eeqw
We deduce that
\beqw
\frac{T}{\gamma^2_T} \lambda_T \alpha_k [\|\theta^{(k)}_0 + \uu^{(k)} \gamma_T /T \|_1 - \|\theta^{(k)}_0\|_1 ] \rightarrow \mu_0 \overset{\cc_k}{\underset{i = 1}{\sum}} \alpha_k \{ |\uu^{(k)}_i| \mathbf{1}_{\theta^{(k)}_{0,i} = 0} + \uu^{(k)}_i \text{sgn}(\theta^{(k)}_{0,i})\mathbf{1}_{\theta^{(k)}_{0,i} \neq 0} \},
\eeqw
under the condition $\lambda_T \gamma^{-1}_T \rightarrow \mu_0$.

\mds

As for the $l^1/l^2$ quantity, for any group $l$, we have
\beqw
\|\theta^{(l)}_0 + u^{(l)} \gamma_T /T\|_2 - \|\theta^{(l)}_0\|_2 =  \gamma_T T^{-1} \{\|u^{(l)}\|_2 \mathbf{1}_{\theta^{(l)}_{0} = \mathbf{0}} + \cfrac{u^{(l)'} \theta^{(l)}_0}{\|\theta^{(l)}_0\|_2} \mathbf{1}_{\theta^{(l)}_{0} \neq \mathbf{0}}\} + o(T^{-1}).
\eeqw
Consequently, we obtain
\beqw
\frac{T}{\gamma_T}\xi_l[\|\theta^{(l)}_0 + u^{(l)} \gamma_T /T\|_2 - \|\theta^{(l)}_0\|_2] \rightarrow \xi_l \{\|u^{(l)}\|_2 \mathbf{1}_{\theta^{(l)}_{0} = \mathbf{0}} + \cfrac{u^{(l)'} \theta^{(l)}_0}{\|\theta^{(l)}_0\|_2} \mathbf{1}_{\theta^{(l)}_{0} \neq \mathbf{0}}\}.
\eeqw

\mds

Now for the unpenalized part, by the central limit theorem of Billingsley (1961), $ \sqrt{T} \dot{\Gb}_T l(\theta_0)$ is asymptotically normal, then $\gamma_T T^{-1/2} \rightarrow \infty$ implies by the Slutsky theorem
\beqw
\frac{\sqrt{T}}{\gamma_T}( \sqrt{T} \dot{\Gb}_T l(\theta_0) \uu ) \overset{\Pb}{\underset{T \rightarrow \infty}{\longrightarrow}} 0.
\eeqw
Furthermore, by the ergodic theorem of Billingsley (1961), we have
\beqw
\ddot{\Gb}_T l(\theta_0) \overset{\Pb}{\underset{T \rightarrow \infty}{\longrightarrow}} \Hb.
\eeqw
As for the third order term, by assumption \ref{third_or} and using the same reasoning as the proof of Theorem \ref{asym_dist_Knight}, we have
\beqw
\frac{\gamma_T}{6 T} \nabla'\{\uu' \ddot{\Gb}_T l(\bar{\theta}) \uu\} \uu \overset{\Pb}{\underset{T \rightarrow \infty}{\longrightarrow}} 0,
\eeqw
using $\gamma_T = o(T)$. Then we proved that $\Kb_T(\uu) \overset{d}{\longrightarrow} \Kb_{\infty}(\uu)$, for a fixed $\uu \in \Rb^d$. We have
\beqw
\uu^*_T = \underset{\uu}{\arg \, \min} \, \{\Kb_T (\uu)\},
\eeqw
and $\Kb_T(.)$ admits as a minimizer $\uu^*_T = \frac{T}{\gamma_T}(\hat{\theta}-\theta_0)$. $\Kb_T(.)$ is convex and $\Kb_{\infty}(.)$ is continuous, then by the convexity Lemma, we deduce
\beqw
\cfrac{T}{\gamma_T} (\hat{\theta}-\theta_0) = \arg \, \min \, \{\Kb_T\} \overset{d}{\longrightarrow} \arg \, \min \, \{\Kb_{\infty}\}.
\eeqw

\end{proof}

\bre
The convergence rate of $\hat{\theta}$ is slower than $\sqrt{T}$. Furthermore, the limiting function is not random.
\ere

We now turn to the oracle property of the SGL. Model selection consistency consists of evaluating the probability that $\{\hat{\Ac} = \Ac\}$, for $T$ large enough. That means we check that the regularization asymptotically allows for identifying the right model.

\begin{proposition}\label{proposition1}
Under assumption \ref{H1}-\ref{third_or}, if $\lambda_T T^{-1/2} \rightarrow \lambda_0$ and $\gamma_T T^{-1/2} \rightarrow \gamma_0$, then
\beqw
\underset{T \rightarrow \infty}{\lim \, \sup} \; \Pb(\hat{\Ac} = \Ac) \leq c < 1,
\eeqw
where $c$ is a constant depending on the true model.
\end{proposition}

\begin{proof}
In Theorem \ref{asym_dist_Knight}, we proved
\beqw
\sqrt{T}(\hat{\theta} - \theta_0) := \underset{\uu \in \Rb^d}{\arg \, \min} \{\Fb_T\} \overset{d}{\longrightarrow} \underset{\uu \in \Rb^d}{\arg \, \min} \{\Fb_{\infty}\},
\eeqw
under the assumption $\lambda_T/\sqrt{T} \rightarrow \lambda_0$ and $\gamma_T / \sqrt{T} \rightarrow \gamma_0$. The limit random function is
\beqw
\begin{array}{llll}
\Fb_{\infty}(\uu) & = & \cfrac{1}{2}\uu' \Hb \uu + \uu' \ZZ + \lambda_0 \overset{m}{\underset{k = 1}{\sum}} \alpha_k \overset{\cc_k}{\underset{i = 1}{\sum}} \{ |\uu^{(k)}_i| \mathbf{1}_{\theta^{(k)}_{0,i} = 0} + \uu^{(k)}_i \text{sgn}(\theta^{(k)}_{0,i})\mathbf{1}_{\theta^{(k)}_{0,i} \neq 0} \}  \\
& + & \gamma_0 \overset{m}{\underset{l = 1}{\sum}} \xi_l \{\|\uu^{(l)}\|_2 \mathbf{1}_{\theta^{(l)}_{0} = \mathbf{0}} + \cfrac{\uu^{(l)'} \theta^{(l)}_0}{\|\theta^{(l)}_0\|_2} \mathbf{1}_{\theta^{(l)}_{0} \neq \mathbf{0}}\}.
\end{array}
\eeqw
First, let us observe that
\beqw
\{\hat{\Ac} = \Ac \} = \{\forall k =1,\cdots,m, \, i \in \Ac^c_k, \hat{\theta}^{(k)}_i = 0\} \cap \{\forall k =1,\cdots,m, \, i \in \hat{\Ac}^c_k, \theta^{(k)}_{0,i} = 0\}.
\eeqw
Both sets describing $\{\hat{\Ac} = \Ac \}$ are symmetric, and thus we can focus on
\beqw
\{\hat{\Ac} = \Ac \} \Rightarrow \{\forall k =1,\cdots,m, \, i \in \Ac^c_k, T^{1/2} \hat{\theta}^{(k)}_i = 0\}.
\eeqw
Hence
\beqw
\Pb( \hat{\Ac} = \Ac ) \leq \Pb( \forall k = 1,\cdots,m, \forall i \in \Ac^c_k, T^{1/2}\hat{\theta}^{(k)}_i = 0).
\eeqw
Denoting by $ \uu^*:= \underset{\uu \in \Rb^d}{\arg \, \min} \{\Fb_{\infty}(\uu)\}$, Theorem \ref{asym_dist_Knight} corresponds to $\sqrt{T}(\hat{\theta}_{\Ac} - \theta_{0,\Ac}) \overset{d}{\longrightarrow} \uu^*_{\Ac}$. By the Portmanteau Theorem (see Wellner and van der Vaart, 1996), we have
\beqw
\underset{T \rightarrow \infty}{\lim \sup} \; \Pb( \forall k = 1,\cdots,m, \forall i \in \Ac^c_k, T^{1/2}\hat{\theta}^{(k)}_i = 0) \leq \Pb(\forall k =1,\cdots,m, \forall i \in \Ac^c_k, \uu^{(k) *}_i = 0),
\eeqw
as $\theta_{0,\Ac^c} = \mathbf{0}$. Consequently, we need to prove that the probability of the right hand side is strictly inferior to $1$, which is upper-bounded by
\beq \label{upper_bound_select}
\begin{array}{llll}
\Pb(\forall k =1,\cdots,m, \forall i \in \Ac^c_k, \uu^{(k) *}_i = 0) \leq && \\
\min(\Pb(k \notin \Sc, \uu^{(k) *} = 0),\Pb(k \in \Sc, \forall i \in \Ac^c_k, \uu^{(k) *}_i = 0)).
\end{array}
\eeq

\mds

If $\lambda_0 = \gamma_0 = 0$, then $\uu^* = -\Hb^{-1} \ZZ$, such that $\Pb_{\uu^*} = \Nc(0,\Hb^{-1} \Mb \Hb^{-1})$. Hence , $c = 0$.

\mds

If $\lambda_0 \neq 0$ or $\gamma_0 \neq 0$, the necessary and sufficient optimality conditions for a group $k$ tell us that $\uu^*$ satisfies
\beq \label{KKTopt}
\left\{\begin{array}{llll}
(\Hb \uu^* + \ZZ)_{(k)} + \lambda_0 \alpha_k \pp^{(k)} + \gamma_0 \xi_k \cfrac{\theta^{(k)}_0}{\|\theta^{(k)}_0\|_2} = 0, & & k \in \Sc,\\
(\Hb \uu^* + \ZZ)_{(k)} + \lambda_0 \alpha_k \ww^{(k)} + \gamma_0 \xi_k \zz^{(k)} = 0, & & \, \text{otherwise},
\end{array}\right.
\eeq
where $\ww^{(k)}$ and $\zz^{(k)}$ are the subgradients of $\|\uu^{(k)}\|_1$ and $\|\uu^{(k)}\|_2$ given by
\beqw
\ww^{(k)}_i = \begin{cases}
\text{sgn}(\uu^{(k)}_i) \, \text{if} \, \uu^{(k)}_i \neq 0,\\
\in \{\ww^{(k)}_i : |\ww^{(k)}_i| \leq 1\} \, \text{if} \, \uu^{(k)}_i = 0,
\end{cases}
\,
\zz^{(k)} = \begin{cases}
\cfrac{\uu^{(k)}}{\|\uu^{(k)}\|_2} \, \text{if} \, \uu^{(k)} \neq 0,\\
\in \{\zz^{(k)} : \|\zz^{(k)}\|_2 \leq 1\} \, \text{if} \, \uu^{(k)} = 0,
\end{cases}
\eeqw
and $\pp^{(k)}_i = \partial_{\uu_i} \{ |\uu^{(k)}_i| \mathbf{1}_{\theta^{(k)}_{0,i} = 0} + \uu^{(k)}_i \text{sgn}(\theta^{(k)}_{0,i})\mathbf{1}_{\theta^{(k)}_{0,i} \neq 0} \}$.

\mds

If $\uu^{(m) *} = 0, \forall m \notin \Sc$, then the optimality conditions (\ref{KKTopt}) become
\beq \label{inactivegroup}
\left\{\begin{array}{llll}
\Hb_{\Sc \Sc} \uu^*_{\Sc} + \ZZ_{\Sc} + \lambda_0 \tau_{\Sc} + \gamma_0 \zeta_{\Sc} = 0, & &\\
\|-\Hb_{(l) \Sc} \uu^*_{\Sc} - \ZZ_{(l)} -\lambda_0 \alpha_l \ww^{(l)}\|_2 \leq \gamma_0 \xi_l, \, \text{as} \, \|\zz^{(l)}\|_2 \leq 1, \, l \in \Sc^c, & &
\end{array}\right.
\eeq
with $\tau_{\Sc} = \text{vec}(k \in \Sc, \alpha_k \pp^{(k)})$ and $\zeta_{\Sc} = \text{vec}(k \in \Sc, \xi_k \cfrac{\theta^{(k)}_0}{\|\theta^{(k)}_0\|_2})$, which are vectors of $\Rb^{\text{card}(\Sc)}$.
\mds

For $k \in \Sc$, that is the vector $\theta^{(k)}_0$ is at least non-zero, then
\beq \label{KKToptwithin}
\left\{\begin{array}{llll}
(\Hb \uu^* + \ZZ)_i + \lambda_0 \alpha_k \text{sgn}(\theta^{(k)}_{0,i}) + \gamma_0 \xi_k \cfrac{\theta^{(k)}_{0,i}}{\|\theta^{(k)}_0\|_2} = 0, \, \text{if} \, k \in \Sc, i \in \Ac_k, & &\\
(\Hb \uu^* + \ZZ)_i + \lambda_0 \alpha_k  \ww^{(k)}_i  = 0, \, i \in \Ac^c_k. & &
\end{array}\right.
\eeq
Consequently, if $\uu^{(k) *}_i = 0, \forall i \in \Ac^c_k$, with $k \in \Sc$, then the conditions (\ref{KKToptwithin}) become
\beqw
\left\{\begin{array}{llll}
\Hb_{\Ac_k \Ac_k} \uu^*_{\Ac_k} + \ZZ_{\Ac_k} + \lambda_0 \alpha_k \text{sgn}(\theta_{0,\Ac_k}) + \gamma_0 \xi_k \cfrac{\theta_{0,\Ac_k}}{\|\theta_{0,\Ac_k}\|_2} = 0, & &\\
|-(\Hb_{\Ac^c_k \Ac_k} \uu^*_{\Ac_k} + \ZZ_{\Ac^c_k})_i| \leq \lambda_0 \alpha_k. & &
\end{array}\right.
\eeqw

\mds

Combining relationships in (\ref{inactivegroup}), we obtain
\beqw
\|\Hb_{(l) \Sc} \Hb^{-1}_{\Sc \Sc} (\ZZ_{\Sc} + \lambda_0 \tau_{\Sc} + \gamma_0 \zeta_{\Sc}) - \ZZ_{(l)} -\lambda_0 \alpha_l \ww^{(l)}\|_2 \leq \gamma_0 \xi_l, l \in \Sc^c.
\eeqw
The same reasoning applies for active groups with inactive components, such that combining relationships in (\ref{KKToptwithin}), we obtain
\beqw
|(\Hb_{\Ac^c_k \Ac_k} \Hb^{-1}_{\Ac_k \Ac_k} (\ZZ_{\Ac_k} + \lambda_0 \alpha_k \text{sgn}(\theta_{0,\Ac_k}) + \gamma_0 \xi_k \cfrac{\theta_{0,\Ac_k}}{\|\theta_{0,\Ac_k}\|_2}) - \ZZ_{\Ac^c_k})_i| \leq \lambda_0 \alpha_k.
\eeqw

\mds

Hence we deduce
\beqw
\begin{array}{llll}
\Pb(\forall k =1,\cdots,m, \forall \in \Ac^c_k, \uu^{(k) *}_i = 0) \leq && \\
\min(\Pb(k \notin \Sc, \uu^{(k) *} = 0),\Pb(k \in \Sc, \forall i \in \Ac^c_k, \uu^{(k) *}_i = 0)) := \min(a_1,a_2).
\end{array}
\eeqw
Under the assumption that $\lambda_0 < \infty$ and $\gamma_0 < \infty$, we obtain
\beqw
\begin{array}{llll}
a_1 = \Pb(l \in\Sc^c, \|\Hb_{(l) \Sc} \Hb^{-1}_{\Sc \Sc} (\ZZ_{\Sc} + \lambda_0 \tau_{\Sc} + \gamma_0 \zeta_{\Sc}) - \ZZ_{(l)} -\lambda_0 \alpha_l \ww^{(l)}\|_2 \leq \gamma_0 \xi_l) < 1, & & \\
a_2 = \Pb(k \in \Sc, i \in \Ac^c_k, |(\Hb_{\Ac^c_k \Ac_k} \Hb^{-1}_{\Ac_k \Ac_k} (\ZZ_{\Ac_k} + \lambda_0 \alpha_k \text{sgn}(\theta_{0,\Ac_k}) + \gamma_0 \xi_k \cfrac{\theta_{0,\Ac_k}}{\|\theta_{0,\Ac_k}\|_2}) - \ZZ_{\Ac^c_k})_i| \leq \lambda_0 \alpha_k) < 1.&&
\end{array}
\eeqw
Thus $c < 1$, which proves (\ref{upper_bound_select}), that is proposition \ref{proposition1}.
\end{proof}

\bre
The result in Proposition \ref{proposition1} highlights that the SGL as proposed by Simon and al. (2013) cannot achieve the oracle property since the penalties cannot recover the unknown set of active indices $\Ac$, which is called model selection consistency. To fix this drawback in an ordinary least square framework, Zou (2006) proposes the adaptive LASSO, where random weights are used to penalize different coefficients and proves that the adaptive LASSO estimator satisfies the oracle property in the sense of Fan and Li (2001), that is asymptotic normality and selection consistency for a proper choice of $\lambda_T$ and $\alpha^{(k)}_i$.  That is also the case for the adaptive group LASSO model proposed by Nardi and Rinaldo (2008), where adaptive weights are used to penalize grouped coefficients differently. We propose the same approach than Zou (2006) and use adaptive weights in the penalties such that the adaptive SGL enjoys the oracle property in the sense of Fan and Li (2001) as proved in Theorem \ref{oracle1}.
\ere

The adaptive specification of the proposed estimator now becomes
\beq \label{2ndcrit}
\hat{\theta} = \underset{\theta \in \Theta}{\arg \, \min} \, \{\Gb_T \psi(\theta)\},
\eeq
where
\beqw
\begin{array}{llll}
\theta \mapsto \Gb_T \psi(\theta) & = & \cfrac{1}{T} \overset{T}{\underset{t=1}{\sum}} l(\eps_t;\theta) + \pp_1(\lambda_T,\tilde{\theta},\theta) + \pp_2(\gamma_T,\tilde{\theta},\theta) \\
& = & \Gb_T l(\theta) + \pp_1(\lambda_T,\tilde{\theta},\theta) + \pp_2(\gamma_T,\tilde{\theta},\theta),
\end{array}
\eeqw
such that both penalties are specified as
\beqw
\pp_1(\lambda_T,\tilde{\theta},\theta) = \lambda_T T^{-1} \overset{m}{\underset{k=1}{\sum}} \overset{\cc_k}{\underset{i=1}{\sum}} \alpha(\tilde{\theta}^{(k)}_i) |\theta^{(k)}_i|, \pp_2(\gamma_T,\tilde{\theta},\theta) = \gamma_T T^{-1} \overset{m}{\underset{l=1}{\sum}} \xi(\tilde{\theta}^{(l)}) \|\theta^{(l)}\|_2.
\eeqw
These penalties are now randomized through the $\tilde{\theta}$ argument in the weights $\alpha$'s and $\xi$'s. This first step estimator $\tilde{\theta}$ is supposed to be a $T^{1/2}$-consistent estimator of $\theta_0$. For instance, it can be defined as an M-estimator of the unpenalized empirical criterion $\Gb_T l(.)$, that is
\beqw
\tilde{\theta} = \underset{\theta \in \Theta}{\arg \, \min} \, \Gb_T l(\theta).
\eeqw
Adaptive weights are also used by Zou and Zhang (2009), who plug the elastic-net estimator in the adaptive weight and then estimate a new elastic net model using these weights, that is the adaptive elastic net.

\mds

The weights we use are now random and for any group $k$ or $l$, $\alpha(\tilde{\theta}^{(k)}) \in \Rb^{\cc_k}_+$, $\xi(\tilde{\theta}^{(l}) \in \Rb_+$ are specified as
\beqw
\alpha^{(k)}_T := \alpha(\tilde{\theta}^{(k)}) = (|\tilde{\theta}^{(k)}_i|^{-\eta},i=1,\cdots,\cc_k), \, \xi_{T,l} := \xi(\tilde{\theta}^{(l}) = \|\tilde{\theta}^{(l)}\|^{-\mu}_2,
\eeqw
for some constants $\eta >0$ and $\mu > 0$ (to be specified).

\mds

\bre
Theorem \ref{bound_prob} can be adapted to the adaptive specification of the SGL in (\ref{2ndcrit}), where the bound in probability of the error would be
\beqw
\|\hat{\theta} - \theta_0 \| = O_p(T^{-1/2} + \lambda_T T^{-1} a_T + \gamma_T T^{-1} b_T),
\eeqw
with $a_T = \text{card}(\Ac) . \{\underset{k \in \Sc}{\max} \, (\underset{i \in \Ac_k}{\max} \, \alpha^{(k)}_{T,i})\}$, $b_T = \text{card}(\Ac) . \{\underset{l \in \Sc}{\max} \, \xi_{T,l}\}$, such that $\lambda_T T^{-1} a_T \overset{\Pb}{\rightarrow} 0$ and $\gamma_T T^{-1} b_T \overset{\Pb}{\rightarrow} 0$. The proof follows exactly the same steps as for Theorem (\ref{bound_prob}), except $a_T$ and $b_T$ are random quantities.
\ere

\begin{theorem} \label{oracle1}
Under assumptions \ref{H1}-\ref{third_or}, if $\lambda_T T^{-1/2} \rightarrow 0$, $\gamma_T T^{-1/2} \rightarrow 0$, $T^{(\eta-1)/2} \lambda_T \rightarrow \infty$ and $T^{(\mu-1)/2} \gamma_T \rightarrow \infty$,  then $\hat{\theta}$ obtained in (\ref{2ndcrit}) satisfies
\beqw
\begin{array}{rclcl}
\underset{T \rightarrow \infty}{\lim} \Pb(\hat{\Ac} = \Ac) & = & 1, \, \text{and} \\
\sqrt{T} (\hat{\theta}_{\Ac} - \theta_{0,\Ac})& \overset{d}{\longrightarrow} & \Nc(0,\Hb^{-1}_{\Ac \Ac} \Mb_{\Ac \Ac} \Hb^{-1}_{\Ac \Ac}).
\end{array}
\eeqw
\end{theorem}
\begin{proof}
We start with the asymptotic distribution and proceed as in the proof of Theorem \ref{asym_dist_Knight}, where we used Lemma \ref{convex_lemma}. To do so, we prove the finite dimensional convergence in distribution of the empirical criterion $\Fb_T(\uu)$ to $\Fb_{\infty}(\uu)$ with $\uu \in \Rb^d$, where these quantities are respectively  defined as
\beqw
\begin{array}{llll}
\Fb_T(\uu) & = & T \Gb_T (\psi(\theta_0 + \uu/\sqrt{T}) - \psi(\theta_0)) \nonumber \\
& = & T \Gb_T (l(\theta_0 + \uu/\sqrt{T}) - l(\theta_0)) + \lambda_T \overset{m}{\underset{k=1}{\sum}} \overset{\cc_k}{\underset{i=1}{\sum}} \alpha^{(k)}_{T,i} [|\theta^{(k)}_{0,i} + \uu^{(k)}_i/\sqrt{T}| - |\theta^{(k)}_{0,i}|] \nonumber \\
& + & \gamma_T \overset{m}{\underset{l=1}{\sum}} \xi_{T,l} [\|\theta^{(l)}_0 + \uu^{(l)}/\sqrt{T}\|_2 - \|\theta^{(l)}_0\|_2],
\end{array}
\eeqw
and
\beq \label{distr_oracle}
\Fb_{\infty}(\uu) = \begin{cases}
\cfrac{1}{2} \uu'_{\Ac} \Hb_{\Ac \Ac} \uu_{\Ac} + \uu_{\Ac}' \ZZ_{\Ac} & \text{if} \; \uu_i = 0, \; \text{when} \;  i \notin \Ac, \, \text{and} \\
\infty & \text{otherwise},
\end{cases}
\eeq
with $\ZZ_{\Ac} \sim \Nc(0,\Mb_{\Ac \Ac})$. By Lemma \ref{convex_lemma}, the finite dimensional convergence in distribution implies $\underset{\uu \in \Rb^d}{\arg \, \min}\{\Fb_T(\uu)\} \overset{d}{\longrightarrow} \underset{\uu \in \Rb^d}{\arg \, \min}\{\Fb_{\infty}(\uu)\}$. We first consider the unpenalized empirical criterion of $\Fb_T(.)$, which can be expanded as
\beqw
T \Gb_T (\psi(\theta_0 + \uu/\sqrt{T}) - \psi(\theta_0)) = T^{1/2}\dot{\Gb}_T l(\theta_0) \uu + \cfrac{1}{2} \uu' \ddot{\Gb}_T l(\bar{\theta}) \uu + \cfrac{1}{6 T^{1/3}} \nabla'\{\uu' \ddot{\Gb}_T l(\bar{\theta})\}\uu,
\eeqw
where $\bar{\theta}$ lies between $\theta_0$ and $\theta_0 + \uu/\sqrt{T}$. First, using the same reasoning on the third order term, we obtain
\beqw
\cfrac{1}{6 T^{1/3}} \nabla'\{\uu' \ddot{\Gb}_T l(\bar{\theta})\}\uu \overset{\Pb}{\underset{T \rightarrow \infty}{\longrightarrow}} 0.
\eeqw
By the ergodic theorem, we deduce $\ddot{\Gb}_T l(\theta_0) \overset{\Pb}{\underset{T \rightarrow \infty}{\longrightarrow}} \Hb$ and by assumption \ref{H4}, $\sqrt{T}\dot{\Gb}_T l(\theta_0) \overset{d}{\longrightarrow} \Nc(0,\Mb)$.

\mds

We now focus on the penalty terms of (\ref{2ndcrit}), we remind that $\alpha^{(k)}_{T,i} = |\tilde{\theta}^{(k)}_i|^{-\eta}$, such that for $i \in \Ac_k, k \in \Sc$, $\tilde{\theta}^{(k)}_i \overset{\Pb}{\rightarrow} \theta^{(k)}_{0,i} \neq 0$. Note that
\beqw
\sqrt{T}(|\theta^{(k)}_0 + \uu^{(k)}/\sqrt{T}| - |\theta^{(k)}_0|] \overset{\Pb}{\underset{T \rightarrow \infty}{\longrightarrow}} \uu^{(k)}_i \text{sgn}(\theta^{(k)}_{0,i})\mathbf{1}_{\theta^{(k)}_{0,i} \neq 0}.
\eeqw
This implies that, for $i \in \Ac_k$, $k \in \Sc$, we have
\beqw
\lambda_T T^{-1/2}\overset{\cc_k}{\underset{i=1}{\sum}} \alpha^{(k)}_{T,i} \sqrt{T}(|\theta^{(k)}_{0,i} + \uu^{(k)}_i/\sqrt{T}| - |\theta^{(k)}_{0,i}|) \overset{\Pb}{\underset{T \rightarrow \infty}{\longrightarrow}} 0,
\eeqw
under the condition $\lambda_T T^{-1/2} \rightarrow 0$. For $i \in \Ac^c_k$, $\theta^{(k)}_{0,i} = 0$, then $T^{\eta/2} (|\tilde{\theta}^{(k)}_i|)^{\eta} = O_p(1)$. Hence under the assumption $\lambda_T T^{(\eta-1)/2} \rightarrow \infty$, we obtain
\beq \label{asymp_div}
\lambda_T T^{-1/2} \alpha^{(k)}_{T,i} \sqrt{T}(|\theta^{(k)}_{0,i} + \uu^{(k)}_i/\sqrt{T} | - |\theta^{(k)}_{0,i}|) = \lambda_T T^{-1/2} |\uu^{(k)}_i| \cfrac{T^{\eta/2}}{(T^{1/2}|\tilde{\theta}^{(k)}_i|)^{\eta}} \overset{\Pb}{\underset{T \rightarrow \infty}{\longrightarrow}} \infty.
\eeq

\mds

As for the $l^1/l^2$ quantity, we remind that $\xi_{T,l} = \|\tilde{\theta}^{(l)}\|^{-\mu}_2$, such that for $l \in \Sc$, $\tilde{\theta}^{(l)} \overset{\Pb}{\rightarrow} \theta^{(l)}_0$, and in this case
\beqw
\sqrt{T} \{ \|\theta^{(l)}_0 + \uu^{(l)}/\sqrt{T}\|_2 - \|\theta^{(l)}_0\|_2\} = \cfrac{\uu^{(l) '} \theta^{(l)}_0}{\|\theta^{(l)}_0\|_2} + o(T^{-1/2}).
\eeqw
Consequently, using $\gamma_T T^{-1/2} \rightarrow 0$, and for $l \in \Sc$, we obtain
\beqw
\gamma_T T^{-1/2} \sqrt{T} \xi_{T,l} ( \|\theta^{(l)}_0 + \uu^{(l)}/\sqrt{T}\|_2 - \|\theta^{(l)}_0\|_2 ) \overset{\Pb}{\underset{T \rightarrow \infty}{\longrightarrow}} 0.
\eeqw
Combining the fact $k \in \Sc$ and $\theta^{(k)}_0$ is partially zero, that is $i \in \Ac^c_k$, we obtain the divergence given in (\ref{asymp_div}). Furthermore, if $l \notin \Sc$, that is $\theta^{(l)}_0 = 0$, then
\beqw
\sqrt{T} \{ \|\theta^{(l)}_0 + \uu^{(l)}/\sqrt{T}\|_2 - \|\theta^{(l)}_0\|_2\} = \|\uu^{(l)}\|_2,
\eeqw
and $T^{\mu/2} (\|\tilde{\theta}^{(l)}\|_2)^{\mu} = O_p(1)$, then under the assumption $\gamma_T T^{(\mu-1)/2} \rightarrow \infty$, we obtain
\beqw
\gamma_T T^{-1/2} \xi_{T,l} \sqrt{T}[\|\theta^{(l)}_0 + \uu^{(l)}/\sqrt{T}\|_2 - \|\theta^{(l)}_0\|_2 ] =  \gamma_T T^{-1/2} \|\uu^{(l)}\|_2 \cfrac{T^{\mu/2}}{(T^{1/2}\|\tilde{\theta}^{(l)}\|_2)^{\mu}} \overset{\Pb}{\underset{T \rightarrow \infty}{\longrightarrow}} \infty.
\eeqw
We deduce the pointwise convergence $\Fb_T (\uu) \overset{d}{\longrightarrow} \Fb_{\infty}(\uu)$, where $\Fb_{\infty}(.)$ is given in (\ref{distr_oracle}). As $\Fb_T(.)$ is convex and $\Fb_{\infty}(.)$ is convex and has a unique minimum $(\Hb^{-1}_{\Ac \Ac} \ZZ_{\Ac},\mathbf{0}_{\Ac^c})$, by Lemma \ref{convex_lemma}, we obtain
\beqw
\sqrt{T}(\hat{\theta} - \theta_0) = \underset{\uu \in \Rb^d}{\arg \, \min}\{\Fb_T(\uu)\} \overset{d}{\longrightarrow} \underset{\uu \in \Rb^d}{\arg \, \min}\{\Fb_{\infty}(\uu)\},
\eeqw
that is to say
\beqw
\sqrt{T}(\hat{\theta}_{\Ac} - \theta_{0,\Ac}) \overset{d}{\longrightarrow} \Hb^{-1}_{\Ac \Ac} \ZZ_{\Ac}, \; \text{and} \; \sqrt{T}(\hat{\theta}_{\Ac^c} - \theta_{0,\Ac^c}) \overset{d}{\longrightarrow} \bf{0}_{\Ac^c}.
\eeqw

\mds

We now prove the model selection consistency. Let $i \in \Ac_k$, then by the asymptotic normality result, $\hat{\theta}^{(k)}_i \overset{\Pb}{\underset{T \rightarrow \infty}{\longrightarrow}} \theta^{(k)}_0$, which implies $\Pb(i \in \hat{\Ac}_k) \rightarrow 1$. Thus the proof consists of proving
\beqw
\forall k = 1,\cdots,m, \forall i \in \Ac^c_k, \Pb(i \in \hat{\Ac}_k) \rightarrow 0.
\eeqw
This problem can be split into two parts as
\beq \label{selection_split}
\forall k \notin \Sc, \Pb(k \in \hat{\Sc}) \rightarrow 0, \; \text{and} \; \forall k \in \Sc, \forall i \in \Ac^c_k, \Pb(i \in \hat{A}_k) \rightarrow 0.
\eeq
Let us start with the case $k \notin \Sc$. If $k \in \hat{\Sc}$, by the optimality conditions given by the Karush-Kuhn-Tucker theorem applied on $\Gb_T \psi(\hat{\theta})$, we have
\beqw
\dot{\Gb}_T l(\hat{\theta})_{(k)} + \cfrac{\lambda_T}{T} \alpha^{(k)}_T \odot \hat{\ww}^{(k)}+ \cfrac{\gamma_T}{T} \xi_{T,k} \cfrac{\hat{\theta}^{(k)}}{\|\hat{\theta}^{(k)}\|_2} = 0,
\eeqw
$\odot$ is the Hadamard product and
\beqw
\hat{\ww}^{(k)}_i = \begin{cases}
\text{sgn}(\hat{\theta}^{(k)}_i) \; \text{if} \; \hat{\theta}^{(k)}_i \neq 0,\\
\in \{\hat{\ww}^{(k)}_i : |\hat{\ww}^{(k)}_i| \leq 1\} \; \text{if} \; \hat{\theta}^{(k)}_i = 0.
\end{cases}
\eeqw
Multiplying the unpenalized part by $T^{1/2}$, we have the expansion
\beqw
\begin{array}{llll}
T^{1/2} \dot{\Gb}_T l(\hat{\theta})_{(k)} & = & T^{1/2} \dot{\Gb}_T l(\theta_0)_{(k)} + T^{1/2} \ddot{\Gb}_T l(\theta_0)_{(k) (k)} (\hat{\theta}-\theta_0)_{(k)} \\
& + & T^{1/2} \nabla'\{(\hat{\theta}-\theta_0)'_{(k)} \ddot{\Gb}_T l(\bar{\theta})_{(k)(k)} (\hat{\theta}-\theta_0)_{(k)}\} ,
\end{array}
\eeqw
which is asymptotically normal by consistency, assumption \ref{third_or} regarding the bound on the third order term, the Slutsky theorem and the central limit theorem of Billingsley (1961). Furthermore, we have
\beqw
\gamma_T T^{-1/2}\xi_{T,k} \cfrac{\hat{\theta}^{(k)}}{\|\hat{\theta}^{(k)}\|_2} = \gamma_T T^{(\mu-1)/2} (T^{1/2} \|\tilde{\theta}^{(k)}\|_2 )^{-\mu}\cfrac{\hat{\theta}^{(k)}}{\|\hat{\theta}^{(k)}\|_2} \overset{\Pb}{\underset{T \rightarrow \infty}{\longrightarrow}} \infty.
\eeqw
We obtain the same when adding $\lambda_T T^{-1/2} \alpha^{(k)}_T \odot \hat{\ww}^{(k)}$. Therefore, we have
\beqw
\forall k \notin \Sc, \Pb(k \in \hat{\Sc}) \leq \Pb(-\dot{\Gb}_T l(\hat{\theta})_{(k)} = \cfrac{\lambda_T}{T} \alpha^{(k)}_T \odot \hat{\ww}^{(k)}_i + \cfrac{\gamma_T}{T} \xi_{T,k} \cfrac{\hat{\theta}^{(k)}}{\|\hat{\theta}^{(k)}\|_2}) \rightarrow 0.
\eeqw
\mds

We now pick $k \in \Sc$ and consider the event $\{i \in \hat{\Ac}_k\}$. Then the Karush-Kuhn-Tucker conditions for $\Gb_T \psi(\hat{\theta})$ are given by
\beqw
(\dot{\Gb}_T l(\hat{\theta}))_{(k),i} + \cfrac{\lambda_T}{T} \alpha^{(k)}_{T,i} \text{sgn}(\hat{\theta}^{(k)}_{T,i}) + \cfrac{\gamma_T}{T} \xi_{T,k} \cfrac{\hat{\theta}^{(k)}_i}{\|\hat{\theta}^{(k)}\|_2} = 0.
\eeqw
Using the same reasoning as previously, $T^{1/2}(\dot{\Gb}_T l(\hat{\theta}))_{(k),i}$ is also asymptotically normal, and $\tilde{\theta}^{(k)} \overset{\Pb}{\underset{T \rightarrow \infty}{\longrightarrow}} \theta^{(k)}_0$ for $k \in \Sc$, and besides
\beqw
\lambda_T T^{-1/2}\alpha^{(k)}_{T,i} \text{sgn}(\hat{\theta}^{(k)}_i) = \lambda_T \cfrac{T^{(\eta-1)/2}}{(T^{1/2}|\tilde{\theta}^{(k)}_i|)^{\eta}} \overset{\Pb}{\underset{T \rightarrow \infty}{\longrightarrow}} \infty,
\eeqw
such that we obtain the same when adding $\gamma_T T^{-1/2}\xi_{T,k} \cfrac{\hat{\theta}^{(k)}_i}{\|\hat{\theta}^{(k)}\|_2}$. Therefore, we have
\beqw
\forall k \in \Sc, \forall i \notin \Ac_k, \Pb(i \in \hat{\Ac}_k) \leq \Pb(-(\dot{\Gb}_T l(\hat{\theta}))_{(k),i} = \cfrac{\lambda_T}{T} \alpha^{(k)}_{T,i} \text{sgn}(\hat{\theta}^{(k)}_i) + \cfrac{\gamma_T}{T} \xi_{T,k} \cfrac{\hat{\theta}^{(k)}_i}{\|\hat{\theta}^{(k)}\|_2}) \rightarrow 0.
\eeqw
We have proved (\ref{selection_split}).
\end{proof}

\section{Double asymptotic}

In the previous sections, we worked with a fixed dimension $d$, where $d = \overset{m}{\underset{k=1}{\sum}} \cc_k$. From now on, let us consider the case where $d = d_T$, such that $d_T \rightarrow \infty$ as $T \rightarrow \infty$. Note that $\text{card}(\Sc) = O(\text{card}(\Ac)) = O(d_T)$. The speed of growth of the dimension is supposed to be $d_T = O(T^c)$ for some $q_2<c<q_1$. In this section, we prove that the adaptive SGL enjoys the oracle property, that is model selection consistency and optimal rate of convergence for proper choices of $0 \leq q_1<q_2<1$. We highlight that our general framework unfortunatly hampers a high degree of flexibility on the behavior of $d_T$, that is $c$ cannot be set in $(0,1)$. This issue was encountered by Fan and Peng (2004) in an i.i.d. and non-adaptive framework. This lack of flexibility is a necessary cost to cope with the random remainder of the Taylor expansions as we should take the third order term into account. This problem is moved aside when considering the simple linear model, where the third order derivative is zero. For instance, Zou and Zhang (2009) proved the oracle property of the adaptive elastic-net in a double-asymptotic framework for linear models where $0 \leq c < 1$.

\mds

For the asymptotic normality, we use the method of Fan and Peng (2004) and Zou and Zhang (2009), where we derive the asymptotic distribution of the discrepancy $\sqrt{T}(\hat{\theta}-\theta_0)_{\Ac}$ times a matrix sequence $(Q_T)$ of size $r \times \text{card}(\Ac)$, $r$ being arbitrary but finite. This allows for switching from infinite dimensional distribution to finite dimensional distribution, where we can apply usual tools of asymptotic analysis.

\mds

In this section, we provide the conditions to achieve the oracle property as in Fan and Peng (2004) or Zou and Zhang (2009).  In this double asymptotic framework, the quantities depend on $d_T$, hence on $T$. They should be indexed by $T$, which expresses that the dimension depend on the sample size. In the rest of the paper, we denote $\Hb_T := \Eb[\nabla^2_{\theta \theta} l(\eps_t;\theta_0)]$ and $\Mb_T := \Eb[\nabla_{\theta} l(\eps_t;\theta_0) \nabla_{\theta'} l(\eps_t;\theta_0)]$. To make the reading easier, we do not index other quantities by $T$, which will be implicit. We remind that the criterion is
\beq \label{DA_SGL}
\begin{array}{llll}
\hat{\theta} & = & \underset{\theta \in \Theta}{\arg \, \min} \{ \Gb_T l(\theta) + \pp_1(\lambda_T,\tilde{\theta},\theta) + \pp_2(\gamma_T,\tilde{\theta},\theta)\} \\
& = & \underset{\theta \in \Theta}{\arg \, \min} \{\cfrac{1}{T} \overset{T}{\underset{t=1}{\sum}} l(\eps_t;\theta) + \cfrac{\lambda_T}{T} \overset{m}{\underset{k=1}{\sum}} \overset{\cc_k}{\underset{i=1}{\sum}} \alpha^{(k)}_{T,i} |\theta^{(k)}_i|+\cfrac{\gamma_T}{T}\overset{m}{\underset{l=1}{\sum}} \xi_{T,l} \|\theta^{(l)}\|_2\},
\end{array}
\eeq
with $\alpha^{(k)}_{T,i} = |\tilde{\theta}^{(k)}_i|^{-\eta}$ and $\xi_{T,l} = \|\tilde{\theta}^{(l)}\|^{-\mu}_2$, where $\eta > 0, \mu > 0$, and $\tilde{\theta}$ is a first step estimator satisfying
\beqw
\tilde{\theta} = \underset{\theta \in \Theta}{\arg \, \min} \{ \Gb_T l(\theta) \}.
\eeqw

\mds

The double asymptotic framework implies that the empirical criterion can be viewed as a sequence of dependent random variables for which we need refined asymptotic theorems for dependent sequence of arrays. Shiryaev (1991) proposes a version of the central limit theorem for dependent sequence of arrays, provided this sequence is a square integrable martingale difference satisfying the so-called Lindeberg condition. A similar theorem can be found in Billingsley (1995, theorem 35.12, p.476). We provide here the theorem of Shiryaev (see Theorem 4, p.543 of Shiryaev, 1991) that we will use to derive the asymptotic distribution of the adaptive SGL estimator.
\begin{theorem} \label{Shishi} \emph{(Shiryaev, 1991)} \\
Let a sequence of square-integrable martingale differences $\xi^T = (\xi_{T,t},\Fc^T_t),T \geq 0$, with $\Fc^T_t = \sigma(\xi_{T,s},s \leq t)$, satisfy the Lindeberg condition, for $\eps > 0$, given by
\beqw
\overset{T}{\underset{t=0}{\sum}} \Eb[\xi^2_{T,t} \mathbf{1}_{|\xi_{T,t}| > \eps} | \Fc^T_{t-1} ] \overset{\Pb}{\underset{T \rightarrow \infty}{\longrightarrow}} 0,
\eeqw
then if $\overset{T}{\underset{t=0}{\sum}} \Eb[\xi^2_{T,t}| \Fc^T_{t-1} ] \overset{\Pb}{\underset{T \rightarrow \infty}{\longrightarrow}} \sigma^2_t$, or $\overset{T}{\underset{t=0}{\sum}} \xi^2_{T,t} \overset{\Pb}{\underset{T \rightarrow \infty}{\longrightarrow}} \sigma^2_t$, then $\overset{T}{\underset{t=0}{\sum}} \xi_{T,t} \overset{d}{\longrightarrow} \Nc(0,\sigma^2_t).$
\end{theorem}

\bre
Note that central limit theorems relaxing the stationarity and martingale difference assumptions for sequences of arrays exist. Neumann (2013) proposes such a central limit theorem for weakly dependent sequences of arrays. Such sequences should also satisfy a Lindeberg condition and conditions on covariances. In the rest of the paper, we use Shiryaev's result.
\ere

\mds

We consider problem (\ref{DA_SGL}), which is the adaptive SGL estimator. In the first step, we study the convergence rate of the first step unpenalized estimator, which is plugged in the adaptive specification. The convergence rate of a classic M-estimator is $T^{1/2}$, for $d$ fixed. For $d$ diverging, we need some additional assumptions.

\mds

The two next assumptions are similar to condition (F) of Fan and Peng (2004) and allows for controlling the minimum and maximum eigenvalues of the limits of the empirical Hessian and the score cross-product. We denote by $\lambda_{\min}(\MM)$ and $\lambda_{\max}(\MM)$ the minimum and maximum eigenvalues of any positive definite square matrix $\MM$.
\begin{assumption}\label{H6}
$\Hb_T$ and $\Mb_T$ exist. $\Hb_T$ is nonsingular, and there exist $b_1,b_2$ with  $0<b_1<b_2<\infty$ and $c_1,c_2$ with $0<c_1<c_2<\infty$ such that, for all $T$,
\beqw
b_1 < \lambda_{\min}(\Mb_T) < \lambda_{\max}(\Mb_T) < b_2, \, c_1 < \lambda_{\min}(\Hb_T) < \lambda_{\max}(\Hb_T) < c_2.
\eeqw
Let $\Vb_T = \Hb^{-1}_T \Mb_T \Hb^{-1}_T$, we deduce there exist $a_1,a_2$ with $0<a_1<a_2<\infty$ such that, for all $T$,
\beqw
a_1 < \lambda_{\min}(\Vb_T) < \lambda_{\max}(\Vb_T) < a_2.
\eeqw
\end{assumption}

\begin{assumption} \label{H7}
$\Eb[\{\nabla_{\theta} l(\eps_t;\theta_0) \nabla_{\theta'} l(\eps_t;\theta_0) \}^2] < \infty$, for every $d_T$ (and then of $T$).
\end{assumption}

\begin{assumption}\label{H8}
There exist some functions $\Psi(.)$ such that, for all $T$,
\beqw
\underset{k=1,\cdots,d_T}{\sup} \Eb[\partial_{\theta_k}l(\eps_t;\theta) \partial_{\theta_k}l(\eps_{t'};\theta)] \leq \Psi(|t-t'|),
\eeqw
and
\beqw
\underset{T}{\sup} \, \cfrac{1}{T} \overset{T}{\underset{t,t'=1}{\sum}} \Psi(|t-t'|) < \infty.
\eeqw
\end{assumption}

\begin{assumption} \label{H9}
Let $\zeta_{kl,t} := \partial^2_{\theta_k \theta_l} l(\eps_t;\theta_0) - \Eb[\partial^2_{\theta_k \theta_l} l(\eps_t;\theta_0)]$. There exist some functions $\chi(.)$ such that
\beqw
|\Eb[\zeta_{kl,t} \zeta_{k'l',t'}]| \leq \chi(|t-t'|),
\eeqw
and
\beqw
\underset{T}{\sup} \, \cfrac{1}{T} \overset{T}{\underset{t,t'=1}{\sum}} \chi(|t-t'|) < \infty.
\eeqw
\end{assumption}

\begin{assumption}\label{H10}
Let $\upsilon_t(C) := \underset{k,l,m = 1,\cdots,d_T}{\sup} \{ \underset{\theta:\|\theta-\theta_0\|_2 \leq \nu_T C}{\sup} |\partial^3_{\theta_k \theta_l \theta_m} l(\eps_t;\theta)|\}$, where $C > 0$ is a fixed constant and $\nu_T = (d_T/T)^{1/2}$. Then
\beqw
\eta(C):= \cfrac{1}{T^2} \overset{T}{\underset{t,t'=1}{\sum}} \Eb[\upsilon_t(C) \upsilon_{t'}(C)] < \infty.
\eeqw
\end{assumption}

\begin{theorem}\label{bound_prob_double_init}
Under Assumptions \ref{H1}-\ref{H3}, \ref{H6}-\ref{H10} and if $d^4_T = o(T)$, the sequence of unpenalized M-estimators solving $\tilde{\theta} = \underset{\theta \in \Theta}{\arg \, \min} \, \{\Gb_T l(\theta)\}$ satisfies
\beqw
\|\tilde{\theta} - \theta_0\|_2 = O_p((\cfrac{d_T}{T})^{\frac{1}{2}}).
\eeqw
Both vectors $\tilde{\theta}$ and $\theta_0$ depend on $T$ such that $\tilde{\theta} = \tilde{\theta}_T$ and $\theta_0 = \theta_{0,T}:= \theta_{0,\infty}.e_T$.
\end{theorem}
\bre
Note that this consistency result requires at most $d^4_T = o(T)$, as Theorem 1 of Fan and Peng (2004).
\ere
\begin{proof}
We proceed as in the proof of Theorem \ref{bound_prob}. We denote $\nu_T = (d_T / T)^{1/2}$ and we would like to prove that, for any $\eps > 0$, there exists $C_{\eps} > 0$ such that
\beq \label{consistency_double_first_un}
\Pb(\|\tilde{\theta} - \theta_0\|_2/\tilde{\nu}_T > C_{\eps}) < \eps.
\eeq
To prove (\ref{consistency_double_first_un}), it is sufficient to show that for any $\eps > 0$, there exists $C_{\eps} > 0$ such that
\beqw
\begin{array}{llll}
\Pb(\|\tilde{\theta}-\theta_0\|_2 > C_{\eps} \nu_T) & \leq & \Pb(\exists \uu \in \Rb^{d_T}, \|\uu\|_2 \geq C_{\eps}: \Gb_T l(\theta_0 + \nu_T \uu) \leq \Gb_T l(\theta_0)) \\
& = & \Pb(\exists \uu \in \Rb^{d_T}, \|\uu\|_2 = C_{\eps}: \Gb_T l(\theta_0 + \nu_T \uu) \leq \Gb_T l(\theta_0)),
\end{array}
\eeqw
by convexity. By a Taylor expansion of $\Gb_T l(\theta_0 + \nu_T \uu)$, we obtain
\beqw
\Gb_T l(\theta_0 + \nu_T \uu) = \Gb_T l(\theta_0) + \nu_T \dot{\Gb}_T l(\theta_0)  \uu  + \cfrac{\nu^2_T}{2} \uu' \ddot{\Gb}_T l(\theta_0) \uu + \cfrac{\nu^3_T}{6} \nabla' \{ \uu' \ddot{\Gb}_T l(\bar{\theta}) \uu \} \uu,
\eeqw
where $\bar{\theta} \in \Theta$ such that $\|\bar{\theta}-\theta_0\|_2 \leq C_{\eps} \nu_T$. We would like to prove
\beq \label{LocalMinDM}
\Pb(\exists \uu \in \Rb^{d_T}, \|\uu\|_2 = C_{\eps}: \nu_T \dot{\Gb}_T l(\theta_0)  \uu  + \cfrac{\nu^2_T}{2} \uu' \ddot{\Gb}_T l(\theta_0) \uu
+ \cfrac{\nu^3_T}{6} \nabla' \{ \uu' \ddot{\Gb}_T l(\bar{\theta}) \uu \} \uu \leq 0 ) < \eps.
\eeq
To do so, we focus on each quantity of the Taylor expansion to extract the dominant term. First, for $a > 0$ and the Markov inequality, we have for the score term
\beqw
\begin{array}{llll}
\Pb(\underset{\uu: \|\uu\|_2 =C_{\eps}}{\sup} |\dot{\Gb}_T l(\theta_0) \uu| > a) & \leq & \Pb(\underset{\uu: \|\uu\|_2 =C_{\eps}}{\sup} \|\dot{\Gb}_T l(\theta_0)\|_2 \|\uu\|_2 > a) \\
& \leq & \Pb(\|\dot{\Gb}_T l(\theta_0)\|_2 > \frac{a}{C_{\eps}}) \\
& \leq & (\frac{C_{\eps}}{a})^2 \Eb[\|\dot{\Gb}_T l(\theta_0)\|^2_2]\\
& \leq & (\frac{C_{\eps}}{a})^2 \overset{d_T}{\underset{k=1}{\sum}} \Eb[(\partial_{\theta_k} \Gb_T l(\theta_0))^2] \\
& = &  (\frac{C_{\eps}}{a})^2 \frac{1}{T^2} \overset{T}{\underset{t,t'=1}{\sum}} \overset{d_T}{\underset{k=1}{\sum}} \Eb[\partial_{\theta_k} l(\eps_t;\theta_0) \partial_{\theta_k} l(\eps_{t'};\theta_0)] \\
& \leq & (\frac{C_{\eps}}{a})^2 \{\frac{1}{T^2} \overset{T}{\underset{t,t'=1}{\sum}} \Psi(|t-t'|) \} .d_T.
\end{array}
\eeqw
By assumption \ref{H8}, $\underset{k=1,\cdots,d_T}{\sup} \Eb[\partial_{\theta_k} l(\eps_t;\theta_0) \partial_{\theta_k} l(\eps_{t'};\theta_0)] \leq \Psi(|t-t'|)$ and $\frac{1}{T} \overset{T}{\underset{t,t'=1}{\sum}} \Psi(|t-t'|) < \infty$. This implies
\beqw
\Pb(\underset{\uu: \|\uu\|_2 =C_{\eps}}{\sup} |\dot{\Gb}_T l(\theta_0) \uu| > a) \leq \cfrac{C^2_{\eps} d_T}{T a^2} K_1,
\eeqw
for some constant $K_1 > 0$.

\mds

We now focus on the hessian quantity that can be rewritten as
\beqw
\uu' \ddot{\Gb}_T l(\theta_0) \uu = \uu' \Eb[\ddot{\Gb}_T l(\theta_0)] \uu + \Rc_T(\theta_0),
\eeqw
where $\Rc_T(\theta_0) = \overset{d_T}{\underset{k,l=1}{\sum}} \uu_k \uu_l \{\partial^2_{\theta_k \theta_l} \Gb_T l(\theta_0) - \Eb[\partial^2_{\theta_k \theta_l} \Gb_T l(\theta_0)]\}$. We have
\beqw
\Eb[\Rc_T(\theta_0)] = 0, \; \text{Var}(\Rc_T(\theta_0)) = \frac{1}{T^2} \overset{T}{\underset{t,t'=1}{\sum}} \overset{d_T}{\underset{k,k',l,l'=1}{\sum}} \uu_k \uu_{k'} \uu_l \uu_{l'} \Eb[ \zeta_{kl,t}.\zeta_{k'l',t'} ],
\eeqw
where $\zeta_{kl,t} = \partial^2_{\theta_k \theta_l} l(\eps_t;\theta_0) - \Eb[\partial^2_{\theta_k \theta_l} l(\eps_t;\theta_0)]$. Let $b>0$, we deduce by the Markov inequality and assumption \ref{H9},
\beqw
\Pb(|\Rc_T(\theta_0)| > b) \leq \cfrac{1}{b^2}\Eb[ \Rc^2_T(\theta_0)] \leq \cfrac{K_2}{b^2} \cfrac{\|\uu\|^4_2 d^2_T}{T} \leq \cfrac{K_2 C^4_{\eps} d^2_T}{b^2 T},
\eeqw
where $K_2 > 0$. Furthermore, by assumption \ref{H6},
\beqw
\uu' \Eb[\ddot{\Gb}_T l(\theta_0)] \uu \geq \lambda_{\min}(\Hb_T) \uu' \uu.
\eeqw

\mds

As for the third order term, we have
\beqw
\begin{array}{llll}
|\nabla \{\uu' \ddot{\Gb}_T l(\bar{\theta}) \uu \} \uu|^2 & \leq & \cfrac{1}{T^2} \overset{T}{\underset{t,t'=1}{\sum}} \underset{k_1,k_2,k_3}{\sum} \underset{l_1,l_2,l_3}{\sum}| u_{k_1} u_{k_2} u_{k_3} u_{l_1} u_{l_2} u_{l_3}| |\partial^3_{\theta_{k_1} \theta_{k_2} \theta_{k_3}} l(\eps_t;\bar{\theta}) . \partial^3_{\theta_{l_1} \theta_{l_2} \theta_{l_3}} l(\eps_{t'};\bar{\theta})| \\
& \leq & \|\uu\|^6_2 d^3_T \cfrac{1}{T^2} \overset{T}{\underset{t,t'=1}{\sum}} \upsilon_t(C_{\eps}) \upsilon_{t'}(C_{\eps}),
\end{array}
\eeqw
where
\beqw
\upsilon_t(C_0) = \underset{k_1 k_2 k_3}{\sup} \{\underset{\theta:\|\theta - \theta_0\|_2 \leq \nu_T C_0}{\sup} |\partial^3_{\theta_{k_1} \theta_{k_2} \theta_{k_3}} l(\eps_t;\theta)|\}.
\eeqw
Note that $\upsilon_t(C_0)$ depends on $d_T$ and $C_0$. By assumption \ref{H10}, we have
\beqw
\eta(C_0):= \cfrac{1}{T^2} \overset{T}{\underset{t,t'=1}{\sum}} \Eb[\upsilon_t(C_0) \upsilon_{t'}(C_0)] < \infty.
\eeqw
By the Markov inequality, for $c>0$, we conclude that
\beqw
\Pb(\exists \uu, \|\uu\|_2 = C_{\eps}: \cfrac{\nu^2_T}{6} \underset{\|\bar{\theta}-\theta_0\|_2 \leq \nu_T C_{\eps}}{\sup} |\nabla \{\uu' \ddot{\Gb}_T l(\bar{\theta}) \uu \} \uu|>c) \leq \cfrac{\nu^4_T d^3_T C^6_{\eps}}{36 c^2} \eta(C_{\eps}).
\eeqw

\mds

We can now bound (\ref{LocalMinDM}) thanks to proper choices of $a,b,c$ and $C_{\eps}$. We denote by $\delta_T = \lambda_{\min}(\Hb_T) C^2_{\eps} \nu_T$, and using $\cfrac{\nu_T}{2} \Eb[\uu' \ddot{\Gb}_T l(\theta_0) \uu ]\geq \delta_T$, we have
\beqw
\begin{array}{llll}
\Pb(\exists \uu \in \Rb^{d_T}, \|\uu\|_2 = C_{\eps}: \dot{\Gb}_T l(\theta_0)  \uu  + \cfrac{\nu_T}{2} \uu' \ddot{\Gb}_T l(\theta_0) \uu
+ \cfrac{\nu^2_T}{6} \nabla \{ \uu' \ddot{\Gb}_T l(\bar{\theta}) \uu \} \uu \leq 0) &&\\
\leq  \Pb(\exists \uu \in \Rb^{d_T}, \|\uu\|_2 = C_{\eps}: |\dot{\Gb}_T l(\theta_0) \uu| > \delta_T/4)  + \Pb(\exists \uu \in \Rb^{d_T}, \|\uu\|_2 = C_{\eps}:\frac{\nu_T}{2}|\Rc_T(\theta_0)| > \delta_T/4) && \\
+ \Pb(\exists \uu \in \Rb^{d_T}, \|\uu\|_2 = C_{\eps}: \frac{\nu^2_T}{6}\underset{\bar{\theta}:\|\bar{\theta}-\theta_0\|_2 < \nu_T C_{\eps}}{\sup}|\nabla \{\uu' \ddot{\Gb}_T l(\bar{\theta}) \uu \} \uu| > \delta_T/4) &&\\
\leq \cfrac{16 C^2_{\eps} d_T K_1}{T \delta^2_T} + \cfrac{4 \nu^2_T d^2_T C^4_{\eps}}{T \delta^2_T}+ \cfrac{16 \nu^4_T d^3_T C^6_{\eps}}{36 \delta^2_T} \eta(C_{\eps}) &&\\
\leq C_1 \cfrac{d_T}{T C^2_{\eps} \nu^2_T} + C_2 \cfrac{d^2_T}{T}+ C_3 \nu^2_T d^3_T C^2_{\eps} \eta(C_{\eps}),
\end{array}
\eeqw
where $C_1,C_2,C_3$ are strictly positive constants. We chose $\nu_T = (\cfrac{d_T}{T})^{\frac{1}{2}}$, we then deduce
\beqw
\begin{array}{llll}
\Pb(\exists \uu \in \Rb^{d_T}, \|\uu\|_2 = C_{\eps}:\nu_T \dot{\Gb}_T l(\theta_0)  \uu  + \cfrac{\nu^2_T}{2} \uu' \ddot{\Gb}_T l(\theta_0) \uu
+ \cfrac{\nu^3_T}{6} \nabla \{ \uu' \ddot{\Gb}_T l(\bar{\theta}) \uu \} \uu \leq 0) && \\
\leq \cfrac{C_1}{C^2_{\eps}} + C_2 \cfrac{d^2_T}{T} + \cfrac{d^4_T C^2_{\eps}}{T} \eta(C_{\eps}).
\end{array}
\eeqw
Now we fix $C_{\eps}$ sufficiently large enough, such that $C_1/C^2_{\eps}<\eps/3$. Once this constant is fixed, there exists a $T_0$ such that for $T>T_0$ we have $C_2 \frac{d^2_T}{T} < \eps/3$ and $C_3 \frac{d^4_T C^2_{\eps}}{T} \eta(C_{\eps}) < \eps/3$ under the assumption that $d^4_T = o(T)$.
Consequently, we obtain
\beqw
\Pb(\exists \uu \in \Rb^{d_T}, \|\uu\|_2 = C_{\eps}: \Gb_T l(\theta_0) + \nu_T \dot{\Gb}_T l(\theta_0)  \uu  + \cfrac{\nu^2_T}{2} \uu' \ddot{\Gb}_T l(\theta_0) \uu
+ \cfrac{\nu^3_T}{6} \nabla \{ \uu' \ddot{\Gb}_T l(\bar{\theta}) \uu \} \uu \leq 0) < \eps.
\eeqw
This proves (\ref{consistency_double_first_un}), that is $\|\tilde{\theta}-\theta_0\|_2 = O_p((\cfrac{d_T}{T})^{\frac{1}{2}})$.
\end{proof}

The first step estimator used for the adaptive weights is $(T/d_T)^{1/2}$-consistent. However, the estimated quantities on $\Ac^c$ converge to zero by consistency. We then propose a slight modification of the first step estimator, denoted $\tilde{\tilde{\theta}}$, which disappears asymptotically as follows
\beqw
\tilde{\tilde{\theta}} = \tilde{\theta} + e_T,
\eeqw
such that $e_T \rightarrow 0$ is a strictly positive quantity. We choose $e_T = T^{-\kappa}$ with $\kappa > 0$. This means we add in the adaptive weights a power of $T$  to the first step estimator, that is
\beqw
\alpha^{(k)}_{T,i} = |\tilde{\tilde{\theta}}^{(k)}_i| = |\tilde{\theta} + T^{-\kappa}|^{-\eta}, \; \xi_{T,l} = \|\tilde{\tilde{\theta}}^{(l)}\|_2 = \|\tilde{\theta}^{(l)} + T^{-\kappa}\|^{-\mu}_2.
\eeqw

\begin{theorem}\label{bound_prob_double}
Under assumptions \ref{H1}-\ref{H3}, \ref{H6}-\ref{H10}, if $d^4_T = o(T)$, and if $\cfrac{\gamma_T}{\sqrt{T}} T^{\frac{c}{2} + \kappa \mu} \underset{T \rightarrow \infty}{\longrightarrow} 0$, $\cfrac{\lambda_T}{\sqrt{T}} T^{\kappa \eta } \underset{T \rightarrow \infty}{\longrightarrow} 0$, then the sequence of penalized estimators $\hat{\theta}$ solving \ref{DA_SGL} satisfies
\beqw
\|\hat{\theta} - \theta_0\|_2 = O_p((\cfrac{d_T}{T})^{\frac{1}{2}}).
\eeqw
\end{theorem}

\begin{remark}
Note that $d^4_T = o(T)$ is as in Fan and Peng (2004), Theorem 1. The asymptotic behaviors of the regularization terms provide a condition such that the penalty terms are negligible with respect to $(d_T/T)^{1/2}$.
\end{remark}

\begin{proof}
We proceed as we did for proving Theorem \ref{bound_prob_double_init}. Let $\nu_T = (d_T/T)^{1/2} $. We would like to prove that for any $\eps > 0$, there exists $C_{\eps} > 0$ such that
\beq \label{consistency_double_first}
\Pb(\|\hat{\theta} - \theta_0\|_2/\nu_T > C_{\eps}) < \eps.
\eeq
To prove (\ref{consistency_double_first}), we show
\beq \label{LocalMinDM_ada}
\begin{array}{llll}
\Pb(\exists \uu \in \Rb^{d_T}, \|\uu\|_2=C_{\eps}: \dot{\Gb}_T l(\theta_0) \uu + \cfrac{\nu_T}{2} \uu' \ddot{\Gb}_T l(\theta_0) \uu + \cfrac{\nu^2_T}{6}\nabla'\{ \uu' \ddot{\Gb}_T l(\bar{\theta}) \uu\}\uu \\
+ \nu^{-1}_T\{\pp_1(\lambda_T,\tilde{\tilde{\theta}},\theta_0 + \nu_T \uu)-\pp_1(\lambda_T,\tilde{\tilde{\theta}},\theta_0) + \pp_2(\gamma_T,\tilde{\tilde{\theta}},\theta_0 + \nu_T \uu)-\pp_2(\gamma_T,\tilde{\tilde{\theta}},\theta_0\})\leq 0) < \eps. &&
\end{array}
\eeq
a relationship obtained by convexity and a Taylor expansion.
\mds

The score quantity can be upper bounded as
\beqw
|\dot{\Gb}_T l(\theta_0) \uu| \leq \|\dot{\Gb}_T l(\theta_0)\|_2 \|\uu\|_2 = O_p((\cfrac{d_T}{T})^{\frac{1}{2}}) \|\uu\|_2 = O_p(\nu_T) \|\uu\|_2,
\eeqw
where we used assumption \ref{H8} to obtain the bound in probability of the score.

\mds

As for the third order term, we have by the Cauchy-Schwartz inequality
\beqw
\begin{array}{llll}
|\nabla'\{ \uu' \ddot{\Gb}_T l(\bar{\theta}) \uu\}\uu |^2 & \leq & \|\uu\|^6_2 d^3_T \cfrac{1}{T^2} \overset{T^2}{\underset{t,t'=1}{\sum}}\{\overset{d_T}{\underset{k_1,l_1,m_1=1}{\sum}} \overset{d_T}{\underset{k_2,l_2,m_2=1}{\sum}}\partial^3_{\theta_{k_1} \theta_{l_1} \theta_{m_1}} l(\eps_t;\bar{\theta})  \partial^3_{\theta_{k_2} \theta_{l_2} \theta_{m_2}}l(\eps_{t'};\bar{\theta})\} \\
& = & \|\uu\|^6_2 d^3_T \eta(C_{\eps}).
\end{array}
\eeqw
This implies
\beqw
\nabla'\{ \uu' \ddot{\Gb}_T l(\bar{\theta}) \uu\}\uu =  O_p(d^{3/2}_T \|\uu\|^3_2).
\eeqw
Hence by the Markov inequality
\beqw
\Pb(\exists \uu \in \Rb^{d_T}, \|\uu\|_2 = C_{\eps}: |\nu^2_T \nabla'\{ \uu' \ddot{\Gb}_T l(\bar{\theta}) \uu\}\uu | > a) \leq \cfrac{\nu^4_T C^6_{\eps} d^3_T}{a^2} \eta(C_{\eps}).
\eeqw
where we used assumption \ref{H10}.

\mds

Finally, the hessian quantity can be treated as in the proof of Theorem \ref{bound_prob_double_init}. We denote by $\Rc_T(\theta_0) = \overset{d_T}{\underset{k,l=1}{\sum}} \uu_k \uu_l \{\partial^2_{\theta_k \theta_l} \Gb_T l(\theta_0) - \Eb[\partial^2_{\theta_k \theta_l} \Gb_T l(\theta_0)]\}$. We have
\beqw
\uu' \ddot{\Gb}_T l(\theta_0) \uu = \uu' \Eb[\ddot{\Gb}_T l(\theta_0)] \uu + \Rc_T(\theta_0).
\eeqw
By assumption \ref{H9} and the Markov inequality, for any $\kappa > 0$, we obtain
\beqw
\Pb(|\Rc_T(\theta_0)| > \kappa) \leq \cfrac{1}{\kappa^2}\Eb[ \Rc^2_T(\theta_0)] \leq \cfrac{K_2}{\kappa^2} \cfrac{\|\uu\|^4_2 d^2_T}{T} \leq \cfrac{K_2 C^4_{\eps} d^2_T}{\kappa^2 T},
\eeqw
with $K_2>0$. This relationship holds for any $\kappa > 0$. Then for $T$ large enough, we deduce that $|\Rc_T(\theta_0)| = o_p(1)$. Consequently
\beqw
\cfrac{\nu^2_T}{2}\uu' \ddot{\Gb}_T l(\theta_0) \uu \geq \cfrac{\nu^2_T}{2} \lambda_{\min}(\Hb_T)\|\uu\|^2_2 + o_p(1) \nu^2_T\|\uu\|^2_2.
\eeqw

\mds

We focus on the penalty terms. We have
\beqw
\begin{array}{llll}
\pp_1(\lambda_T,\tilde{\tilde{\theta}},\theta_0 + \nu_T \uu)-\pp_1(\lambda_T,\tilde{\tilde{\theta}},\theta_0) & = & \lambda_T T^{-1}\underset{k \in \Sc }{\sum} \underset{i \in \Ac_k}{\sum} \alpha^{(k)}_{T,i} \{|\theta^{(k)}_{0,i} + \nu_T \uu^{(k)}_i|  - |\theta^{(k)}_{0,i}| \},  \\
\text{and} \; |\pp_1(\lambda_T,\tilde{\tilde{\theta}},\theta_0 + \nu_T \uu)-\pp_1(\lambda_T,\tilde{\tilde{\theta}},\theta_0) | & \leq & \lambda_T T^{-1} \underset{k \in \Sc }{\sum} \underset{i \in \Ac_k}{\sum} \alpha^{(k)}_{T,i} \nu_T |\uu^{(k)}_i|.
\end{array}
\eeqw
As for the $l^1/l^2$ norm, we obtain
\beqw
\begin{array}{llll}
\pp_2(\gamma_T,\tilde{\tilde{\theta}},\theta_0 + \nu_T \uu)-\pp_2(\gamma_T,\tilde{\tilde{\theta}},\theta_0) & = &  \gamma_T T^{-1} \underset{l \in \Sc 1}{\sum} \xi_{T,l} \{ \|\theta^{(l)}_0 + \nu_T \uu\|_2 - \|\theta^{(l)}_0\|_2\} \\
\text{and} \; |\pp_2(\gamma_T,\tilde{\tilde{\theta}},\theta_0 + \nu_T \uu)-\pp_2(\gamma_T,\tilde{\tilde{\theta}},\theta_0)|& \leq & \gamma_T T^{-1} \underset{l \in \Sc}{\sum} \xi_{T,l} \nu_T \|\uu^{(l)}\|_2.
\end{array}
\eeqw
For the $l^1$ norm penalty, using $\{\underset{k \in \Sc , i \in \Ac_k}{\min} \; |\tilde{\tilde{\theta}}^{(k)}_i|\}^{-\eta} \leq T^{\kappa \eta}$, then
\beqw
\begin{array}{llll}
\lambda_T T^{-1} \underset{k \in \Sc }{\sum} \underset{i \in \Ac_k}{\sum} \alpha^{(k)}_{T,i} \nu_T |\uu^{(k)}_i| & \leq & \lambda_T T^{-1}\nu_T \{\underset{k \in \Sc }{\sum} \underset{i \in \Ac_k}{\sum} |\tilde{\tilde{\theta}}^{(k)}_i|^{-2\eta}\}^{1/2}  \|\uu\|_2 \\
& \leq & \lambda_T T^{-1} \nu_T  \cfrac{\sqrt{d_T}}{\{\underset{k \in \Sc , i \in \Ac_k}{\min} \; |\tilde{\tilde{\theta}}^{(k)}_i|\}^{\eta}} \|\uu\|_2 \\
& \leq & \lambda_T T^{-1} \nu_T  \sqrt{d_T} T^{\kappa \eta} \|\uu\|_2,
\end{array}
\eeqw
by the Cauchy-Schwartz inequality. Then if $\lambda_T T^{\frac{c}{2}-1+\kappa \eta} $ is bounded, we obtain
\beqw
\pp_1(\lambda_T,\tilde{\tilde{\theta}},\theta_0 + \nu_T \uu)-\pp_1(\lambda_T,\tilde{\tilde{\theta}},\theta_0) = O(\nu^2_T)\|\uu\|_2.
\eeqw
As for the $l^1/l^2$ term, using $\{\underset{l \in \Sc}{\min}\; \|\tilde{\tilde{\theta}}^{(l)}\|_2\}^{-\mu} \leq T^{\kappa \mu}$, we obtain
\beqw
\begin{array}{llll}
\gamma_T T^{-1}  \overset{m}{\underset{l=1}{\sum}} \xi_{T,l} \nu_T \|\uu^{(l)}\|_2 & \leq & \gamma_T T^{-1} \nu_T \{\underset{l \in \Sc}{\sum} \|\tilde{\tilde{\theta}}^{(l)}\|^{-2\mu}_2 \}^{1/2} \|\uu\|_2 \\
& \leq & \gamma_T T^{-1} \nu_T \cfrac{\sqrt{d_T}}{\{\underset{l \in \Sc}{\min}\; \|\tilde{\tilde{\theta}}^{(l)}\|_2\}^{\mu}} \|\uu\|_2 \\
& \leq & \gamma_T T^{-1} \nu_T \sqrt{d_T} T^{\kappa \mu} \|\uu\|_2,
\end{array}
\eeqw
by the Cauchy-Schwartz inequality. Then if $\gamma_T T^{\frac{c}{2}-1+\kappa \mu}$ is bounded, we obtain
\beqw
\pp_2(\gamma_T,\tilde{\tilde{\theta}},\theta_0 + \nu_T \uu)-\pp_2(\gamma_T,\tilde{\tilde{\theta}},\theta_0) = O(\nu^2_T)\|\uu\|_2.
\eeqw
\mds
We now can prove (\ref{LocalMinDM_ada}). Let $\delta_T = \lambda_{\min}(\Hb_T) C^2_{\eps} \nu_T$ and using $\frac{\nu_T}{2}\Eb[\uu' \ddot{\Gb}_T l(\theta_0)\uu] \geq \delta_T$, we have
\beqw
\begin{array}{llll}
\Pb(\exists \uu \in \Rb^{d_T}, \|\uu\|_2 = C_{\eps}: \dot{\Gb}_T l(\theta_0)  \uu  + \nu_T \uu' \ddot{\Gb}_T l(\theta_0) \uu/2
+ \nu^2_T \nabla \{ \uu' \ddot{\Gb}_T l(\bar{\theta}) \uu \} \uu/6 && \\
+ \nu^{-1}_T\{\pp_1(\lambda_T,\tilde{\tilde{\theta}},\theta_0 + \nu_T \uu)-\pp_1(\lambda_T,\tilde{\tilde{\theta}},\theta_0) + \pp_2(\gamma_T,\tilde{\tilde{\theta}},\theta_0 + \nu_T \uu)-\pp_2(\gamma_T,\tilde{\tilde{\theta}},\theta_0)\} \leq 0) &&\\

\leq \Pb(\exists \uu \in \Rb^{d_T}, \|\uu\|_2 = C_{\eps}: |\nu_T \uu' \ddot{\Gb}_T l(\theta_0) \uu/2| \leq |\dot{\Gb}_T l(\theta_0)  \uu| + |\nu^2_T \nabla \{ \uu' \ddot{\Gb}_T l(\bar{\theta}) \uu \} \uu/6 | && \\
+ \nu^{-1}_T \{| \pp_1(\lambda_T,\tilde{\tilde{\theta}},\theta_0 )-\pp_1(\lambda_T,\tilde{\tilde{\theta}},\theta_0+ \nu_T \uu)| + |\pp_2(\gamma_T,\tilde{\tilde{\theta}},\theta_0)-\pp_2(\gamma_T,\tilde{\tilde{\theta}},\theta_0+\nu_T \uu)|\} ) &&\\

\leq \Pb(\exists \uu \in \Rb^{d_T}, \|\uu\|_2 = C_{\eps}: |\dot{\Gb}_T l(\theta_0) \uu| > \delta_T/8) + \Pb(\exists \uu \in \Rb^{d_T}, \|\uu\|_2 = C_{\eps}: \frac{\nu_T}{2} |\Rc_T(\theta_0)| > \delta_T/8) &&\\
+ \Pb(\exists \uu \in \Rb^{d_T}, \|\uu\|_2 = C_{\eps}: |\frac{\nu^2_T}{6}\nabla \{\uu' \ddot{\Gb}_T l(\bar{\theta}) \uu \} \uu| > \delta_T/8) && \\
+ \Pb(\exists \uu \in \Rb^{d_T}, \|\uu\|_2 = C_{\eps}: |\pp_1(\lambda_T,\tilde{\tilde{\theta}},\theta_0)-\pp_1(\lambda_T,\tilde{\tilde{\theta}},\theta_0 + \nu_T \uu) | > \nu_T \delta_T/8) && \\
+ \Pb((\exists \uu \in \Rb^{d_T}, \|\uu\|_2 = C_{\eps}:|\pp_2(\gamma_T,\tilde{\tilde{\theta}},\theta_0)-\pp_2(\gamma_T,\tilde{\tilde{\theta}},\theta_0+ \nu_T \uu)| > \nu_T \delta_T/8) && \\
\leq \cfrac{C_{st}}{C^2_{\eps}} + C_{st} \{\nu^2_T C^6_{\eps} d^3_T \eta(C_{\eps})\} + \cfrac{C_{st} \nu^2_T d^2_T C^4_{\eps}}{T \delta^2_T} + \eps/5 + \eps/5 && \\
< \eps, &&
\end{array}
\eeqw
with $C_{st}>0$ a generic constant. We used $d^4_T = o(T)$ and for $C_{\eps}$ large enough
\beqw
\begin{array}{llll}
\Pb(\exists \uu \in \Rb^{d_T}, \|\uu\|_2 = C_{\eps}: |\pp_1(\lambda_T,\tilde{\tilde{\theta}},\theta_0)-\pp_1(\lambda_T,\tilde{\tilde{\theta}},\theta_0 + \nu_T \uu) | > \nu_T \delta_T/8) & \leq & \eps/5,\\
\Pb(\exists \uu \in \Rb^{d_T}, \|\uu\|_2 = C_{\eps}:|\pp_2(\gamma_T,\tilde{\tilde{\theta}},\theta_0)-\pp_2(\gamma_T,\tilde{\tilde{\theta}},\theta_0+ \nu_T \uu)| > \nu_T \delta_T/8) & \leq & \eps/5.
\end{array}
\eeqw
Thus we obtain for $C_{\eps}$ and $T$ large enough, with the conditions $\gamma_T T^{\frac{c}{2}-1+\kappa \mu} \rightarrow 0$ and $\lambda_T T^{\frac{c}{2}-1+\kappa \eta} \rightarrow 0$ that
\beqw
\|\hat{\theta} - \theta_0\|_2 = O_p(\nu_T) = O_p((\cfrac{d_T}{T})^{\frac{1}{2}}).
\eeqw
\end{proof}
To achieve the oracle property, we need some additional assumptions regarding the adaptive penalty components.
\begin{assumption} \label{H12}
For any $T$, there exists $\beta$ such that $0 < \beta < \underset{i \in \Ac_k}{\min} \, \theta_{0,i,\Ac_k}, k \in \Sc$. Moreover,
\beqw
\beta^{-1}T^{-1}\{\lambda_T d^{1/2}_T \Eb[\underset{k \in \Sc, i \in \Ac_k}{\max} \; \alpha_{T,\Ac_k,i} ] + \gamma_T \Eb[ \underset{k \in \Sc}{\max} \; \xi_{T,k} ] \} \underset{T \rightarrow \infty}{\longrightarrow} 0.
\eeqw
\end{assumption}

\begin{assumption}\label{H13}
The model complexity is assumed to behave as $d^5_T = o(T)$, which implies that $0 < c < \frac{1}{5}$. The regularization parameters are chosen such that they satisfy
\beqw
\begin{array}{llll}
\cfrac{\gamma_T}{\sqrt{T}} T^{\frac{c}{2} + \kappa \mu} \underset{T \rightarrow \infty}{\longrightarrow} 0, && \cfrac{\gamma_T}{\sqrt{T}} T^{\frac{1}{2}[(1+\mu)(1-c) -1]} \underset{T \rightarrow \infty}{\longrightarrow} \infty, \\
\cfrac{\lambda_T}{\sqrt{T}} T^{\kappa \eta } \underset{T \rightarrow \infty}{\longrightarrow} 0, && \cfrac{\lambda_T}{\sqrt{T}} T^{\frac{1}{2}[(1+\eta)(1-c) -1]} \underset{T \rightarrow \infty}{\longrightarrow} \infty,\\
\cfrac{\gamma_T}{\lambda^{1+\mu}_T} T^{(1+\mu)(1-\frac{c}{2}-\kappa \eta)-1} \underset{T \rightarrow \infty}{\longrightarrow} \infty. &&
\end{array}
\eeqw
\end{assumption}

\bre
The main condition is $d^5_T = o(T)$, which is the same as Fan and Peng (2004). This condition comes from the control for the third order derivative of the empirical criterion. Note that simple cases allow for a framework where $0 \leq c < 1$. Moreover, these asymptotic behaviors are closely related to condition (A5) of Zou and Zhang (2009). In section \ref{Sim}, we provide further details about the calibration of the adaptive weights and $\kappa$.
\ere

\begin{assumption}\label{AN}
Let $\Fc^T_{t} = \sigma(X_{T,s},s \leq t)$ with $X_{T,t} = \sqrt{T} Q_T \Vb^{-1/2}_{T,\Ac \Ac} \Hb^{-1}_{T,\Ac \Ac} \dot{\Gb}_T l_t(\theta_0)_{\Ac}$, $(Q_T)$ is a sequence of $r \times \text{card}(\Ac)$ matrices such that $Q_T \times Q'_T \overset{\Pb}{\rightarrow} \Cb$, for some $r \times r$ nonnegative symmetric matrix $\Cb$, $\Vb_{T,\Ac \Ac} = (\Hb^{-1}_T \Mb_T \Hb^{-1}_T)_{\Ac \Ac}$ and $\dot{\Gb}_T l_t(\theta_0)_{\Ac} = \frac{1}{T} \nabla_{\Ac} l(\eps_t;\theta_0)$. Then $X_{T,t}$ is a martingale difference and we have
\beqw
\Eb[  \underset{i,j = 1,\cdots,d_T}{\sup} \Eb[\{\partial_{\theta_i} l(\eps_t;\theta_0) \partial_{\theta_j} l(\eps_t;\theta_0)\}^2 | \Fc^T_{t-1} ] \lambda_{\max,t-1}(\Hb^T_{t-1}) ] \leq \bar{B}  < \infty,
\eeqw
with
\beqw
\Hb^T_{t-1} := \Eb[\nabla l(\eps_t;\theta_0) \nabla' l(\eps_t;\theta_0) | \Fc^T_{t-1}] \leq \lambda_{\max}(\Hb^T_{t-1}) < \infty.
\eeqw
\end{assumption}

\begin{theorem} \label{oracle_double}
Under assumptions \ref{H1}-\ref{H3}, and assumptions \ref{H6}-\ref{AN}, the sequence of adaptive estimator $\hat{\theta}$ solving (\ref{DA_SGL}) satisfies
\beqw
\begin{array}{llll}
\underset{T \rightarrow \infty}{\lim} \Pb(\hat{\Ac} = \Ac) = 1, \; \text{and}& & \\
\sqrt{T} Q_T \Vb^{-1/2}_{T,\Ac \Ac} (\hat{\theta}_{\Ac} - \theta_{0,\Ac}) \overset{d}{\rightarrow} \Nc(0,\Cb), & &
\end{array}
\eeqw
where $(Q_T)$ is a sequence of $r \times \text{card}(\Ac)$ matrices such that $Q_T \times Q'_T \overset{\Pb}{\rightarrow} \Cb$, for some $r \times r$ nonnegative symmetric matrix $\Cb$ and $\Vb_{T,\Ac \Ac} = (\Hb^{-1}_T \Mb_T \Hb^{-1}_T)_{\Ac \Ac}$.
\end{theorem}

\begin{proof}

Model selection consistency consists of proving that the probability of the event $\{\hat{\Ac} = \Ac\}$ tends to one asymptotically. This event is
\beqw
\{\hat{\Ac} = \Ac\} = \{\forall k \in \Sc, \forall i \in \Ac_k, |\hat{\theta}^{(k)}_i| > 0 \} \cap \{\forall k = 1,\cdots,m, \forall i \in \Ac^c_k, \hat{\theta}^{(k)}_i = 0\}.
\eeqw
Hence we prove
\beq \label{model_select_double}
\Pb(\{\forall k \in \Sc, \forall i \in \Ac_k, |\hat{\theta}^{(k)}_i| > 0 \} \cap \{\forall k = 1,\cdots,m, \forall i \in \Ac^c_k, \hat{\theta}^{(k)}_i = 0\}) \underset{T \rightarrow \infty} {\longrightarrow} 1.
\eeq
Model selection consistency can be decomposed into two parts: recovering the active indices by estimating nonzero coefficients; discarding the inactive indices by shrinking to zero the related coefficients. Now (\ref{model_select_double}) can be proved by first showing that for any $T$, there exists $\beta$ such that $0 < \beta < \underset{i \in \Ac_k}{\min} \, \theta_{0,i,\Ac_k}$, with $k \in \Sc$ and
\beq \label{model_select_indiv}
\Pb(\|\hat{\theta}_{\Ac} - \theta_{0,\Ac}\|_2 < \beta) \underset{T \rightarrow \infty}{\longrightarrow} 1.
\eeq
The second part regarding nonactive indices can be proved as
\beq \label{model_select_zero}
\left\{\begin{array}{llll}
\Pb(\underset{k \in \Sc^c}{\cap}\{\|\hat{\zz}^{(k)}\|_2 < 1 \}) \underset{T \rightarrow \infty}{\longrightarrow} 1, & & \\
\Pb(\underset{k \in \Sc}{\cap}\underset{i \in \Ac^c_k}{\cap}\{ |\hat{\ww}^{(k)}_i| < 1\}) \underset{T \rightarrow \infty}{\longrightarrow} 1, & &
\end{array} \right.
\eeq
where $\hat{\zz}^{(k)}$ (resp. $\hat{\ww}^{(k)}$) is the subgradient of $\|\hat{\theta}^{(k)}\|_2$ (resp. $\|\hat{\theta}^{(k)}\|_1$) given in (\ref{opt1}). Hence (\ref{model_select_indiv}) and (\ref{model_select_zero}) prove (\ref{model_select_double}).

\mds

We first focus on (\ref{model_select_indiv}), which is equivalent to
\beqw
\Pb(\|\hat{\theta}_{\Ac_k} - \theta_{0,\Ac_k}\|_2 > \beta) \underset{T \rightarrow \infty}{\longrightarrow} 0.
\eeqw
By the Karush-Kuhn-Tucker optimality conditions, we have
\beqw
\dot{\Gb}_T l(\hat{\theta})_{\Ac} + \lambda_T T^{-1} \alpha_{T,\Ac} \odot \text{sgn}(\hat{\theta}_{\Ac}) + \gamma_T T^{-1} \varsigma_T = 0,
\eeqw
where $\varsigma_T = \text{vec}(\xi_{T,k} \cfrac{\hat{\theta}_{\Ac_k}}{\|\hat{\theta}_{\Ac_k}\|_2}, k \in \Sc)$. We denote by $\alpha_{T,\Ac_k} = (\alpha_{T,i},i \in \Ac_k)$, a vector of size $\Rb^{\cc_{\Ac_k}}$. By a Taylor expansion of the gradient component around $\theta_{0,\Ac}$, we have
\beqw
\begin{array}{llll}
\dot{\Gb}_T l(\theta_0)_{\Ac} + \Hb_{T,\Ac \Ac} (\hat{\theta}_{\Ac} - \theta_{0,\Ac})  +  \Pc_{T}(\theta_0) (\hat{\theta}_{\Ac} - \theta_{0,\Ac}) + \cfrac{1}{2} \nabla'_{\Ac} \{(\hat{\theta}_{\Ac} - \theta_{0,\Ac})'\ddot{\Gb}_T l(\bar{\theta}) (\hat{\theta}_{\Ac} - \theta_{0,\Ac})\} & & \\
+ \lambda_T T^{-1} \alpha_{T,\Ac} \odot \text{sgn}(\hat{\theta}_{\Ac}) + \gamma_T T^{-1} \varsigma_T = 0 & & \\

\Leftrightarrow \hat{\theta}_{\Ac} = \theta_{0,\Ac} - \Hb^{-1}_{T,\Ac \Ac} (\dot{\Gb}_T l(\theta_0)_{\Ac} + \lambda_T T^{-1} \alpha_{T,\Ac} \odot \text{sgn}(\hat{\theta}_{\Ac}) + \gamma_T T^{-1} \varsigma_T&& \\
- \Hb^{-1}_{T,\Ac \Ac} \cfrac{1}{2} \nabla'_{\Ac} \{ (\hat{\theta}_{\Ac} - \theta_{0,\Ac})' \ddot{\Gb}_T l(\bar{\theta})_{\Ac \Ac} (\hat{\theta}_{\Ac} - \theta_{0,\Ac})\} - \Hb^{-1}_{T,\Ac \Ac} \Pc_{T}(\theta_0) (\hat{\theta}_{\Ac} - \theta_{0,\Ac}), &&
\end{array}
\eeqw
where $\|\bar{\theta}-\theta_0\|_2 \leq \|\hat{\theta}-\theta_0\|_2$, $\Pc_{T}(\theta_0) = \ddot{\Gb}_T l(\theta_0)_{\Ac \Ac} - \Hb_{T,\Ac \Ac}$ and $\Hb_{T,\Ac \Ac} = \Eb[\nabla^2_{\theta \theta'} l(\eps_t;\theta_0)]_{\Ac \Ac}$. Then using $\|\hat{\theta}_{\Ac} - \theta_{0,\Ac}\|_2 = O_p((\cfrac{d_T}{T})^{\frac{1}{2}})$, we obtain
\beqw
\begin{array}{llll}
\Pb(\|\hat{\theta}_{\Ac} - \theta_{0,\Ac}\|_2 > \beta) & \leq & \Pb( \| \Hb^{-1}_{T,\Ac \Ac} \dot{\Gb}_T l(\theta_0)_{\Ac}  \|_2 + \| \Hb^{-1}_{T,\Ac \Ac} \|_2 \| \lambda_T T^{-1} \alpha_{T,\Ac} \odot \text{sgn}(\hat{\theta}_{\Ac}) \|_2 \\
& + & \| \Hb^{-1}_{T,\Ac \Ac} \|_2 \|\gamma_T T^{-1} \varsigma_T \|_2 \\
& + & \|\Hb^{-1}_{T,\Ac \Ac} \|_2 \| \nabla'_{\Ac} \{(\hat{\theta}_{\Ac} - \theta_{0,\Ac})' \ddot{\Gb}_T l(\bar{\theta})_{\Ac \Ac}(\hat{\theta}_{\Ac} - \theta_{0,\Ac}) \}/2\|_2 \\
& + & \| \Hb^{-1}_{T,\Ac \Ac} \|_2 \|\Pc_{T}(\theta_0) (\hat{\theta}_{\Ac} - \theta_{0,\Ac}) \|_2 > \beta) \\
& \leq & \Pb( \lambda^{-1}_{\min}(\Hb_T) \| \dot{\Gb}_T l(\theta_0)_{\Ac}  \|_2 + \lambda^{-1}_{\min}(\Hb_T)\lambda_T T^{-1} \|\alpha_{T,\Ac} \|_2 \\
& + & \lambda^{-1}_{\min}(\Hb_T) \gamma_T T^{-1} \|\varsigma_T\|_2 + \lambda^{-1}_{\min}(\Hb_T) C^2_0 (d_T/2T) \|\nabla'_{\Ac} \{\ddot{\Gb}_T l(\bar{\theta})_{\Ac \Ac}\}\|_2 \\
& + & \lambda^{-1}_{\min}(\Hb_T) C_0 (d_T/T)^{1/2} \|\Pc_{T}(\theta_0)\|_2 > \beta) + \Pb(\|\hat{\theta}_{\Ac}-\theta_{0,\Ac}\|_2 > (d_T/T)^{1/2} C_0),
\end{array}
\eeqw
for $C_0 > 0$ large enough, and we used $\|\Hb^{-1}_T \xx \|_2 \leq \lambda^{-1}_{\min}(\Hb_T) \|\xx\|_2$ for any vector $\xx \in \Rb^{d_T}$.

\mds

Let us proceed element-by-element. We have by the Markov inequality
\beqw
\Pb(\lambda^{-1}_{\min}(\Hb_T) C_0 \sqrt{\frac{d_T}{T}} \|\Pc_{T}(\theta_0)\|_2 > \frac{\beta}{6}) \leq \cfrac{36 \lambda^{-2}_{\min}(\Hb_T) C^2_0 d_T}{T\beta^2}\Eb[\|\Pc_{T}(\theta_0)\|^2_2].
\eeqw
We have
\beqw
\Eb[\|\Pc_{T}\|^2_2] = \cfrac{1}{T^2}\overset{T}{\underset{t,t'=1}{\sum}} \underset{k,k' \in \Ac}{\sum} \underset{l,l'\in \Ac}{\sum} \Eb[\zeta_{kl,t} \zeta_{k'l',t'}],
\eeqw
where $\zeta_{kl,t} = \partial^2_{\theta_k \theta_l} l(\eps_t;\theta_0) - \Eb[\partial^2_{\theta_k \theta_l} l(\eps_t;\theta_0)]$. By assumption \ref{H9}, we obtain
\beqw
\Pb(\lambda^{-1}_{\min}(\Hb_T) C_0 \sqrt{\frac{d_T}{T}} \|\Pc_{T}(\theta_0)\|_2 > \frac{\beta}{6}) \leq \cfrac{36 \lambda^{-2}_{\min}(\Hb_T) C^2_0 }{\beta^2} \cfrac{d^3_T}{T^2}.
\eeqw

\mds

As for the third order term, by the Markov inequality
\beqw
\Pb(\frac{1}{2}\lambda^{-1}_{\min}(\Hb_T) C^2_0 \frac{d_T}{T} \|\nabla'_{\Ac} \{\ddot{\Gb}_T l(\bar{\theta})_{\Ac \Ac}\}\|_2 > \frac{\beta}{6}) \leq \cfrac{9 \lambda^{-2}_{\min}(\Hb_T) C^4_0 d^2_T}{T^2} \Eb[\|\nabla'_{\Ac} \{\ddot{\Gb}_T l(\bar{\theta})_{\Ac \Ac}\}\|^2_2].
\eeqw
We obtain
\beqw
\begin{array}{llll}
\Eb[\|\nabla'_{\Ac} \{\ddot{\Gb}_T l(\bar{\theta})_{\Ac \Ac}\}\|^2_2] & \leq & \cfrac{1}{T^2} \overset{T}{\underset{t,t'=1}{\sum}} \underset{k_1,k_2,k_3 \in \Ac}{\sum} \underset{l_1,l_2,l_3 \in \Ac}{\sum} \Eb[ |\partial^3_{\theta_{k_1} \theta_{k_2} \theta_{k_3}} l(\eps_t;\theta_0).\partial^3_{\theta_{l_1} \theta_{l_2} \theta_{l_3}} l(\eps_t;\theta_0) l(\eps_{t'};\theta_0) |] \\
& \leq & \cfrac{1}{T^2} d^3_T \overset{T}{\underset{t,t'=1}{\sum}} \Eb[\upsilon_t(C_0) \upsilon_{t'}(C_0)] = \eta(C_0) d^3_T,
\end{array}
\eeqw
by assumption \ref{H10}, where $\upsilon_t(C_0) = \underset{k_1 k_2 k_3}{\sup} \underset{\theta:\|\theta-\theta_0\|_2 \leq \sqrt{\frac{d_T}{T}} C_0}{\sup} |\partial^3_{\theta_{k_1} \theta_{k_2} \theta_{k_3}} l(\eps_t;\theta_0)|$. We deduce that
\beqw
\Pb(\frac{1}{2}\lambda^{-1}_{\min}(\Hb_T) C^2_0 \frac{d_T}{T} \|\nabla'_{\Ac} \{\ddot{\Gb}_T l(\bar{\theta})_{\Ac \Ac}\}\|_2 > \frac{\beta}{6}) \leq \cfrac{9 \lambda^{-2}_{\min}(\Hb_T) C^4_0}{4}  \cfrac{d^5_T}{T^2} \eta(C_0).
\eeqw

\mds

We now turn to the score quantity. By the Markov inequality and assumption \ref{H8}, we have
\beqw
\begin{array}{llll}
\Pb( \lambda^{-1}_{\min}(\Hb_T) \| \dot{\Gb}_T l(\theta_0)_{\Ac}  \|_2 > \beta/6) & \leq & \cfrac{\lambda^{-2}_{\min}(\Hb_T) 36}{\beta^2} \Eb[\| \dot{\Gb}_T l(\theta_0)_{\Ac}  \|_2] \\
& \leq & \cfrac{\lambda^{-2}_{\min}(\Hb_T) 36}{\beta^2} \cfrac{1}{T^2} \overset{T}{\underset{t,t'=1}{\sum}} \underset{k \in \Ac}{\sum} \Eb[\partial_{\theta_k} l(\eps_t;\theta_0) \partial_{\theta_k} l(\eps_{t'};\theta_0)] \\
& \leq & \cfrac{\lambda^{-2}_{\min}(\Hb_T) 36}{\beta^2} \cfrac{1}{T} \{ \cfrac{1}{T} \overset{T}{\underset{t,t'=1}{\sum}} \Psi(|t-t'|) \} d_T \\
& \leq & \cfrac{\lambda^{-2}_{\min}(\Hb_T) 36 K d_T}{T\beta^2},
\end{array}
\eeqw
with $K > 0$. Hence we deduce
\beqw
\begin{array}{llll}
\Pb(\|\hat{\theta}_{\Ac} - \theta_{0,\Ac}\|_2 > \beta) & \leq & \Pb(\lambda^{-1}_{\min}(\Hb_T)\lambda_T T^{-1} \|\alpha_{T,\Ac} \|_2 + \lambda^{-1}_{\min}(\Hb_T) \gamma_T T^{-1} \|\varsigma_T\|_2 > \beta/2) \\
& + &  \Pb(\|\hat{\theta}_{\Ac}-\theta_{0,\Ac}\|_2 > (d_T/T)^{1/2} C_0) + \Pb(\lambda^{-1}_{\min}(\Hb_T) C_0 (d_T/T)^{1/2} \|\Pc_{T}(\theta_0)\|_2 > \beta/6) \\
& + & \Pb(\lambda^{-1}_{\min}(\Hb_T) C^2_0 (d_T/2T) \|\nabla'_{\Ac} \{\ddot{\Gb}_T l(\bar{\theta})_{\Ac \Ac}\}\|_2 > \beta/6) \\
& + & \Pb( \lambda^{-1}_{\min}(\Hb_T) \| \dot{\Gb}_T l(\theta_0)_{\Ac}  \|_2 > \beta/6) \\
& \leq &  \cfrac{2 \lambda^{-1}_{\min}(\Hb_T) }{\beta} \{ \lambda_T T^{-1} d^{1/2}_T \Eb[\underset{k \in \Sc, i \in \Ac_k}{\max} \; \alpha_{T,\Ac_k,i} ] + \gamma_T T^{-1} \Eb[ \underset{k \in \Sc}{\max} \; \xi_{T,k} ] \} \\
& + & \cfrac{36 \lambda^{-2}_{\min}(\Hb_T)  K d_T}{T\beta^2} +\cfrac{9 \lambda^{-2}_{\min}(\Hb_T) C^4_0}{4}  \cfrac{d^5_T}{T^2} \eta(C_0) + \cfrac{36 \lambda^{-2}_{\min}(\Hb_T) C^2_0 }{\beta^2} \cfrac{d^3_T}{T^2} + \eps.
\end{array}
\eeqw
For $T$ and $C_0$ large enough, if $d^5_T = o(T)$, by assumption \ref{H12}, that is if
\beqw
\beta^{-1}T^{-1}\{\lambda_T d^{1/2}_T \Eb[\underset{k \in \Sc, i \in \Ac_k}{\max} \; \alpha_{T,\Ac_k,i} ] + \gamma_T \Eb[ \underset{k \in \Sc}{\max} \; \xi_{T,k} ] \} \underset{T \rightarrow \infty}{\longrightarrow} 0,
\eeqw
 then
\beqw
\Pb(\|\hat{\theta}_{\Ac} - \theta_{0,\Ac}\|_2 > \beta) \underset{T \rightarrow \infty}{\longrightarrow} 0.
\eeqw

\mds

We  now turn to the second step of model selection consistency. First we prove
\beq \label{inactive_group_select}
\Pb(\underset{k \in \Sc^c}{\cap}\{ \|\hat{\zz}^{(k)}\|_2 < 1\}) \underset{T \rightarrow \infty}{\longrightarrow} 1 \Leftrightarrow \Pb(\underset{k \in \Sc^c}{\cup}\{ \|\hat{\zz}^{(k)}\|_2 \geq 1\}) \underset{T \rightarrow \infty}{\longrightarrow} 0.
\eeq
This is equivalent to proving
\beqw
\Pb(\underset{k \in \Sc^c}{\cup}\{ \|\dot{\Gb}_T l(\hat{\theta})_{(k)} + \lambda_T T^{-1} \alpha^{(k)}_T \odot \hat{\ww}^{(k)}\|_2 \geq \gamma_T T^{-1} \xi_{T,k}\}) \underset{T \rightarrow \infty}{\longrightarrow} 0.
\eeqw
We have for $k \in \Sc^c$ that $\|\hat{\ww}^{(k)}\|_{\infty} \leq 1$, which implies by the optimality conditions of Karush-Kuhn-Tucker that
\beqw
\begin{array}{llll}
\Pb(\underset{k \in \Sc^c}{\cup}\{ \|\dot{\Gb}_T l(\hat{\theta})_{(k)} + \lambda_T T^{-1} \alpha^{(k)}_T \odot \hat{\ww}^{(k)}\|_2 \geq \gamma_T T^{-1} \xi_{T,k}\}) && \\
\leq \Pb(\underset{k \in \Sc^c}{\cup}\{ \|\dot{\Gb}_T l(\hat{\theta})_{(k)}\|_2 \geq \gamma_T T^{-1} \xi_{T,k} - \lambda_T T^{-1} \|\alpha^{(k)}_T\|_2\}).&&
\end{array}
\eeqw
By a Taylor expansion around $\theta_0$, let $\bar{\theta}$ such that $\|\bar{\theta}-\theta_0\| \leq \|\hat{\theta} - \theta_0\|$, we have
\beqw
\begin{array}{llll}
\Pb(\underset{k \in \Sc^c}{\cup} \{ \|\hat{\zz}^{(k)}\|_2 \geq 1 \}) & \leq & \Pb(\underset{k \in \Sc^c}{\cup}\{ \|\dot{\Gb}_T l(\theta_0)_{(k)}\|_2 \geq \gamma_T T^{-1} \xi_{T,k} - \lambda_T T^{-1} \|\alpha^{(k)}_T\|_2 \\
& - & \|\ddot{\Gb}_T l(\theta_0)_{(k)(k)}\|_2 \|\hat{\theta} - \theta_0\|_2 - \|\nabla'\{\ddot{\Gb}_T l(\bar{\theta})_{(k)(k)}\}_{(k)}\|_2 \|\hat{\theta}-\theta_0\|^2_2\}) \\
& \leq & \Pb(\underset{k \in \Sc^c}{\cup}\{ \|\dot{\Gb}_T l(\theta_0)_{(k)}\|_2 \geq \gamma_T T^{-1} \|\tilde{\tilde{\theta}}^{(k)}\|^{-\mu}_2 - \lambda_T T^{-1} d^{1/2}_T \underset{k \in \Sc^c,i\in \Gc_k}{\max} \, (|\tilde{\tilde{\theta}}^{(k)}_i|^{-\eta})  \\
& - & \|\ddot{\Gb}_T l(\theta_0)_{(k)(k)}\|_2 \|\hat{\theta} - \theta_0\|_2 - \|\nabla'\{\ddot{\Gb}_T l(\bar{\theta})_{(k)(k)}\}_{(k)}\|_2 \|\hat{\theta}-\theta_0\|^2_2\}),
\end{array}
\eeqw
where we used $\|\ddot{\Gb}_T l(\theta_0)_{(k)(k)} (\hat{\theta} - \theta_0)\|_2 \leq \|\ddot{\Gb}_T l(\bar{\theta})_{(k)(k)}\|_2 \|\hat{\theta}- \theta_0\|_2$ and $\|\ddot{\Gb}_T l(\theta_0)_{(k)(k)}\|_2=\|\ddot{\Gb}_T l(\theta_0)_{(k)(k)}\|_s$. Let $\eps > 0$, and $K_{\eps}$ strictly positive constants, we proved for $T$ large enough that
\beqw
\Pb(\|\hat{\theta}-\theta_0\|_2>K_{\eps}(d_T/T)^{1/2} ) < \eps/6.
\eeqw
We deduce that
\beqw
\begin{array}{llll}
\Pb(\underset{k \in \Sc^c}{\cup}\{ \|\hat{\zz}^{(k)}\|_2 \geq 1\}) & \leq & \Pb(\underset{k \in \Sc^c}{\cup}\{ \|\dot{\Gb}_T l(\theta_0)_{(k)}\|_2 \geq \gamma_T T^{-1} \|\tilde{\tilde{\theta}}^{(k)}\|^{-\mu}_2 - \lambda_T T^{-1} d^{1/2}_T \underset{k \in \Sc^c,i\in \Gc_k}{\max} \, (|\tilde{\tilde{\theta}}^{(k)}_i|^{-\eta})  \\
& - & \|\ddot{\Gb}_T l(\theta_0)_{(k)(k)}\|_2 (d_T/T)^{1/2} K_{\eps} - \|\nabla'\{\ddot{\Gb}_T l(\bar{\theta})_{(k)(k)}\}_{(k)}\|_2 (\cfrac{d_T}{T})^2 K^2_{\eps} \}) + \eps/6.
\end{array}
\eeqw
Let $M_{1,T} = (\cfrac{\gamma_T}{T})^{\frac{1}{1+\mu}}$, then we obtain
\beqw
\begin{array}{llll}
\Pb(\underset{k \in \Sc^c}{\cup}\{ \|\hat{\zz}^{(k)}\|_2 \geq 1\}) & \leq & \underset{k \in \Sc^c}{\sum}\{\Pb( \|\dot{\Gb}_T l(\theta_0)_{(k)}\|_2 \geq \gamma_T T^{-1} \|\tilde{\tilde{\theta}}^{(k)}\|^{-\mu}_2 - \lambda_T T^{-1} d^{1/2}_T \underset{k \in \Sc^c,i\in \Gc_k}{\max} \, (|\tilde{\tilde{\theta}}^{(k)}_i|^{-\eta})  \\
& - & \|\ddot{\Gb}_T l(\theta_0)_{(k)(k)}\|_2 (d_T/T)^{\frac{1}{2}} K_{\eps} - \|\nabla'\{\ddot{\Gb}_T l(\bar{\theta})_{(k)(k)}\}_{(k)}\|_2 (\cfrac{d_T}{T})^2 K^2_{\eps}, \|\tilde{\tilde{\theta}}^{(k)}\|_2 \leq M_{1,T}) \\
& + & \Pb(\|\tilde{\tilde{\theta}}^{(k)}\|_2 > M_{1,T})\} + \eps/6.
\end{array}
\eeqw
Consequently, we have the relationship
\beqw
\begin{array}{llll}
\Pb(\underset{k \in \Sc^c}{\cup}\{ \|\hat{\zz}^{(k)}\|_2 \geq 1\})  & \leq & \underset{k \in \Sc^c}{\sum}\{ \Pb(\|\dot{\Gb}_T l(\theta_0)_{(k)}\|_2 \geq \gamma_T T^{-1} M^{-\mu}_{1,T}/4) \\
& + & \Pb(\lambda_T T^{-1} d^{1/2}_T \underset{k \in \Sc^c,i\in \Gc_k}{\max} \, (|\tilde{\tilde{\theta}}^{(k)}_i|^{-\eta}) > \gamma_T T^{-1} M^{-\mu}_{1,T}/4) \\
& + & \Pb(\|\ddot{\Gb}_T l(\theta_0)_{(k)(k)}\|_2 (d_T/T)^{1/2} K_{\eps} > \gamma_T T^{-1} M^{-\mu}_{1,T}/4 ) \\
& + & \Pb(\|\nabla'\{\ddot{\Gb}_T l(\bar{\theta})_{(k)(k)}\}_{(k)}\|_2 (\cfrac{d_T}{T})^2 K^2_{\eps} > \gamma_T T^{-1} M^{-\mu}_{1,T}/4) \\
& + & \Pb(\|\tilde{\tilde{\theta}}^{(k)}\|_2 > M_{1,T}) \} + \eps/6 := \overset{5}{\underset{i=1}{\sum}} T_i + \eps/6.
\end{array}
\eeqw
We then focus on each $T_i$. We have by the Markov inequality
\beqw
\begin{array}{llll}
T_1 := \underset{k \in \Sc^c}{\sum} \Pb(\|\dot{\Gb}_T l(\theta_0)_{(k)}\|_2 > \gamma_T T^{-1} M^{-\mu}_{1,T}/4) & \leq & \underset{k \in \Sc^c}{\sum} \cfrac{16  \Eb[\|\dot{\Gb}_T l(\theta_0)_{(k)}\|^2_2]}{\{\gamma_T T^{-1}M^{-\mu}_{1,T}\}^2} \\
& \leq & \cfrac{16 \Eb[\|\dot{\Gb}_T l(\theta_0)\|^2_2]}{\{\gamma_T T^{-1}M^{-\mu}_{1,T}\}^2} \\
& \leq & \cfrac{16 d_T}{T \{\gamma_T T^{-1}M^{-\mu}_{1,T}\}^2} \\
& = & O((\cfrac{\gamma_T}{\sqrt{T}} T^{\frac{1}{2}[(1+\mu)(1-c)-1]})^{-\frac{2}{1+\mu}}).
\end{array}
\eeqw

Furthermore, using $|\tilde{\tilde{\theta}}^{(k)}_i|^{-\eta} \leq T^{\kappa \eta}$, we have for $T_2$ that
\beq \label{T_2}
\begin{array}{llll}
\Pb(\lambda_T T^{-1} d^{1/2}_T \underset{k \in \Sc^c,i\in \Gc_k}{\max} \, (|\tilde{\tilde{\theta}}^{(k)}_i|^{-\eta}) > \gamma_T T^{-1} M^{-\mu}_{1,T}/4) & \leq & \Pb(\lambda_T T^{-1} d^{1/2}_T T^{\kappa \eta} > \gamma_T T^{-1} M^{-\mu}_{1,T}/4) \\
& \leq & \Pb(\gamma_T T^{-1} M^{-\mu}_{1,T}/4 \{ 1 - 4 \lambda_T \gamma^{-1}_T d^{1/2}_T M^{\mu}_{1,T} T^{\kappa \eta} \} < 0).
\end{array}
\eeq
The quantity of interest is $\gamma_T \lambda^{-1}_T  d^{-1/2}_T M^{-\mu}_{1,T} T^{-\kappa \eta}$ that has to converge to $\infty$ such that (\ref{T_2}) converge to zero for $T$ sufficiently large enough. We have
\beqw
\gamma_T \lambda^{-1}_T  d^{-1/2}_T M^{-\mu}_{1,T} T^{-\kappa \eta} \rightarrow \infty \Leftrightarrow \cfrac{\gamma_T}{\lambda^{1+\mu}_T} d^{-\frac{1+\mu}{2}}_T T^{-\kappa \eta (1+\mu) + \mu} \rightarrow \infty.
\eeqw

As for $T_3$, we have by the Markov inequality
\beqw
\begin{array}{llll}
T_3 :=\underset{k \in \Sc^c}{\sum} \Pb((\|\Hb_{T,(k)(k)}\|_2 + \|\Rc_{T,(k)}(\theta_0)\|_2) (d_T/T)^{1/2} K_{\eps} > \gamma_T T^{-1} M^{-\mu}_{1,T}/4 ) && \\
\leq \underset{k \in \Sc^c}{\sum} \Pb((\|\Rc_{T,(k)}(\theta_0)\|_2 (d_T/T)^{1/2} K_{\eps} > \gamma_T T^{-1} M^{-\mu}_{1,T}/4 - \|\Hb_{T,(k),(k)}\|_2 (d_T/T)^{1/2} K_{\eps}) && \\

\leq \underset{k \in \Sc^c}{\sum} \{\Pb((\|\Rc_{T,(k)}(\theta_0)\|_2 (d_T/T)^{1/2} K_{\eps} > \gamma_T T^{-1} M^{-\mu}_{1,T}/8) + \Pb(\|\Hb_{T,(k),(k)}\|_2 (d_T/T)^{1/2} K_{\eps} > \gamma_T T^{-1} M^{-\mu}_{1,T}/8) \}&& \\
\leq  \underset{k \in \Sc^c}{\sum} \{ \cfrac{64 K^2_{\eps} d_T \Eb[\|\Rc_{T,(k)}(\theta_0)\|^2_2]}{T \gamma^2_T T^{-2} M^{-2\mu}_{1,T} } + \cfrac{64 K^2_{\eps} d_T \|\Hb_{T,(k)(k)}\|^2_2 }{T \gamma^2_T T^{-2} M^{-2\mu}_{1,T} } \} && \\
\leq \cfrac{64 K^2_{\eps}d_T\|\Hb_T\|^2_2}{\gamma^2_T T^{-1} M^{-2\mu}_{1,T}} + \cfrac{64 K^2_{\eps} \Eb[\|\Rc_T(\theta_0)\|^2_2]}{\gamma^2_T M^{-2\mu}_{1,T}}&& \\
\leq \cfrac{64 K^2_{\eps}d_T \lambda^2_{\max}(\Hb_T)}{\gamma^2_T T^{-1} M^{-2\mu}_{1,T}} + \cfrac{64 K^2_{\eps} d^3_T }{\gamma^2_T M^{-2\mu}_{1,T}} && \\
\leq \cfrac{64 K^2_{\eps}\lambda^2_{\max}(\Hb_T)}{\{\gamma_T T^{-1/2} d^{-1/2}_T M^{-\mu}_{1,T}\}^2} + \cfrac{64 K^2_{\eps}}{\{\gamma_T d^{-3/2}_T M^{-\mu}_{1,T}\}^2} && \\
= O((\cfrac{\gamma_T}{\sqrt{T}} T^{\frac{1}{2}[(1+\mu)(1-c)-1]})^{-\frac{2}{1+\mu}}) + O((\cfrac{\gamma_T}{\sqrt{T}} T^{\frac{1}{2}[(1+\mu)(2-3c)-1]})^{-\frac{2}{1+\mu}}).&&
\end{array}
\eeqw

We obtain for $T_4$ by the Markov inequality
\beqw
\begin{array}{llll}
T_4:= \underset{k \in \Sc^c}{\sum} \Pb(\|\nabla'\{\ddot{\Gb}_T l(\bar{\theta})_{(k)(k)}\}_{(k)}\|_2 (\cfrac{d_T}{T})^2 K^2_{\eps} > \gamma_T T^{-1} M^{-\mu}_{1,T}/4) && \\
\leq \underset{k \in \Sc^c}{\sum} \cfrac{16 K^4_{\eps} d^2_T \Eb[\|\nabla'\{\ddot{\Gb}_T l(\bar{\theta})_{(k)(k)}\}_{(k)}\|^2_2]}{T^2 \gamma_T T^{-2} M^{-2\mu}_{1,T}} && \\
\leq \cfrac{16 K^4_{\eps} d^5_T \Eb[\|\nabla'\{\ddot{\Gb}_T l(\bar{\theta})\}\|^2_2]}{\gamma^2_T M^{-2\mu}_{1,T}} && \\
\leq \cfrac{16 K^4_{\eps} d^5_T \eta(K_{\eps})}{\gamma^2_T  M^{-2\mu}_{1,T}} = \cfrac{16 K^4_{\eps} \eta(K_{\eps}) }{\{\gamma_T d^{-5/2}_T M^{-\mu}_{1,T}\}^2} = O((\cfrac{\gamma_T}{\sqrt{T}} T^{\frac{1}{2}[(1+\mu)(2-5c)-1]})^{-\frac{2}{1+\mu}}). &&
\end{array}
\eeqw
Finally, we have for $T_5$ that
\beqw
\begin{array}{llll}
T_5:=\underset{k \in \Sc^c}{\sum} \Pb(\|\tilde{\tilde{\theta}}^{(k)}\|_2 > M_{1,T}) & \leq & \underset{k \in \Sc^c}{\sum} \cfrac{\Eb[\|\tilde{\tilde{\theta}}^{(k)}\|^2_2]}{M^2_{1,T}} \\
& \leq & \cfrac{\Eb[\|\tilde{\tilde{\theta}} - \theta_0\|^2_2]}{M^2_{1,T}} \\
& = & O((\cfrac{\gamma_T}{\sqrt{T}} T^{\frac{1}{2}[(1+\mu)(1-c) -1]})^{-\frac{2}{1+\mu}}).
\end{array}
\eeqw
Hence we obtain from these relationships and using assumption \ref{H13}
\beqw
\begin{array}{llll}
\cfrac{\gamma_T}{\lambda^{1+\mu}_T} T^{\mu-(\frac{c}{2}+\kappa \eta)(1+\mu)} \underset{T \rightarrow \infty}{\longrightarrow} \infty, && \\
\cfrac{\gamma_T}{\sqrt{T}} T^{\frac{1}{2}[(1+\mu)(1-c) -1]} \underset{T \rightarrow \infty}{\longrightarrow} \infty,&&\\
\end{array}
\eeqw
such that the latter implies
\beqw
\begin{array}{llll}
\cfrac{\gamma_T}{\sqrt{T}} T^{\frac{1}{2}[(1+\mu)(2-3c) -1]} \underset{T \rightarrow \infty}{\longrightarrow} \infty,&&\\
\cfrac{\gamma_T}{\sqrt{T}} T^{\frac{1}{2}[(1+\mu)(2-5c)-1]} \underset{T \rightarrow \infty}{\longrightarrow} \infty.&&\\
\end{array}
\eeqw
Consequently each $T_i$ converges to zero for $T$ large enough. Hence
\beqw
\Pb(\underset{k \in \Sc^c}{\cup} \{\|\hat{\zz}^{(k)}\|_2 \geq 1\} ) \leq \overset{5}{\underset{i=1}{\sum}} T_i + \eps/6 \underset{T \rightarrow \infty}{\longrightarrow} \eps.
\eeqw
For $\eps \rightarrow 0$, we prove $\Pb(\underset{k \in \Sc^c}{\cup} \{\|\hat{\zz}^{(k)}\|_2 \geq 1\} ) \rightarrow 0$ for $T$ large enough.

\mds

As for the second part of the model selection procedure, we prove that
\beq \label{inactive_within}
\Pb(\underset{k \in \Sc}{\cap} \underset{i \in \Ac^c_k}{\cap} \{|\hat{\ww}^{(k)}_i| < 1 \}) \underset{T \rightarrow \infty}{\longrightarrow} 1 \Leftrightarrow  \Pb(\underset{k \in \Sc}{\cup} \underset{i \in \Ac^c_k}{\cup}\{ |\hat{\ww}^{(k)}_i| \geq 1\}) \underset{T \rightarrow \infty}{\longrightarrow} 0.
\eeq
By the optimality conditions, we have
\beqw
\Pb(\underset{k \in \Sc}{\cup} \underset{i \in \Ac^c_k}{\cup}\{|\hat{\ww}^{(k)}_i| \geq 1\}) = \Pb(\underset{k \in \Sc}{\cup} \underset{i \in \Ac^c_k}{\cup}\{|\dot{\Gb}_T l(\hat{\theta})_{(k),i}| \geq \lambda_T T^{-1} \alpha^{(k)}_{T,i}\}).
\eeqw
Then by a Taylor expansion around $\theta_0$, with $\bar{\theta}$ between $\hat{\theta}$ and $\theta_0$, we have
\beqw
\begin{array}{llll}
 \Pb(\underset{k \in \Sc}{\cup} \underset{i \in \Ac^c_k}{\cup}\{|\hat{\ww}^{(k)}_i| \geq 1\}) & = & \Pb(\underset{k \in \Sc}{\cup} \underset{i \in \Ac^c_k}{\cup}\{ |\dot{\Gb}_T l(\theta_0)_{(k),i} + [\underset{j}{\sum} \partial^2_{ij} \Gb_T l(\theta_0) (\hat{\theta}_j - \theta_{0,j})]_i \\
& + & [\underset{j,k}{\sum} T^{-1}\overset{T}{\underset{t=1}{\sum}}\partial^3_{ijk} l(\eps_t;\bar{\theta})(\hat{\theta}_j - \theta_{0,j})^2/2]_i| \geq \lambda_T T^{-1} \alpha^{(k)}_{T,i}\}) \\
& \leq & \Pb(\underset{k \in \Sc}{\cup} \underset{i \in \Ac^c_k}{\cup}\{|\dot{\Gb}_T l(\theta_0)_{(k),i} | \geq \lambda_T T^{-1} \alpha^{(k)}_{T,i} - [\underset{j}{\sum} \partial^2_{ij} \Gb_T l(\theta_0) (\hat{\theta}_j - \theta_{0,j})]_i \\
& - & [\underset{j,k}{\sum} T^{-1}\overset{T}{\underset{t=1}{\sum}}\partial^3_{ijk} l(\eps_t;\bar{\theta})(\hat{\theta}_j - \theta_{0,j})^2/2]_i|\}).
\end{array}
\eeqw
Let $M_{2,T} = (\cfrac{\lambda_T}{T})^{\frac{1}{1+\eta}}$. Then using  $\|\hat{\theta}-\theta_0\|_2 = O_p((\cfrac{d_T}{T})^{\frac{1}{2}})$ and the Cauchy-Schwartz inequality, we obtain
\beqw
\begin{array}{llll}
\Pb(\underset{k \in \Sc}{\cup} \underset{i \in \Ac^c_k}{\cup}\{ |\hat{\ww}^{(k)}_i| \geq 1\}) & \leq & \underset{k \in \Sc}{\sum}\underset{i \in \Ac^c_k}{\sum}\{\Pb(|\dot{\Gb}_T l(\theta_0)_{(k),i} | \geq \lambda_T T^{-1} \alpha^{(k)}_{T,i} - [\underset{j}{\sum} \partial^2_{ij} \Gb_T l(\theta_0) (\hat{\theta}_j - \theta_{0,j})]_i \\
& - & [\underset{j,k}{\sum} T^{-1}\overset{T}{\underset{t=1}{\sum}}\partial^3_{ijk} l(\eps_t;\bar{\theta})(\hat{\theta}_j - \theta_{0,j})^2/2]_i|,|\tilde{\tilde{\theta}}^{(k)}_i| \leq M_{2,T}) + \Pb(|\tilde{\tilde{\theta}}^{(k)}_i| > M_{2,T})\} \\
& \leq & \underset{k \in \Sc}{\sum}\underset{i \in \Ac^c_k}{\sum}\{ \Pb(|\dot{\Gb}_T l(\theta_0)_{(k),i} | \geq \lambda_T T^{-1} M^{-\eta}_{2,T}- \{\underset{j}{\sum} (\partial^2_{ij} \Gb_T l(\theta_0))^2 \}^{1/2} K_{\eps} (d_T/T)^{1/2} \\
& - & \{\underset{j,k,l,m}{\sum} T^{-2}\overset{T}{\underset{t,t'=1}{\sum}}\partial^3_{ijk} l(\eps_t;\bar{\theta}) \partial^3_{ilm} l(\eps_{t'};\bar{\theta}) \}^{1/2}K^2_{\eps} (d_T/T)) \\
& + & \Pb(|\tilde{\tilde{\theta}}^{(k)}_i| > M_{2,T}) \}+ \eps/5 \\
& \leq & \underset{k \in \Sc}{\sum} \underset{i \in \Ac^c_k}{\sum}\{ \Pb(|\dot{\Gb}_T l(\theta_0)_{(k),i} | \geq \lambda_T T^{-1} M^{-\eta}_{2,T}/3)\\
& + & \Pb(\{\underset{j}{\sum} (\partial^2_{ij} \Gb_T l(\theta_0))^2 \}^{1/2} K_{\eps} (d_T/T)^{1/2} > \lambda_T T^{-1} M^{-\eta}_{2,T}/3) \\
& + & \Pb(\{\underset{j,k,l,m}{\sum} T^{-2}\overset{T}{\underset{t,t'=1}{\sum}}\partial^3_{ijk} l(\eps_t;\bar{\theta}) \partial^3_{ilm} l(\eps_{t'};\bar{\theta}) \}^{1/2}K^2_{\eps} (d_T/T) > \lambda_T T^{-1} M^{-\eta}_{2,T} /3) \\
& + & \Pb(|\tilde{\tilde{\theta}}^{(k)}_i| > M_{2,T}) \} + \eps/5 := \overset{4}{\underset{i=1}{\sum}} T_i + \eps/5.
\end{array}
\eeqw
We proceed as for inactive groups. For $T_1$, we have by the Markov inequality
\beqw
\begin{array}{llll}
T_1:=\underset{k \in \Sc}{\sum} \underset{i \in \Ac^c_k}{\sum} \Pb(|\dot{\Gb}_T l(\theta_0)_{(k),i} | \geq \lambda_T T^{-1} M^{-\eta}_{2,T}/3) & \leq & \underset{k \in \Sc}{\sum} \underset{i \in \Ac^c_k}{\sum} \cfrac{9 \Eb[|\dot{\Gb}_T l(\theta_0)_{(k),i} |^2]}{\{\lambda_T T^{-1} M^{-\eta}_{2,T}\}^2} \\
& \leq & \cfrac{9 \Eb[\|\dot{\Gb}_T l(\theta_0) \|^2_2]}{\{\lambda_T T^{-1} M^{-\eta}_{2,T}\}^2} \\
& = & O((\cfrac{\lambda_T}{\sqrt{T}} T^{\frac{1}{2}[(1+\eta)(1-c)-1]})^{-\frac{2}{1+\eta}}).
\end{array}
\eeqw
As for $T_2$, we have
\beqw
\begin{array}{llll}
T_2:=\underset{k \in \Sc}{\sum} \underset{i \in \Ac^c_k}{\sum} \Pb(\{\underset{j}{\sum} (\partial^2_{ij} \Gb_T l(\theta_0))^2 \}^{1/2} K_{\eps} (d_T/T)^{1/2} > \lambda_T T^{-1} M^{-\eta}_{2,T}/3) &&\\
=\underset{k \in \Sc}{\sum} \underset{i \in \Ac^c_k}{\sum} \Pb((\{\underset{j}{\sum} \Pc_{T,(k),j}(\theta_0) \}^{1/2} +  \{\underset{j}{\sum} \Hb^2_{T,(k),j} \}^{1/2})K_{\eps} (d_T/T)^{1/2} > \lambda_T T^{-1} M^{-\eta}_{2,T}/3) &&\\
\leq \underset{k \in \Sc}{\sum} \underset{i \in \Ac^c_k}{\sum} \underset{j}{\sum}\{ \cfrac{36 d_T \Eb[ \Pc^2_{T,(k),j}(\theta_0) ]}{T\{\lambda_T T^{-1} M^{-\eta}_{2,T}\}^2}\} + \cfrac{36 d_T \|\Hb_T\|^2_2}{T\{\lambda_T T^{-1} M^{-\eta}_{2,T}\}^2}&&\\
\leq\cfrac{36 d_T\lambda^2_{\max}(\Hb_T)}{T\{\lambda_T T^{-1} M^{-\eta}_{2,T}\}^2} + \cfrac{36 d_T \Eb[\|\Pc_T(\theta_0)\|^2_2]}{T\{\lambda_T T^{-1} M^{-\eta}_{2,T}\}^2} && \\
= O((\cfrac{\lambda_T}{\sqrt{T}} T^{\frac{1}{2}[(1+\eta)(1-c)-1]})^{-\frac{2}{1+\eta}}) + O((\cfrac{\lambda_T}{\sqrt{T}} T^{\frac{1}{2}[(1+\eta)(2-3c)-1]})^{-\frac{2}{1+\eta}}).&&
\end{array}
\eeqw
Furthermore, for the third order term in $T_3$, we have
\beqw
\begin{array}{llll}
T_3:=\underset{k \in \Sc}{\sum} \underset{i \in \Ac^c_k}{\sum} \Pb(\{\underset{j,k,l,m}{\sum} T^{-2}\overset{T}{\underset{t,t'=1}{\sum}}\partial^3_{ijk} l(\eps_t;\bar{\theta}) \partial^3_{ilm} l(\eps_{t'};\bar{\theta}) \}^{1/2}K^2_{\eps} (d_T/T) > \lambda_T T^{-1} M^{-\eta}_{2,T} /3) & & \\
\leq \cfrac{9 d^2_T\Eb[\|\nabla'\{\ddot{\Gb}_T l(\bar{\theta}\}\|^2_2]}{T^2\{\lambda_T T^{-1} M^{-\eta}_{2,T}\}^2} = O((\cfrac{\lambda_T}{\sqrt{T}} T^{\frac{1}{2}[(1+\eta)(2-5c)-1]})^{-\frac{2}{1+\eta}}). &&
\end{array}
\eeqw
Finally, we have for $T_4$ that
\beqw
\begin{array}{llll}
T_4:=\underset{i \in \Ac^c_k}{\sum} \Pb(|\tilde{\tilde{\theta}}^{(k)}_i| > M_{2,T}) & \leq & \underset{k \in \Sc}{\sum} \underset{i \in \Ac^c_k}{\sum} \cfrac{\Eb[|\tilde{\tilde{\theta}}^{(k)}_i|^2]}{M^2_{2,T}} \\
& \leq &  \cfrac{\Eb[\|\tilde{\tilde{\theta}} - \theta_0\|^2_2]}{M^2_{2,T}} = O((\cfrac{\lambda_T}{\sqrt{T}} T^{\frac{1}{2}[(1+\eta)(1-c) -1]})^{-\frac{2}{1+\eta}}).
\end{array}
\eeqw
We have from these relationships and by assumption \ref{H13}, $\cfrac{\lambda_T}{\sqrt{T}} T^{\frac{1}{2}[(1+\eta)(1-c) -1]} \underset{T \rightarrow \infty}{\longrightarrow} \infty$ implies
\beqw
\begin{array}{llll}
\cfrac{\lambda_T}{\sqrt{T}} T^{\frac{1}{2}[(1+\eta)(2-3c)-1]} \underset{T \rightarrow \infty}{\longrightarrow} \infty, &&\\
\cfrac{\lambda_T}{\sqrt{T}} T^{\frac{1}{2}[(1+\eta)(2-5c)-1]} \underset{T \rightarrow \infty}{\longrightarrow} \infty. &&\\
\end{array}
\eeqw
We deduce
\beqw
\Pb(\underset{k \in \Sc}{\cup} \underset{i \in \Ac^c_k}{\cup}\{ |\hat{\ww}^{(k)}_i| \geq 1\}) \underset{T \rightarrow \infty}{\longrightarrow} \eps,
\eeqw
for $T$ sufficiently large enough. We have then concluded the model selection consistency.

\mds

We now focus on the asymptotic normality. Model selection implies that
\beqw
\Pb(\{k \in \Sc, i \in \Ac_k, : \hat{\theta}^{(k)}_i \neq 0\} = \Ac) \underset{T \rightarrow \infty}{\longrightarrow} 1.
\eeqw
As a consequence, the next relationship holds
\beqw
\Pb(\forall k \in \Sc, \dot{\Gb}_T l(\hat{\theta})_{\Ac_k} + \lambda_T T^{-1} \alpha_{T,\Ac_k} \odot \text{sgn}(\hat{\theta}_{\Ac_k}) + \gamma_T T^{-1} \xi_{T,k} \cfrac{\hat{\theta}_{\Ac_k}}{\|\hat{\theta}_{\Ac_k}\|_2} = 0) \underset{T \rightarrow \infty}{\longrightarrow} 1.
\eeqw
By a Taylor expansion of the gradient term around $\theta_{0,\Ac}$, we obtain
\beqw
\begin{array}{llll}
\Pb(\dot{\Gb}_T l(\theta_0)_{\Ac} +  \ddot{\Gb}_T l(\theta_0)_{\Ac \Ac} (\hat{\theta}_{\Ac} - \theta_{0,\Ac}) & + & \cfrac{1}{2} \nabla' \{ (\hat{\theta}_{\Ac} - \theta_{0,\Ac})'  \ddot{\Gb}_T l(\bar{\theta})_{\Ac \Ac} (\hat{\theta}_{\Ac} - \theta_{0,\Ac})\} \\
& + & \lambda_T T^{-1} \alpha_{T,\Ac} \odot \text{sgn}(\hat{\theta}_{\Ac}) + \gamma_T T^{-1} \eta_T = 0) \underset{T \rightarrow \infty}{\longrightarrow} 1,
\end{array}
\eeqw
where $\eta_T = \text{vec}(\xi_{T,k} \cfrac{\hat{\theta}_{\Ac_k}}{\|\hat{\theta}_{\Ac_k}\|_2}, k \in \Sc)$ and $\|\bar{\theta} - \theta_0\|_2 \leq \|\hat{\theta}-\theta_0\|_2$. As a consequence, we have
\beqw
\begin{array}{llll}
\Pc(\theta_0) (\hat{\theta}_{\Ac} - \theta_{0,\Ac}) + \Hb_{T,\Ac \Ac} (\hat{\theta}_{\Ac} - \theta_{0,\Ac}) & = & - \dot{\Gb}_T l(\theta_0)_{\Ac} - \cfrac{1}{2} \nabla' \{(\hat{\theta}_{\Ac} - \theta_{0,\Ac}) \ddot{\Gb}_T l(\bar{\theta})_{\Ac \Ac} (\hat{\theta}_{\Ac} - \theta_{0,\Ac}) \}\\
& - & \lambda_T T^{-1} \alpha_{T,\Ac} \odot \text{sgn}(\hat{\theta}_{\Ac}) - \gamma_T T^{-1} \eta_T + o_p(1),
\end{array}
\eeqw
where $\Pc(\theta_0) = \ddot{\Gb}_T l(\theta_0)_{\Ac \Ac} - \Hb_{T,\Ac \Ac}$ and $\Hb_{T,\Ac \Ac} = \Eb[\nabla^2_{\theta \theta'} l(\eps_t;\theta_0)]_{\Ac \Ac}$. Then multiplying by $\sqrt{T} Q_T \Vb^{-1/2}_{T,\Ac \Ac}$, we obtain
\beqw
\begin{array}{llll}
\sqrt{T}Q_T \Vb^{-1/2}_{T,\Ac \Ac} (\hat{\theta}_{\Ac} - \theta_{0,\Ac}) & = & -\sqrt{T} Q_T \Vb^{-1/2}_{T,\Ac \Ac} \Hb^{-1}_{T,\Ac \Ac} (\lambda_T T^{-1} \alpha_{T,\Ac} \odot \text{sgn}(\hat{\theta}_{\Ac})
 + \gamma_T T^{-1} \eta_T) \\
& - & \sqrt{T} Q_T \Vb^{-1/2}_{T,\Ac \Ac} \Hb^{-1}_{T,\Ac \Ac} \dot{\Gb}_T l(\theta_0)_{\Ac}  \\
& - & \sqrt{T}/2 Q_T \Vb^{-1/2}_{T,\Ac \Ac} \Hb^{-1}_{T,\Ac \Ac}  \nabla' \{(\hat{\theta}_{\Ac} - \theta_{0,\Ac})' \ddot{\Gb}_T l(\bar{\theta})_{\Ac \Ac} (\hat{\theta}_{\Ac} - \theta_{0,\Ac}) \}  \\
& - & \sqrt{T} Q_T \Vb^{-1/2}_{T,\Ac \Ac} \Hb^{-1}_{T,\Ac \Ac}  \Pc(\theta_0) (\hat{\theta}_{\Ac} - \theta_{0,\Ac}) + o_p(1).
\end{array}
\eeqw
We focus on the $l^1$ penalty term, which can be upper bounded as
\beqw
\begin{array}{llll}
N_{1,T} := |\sqrt{T} Q_T \Vb^{-1/2}_{T,\Ac \Ac} \Hb^{-1}_{T,\Ac \Ac} (\cfrac{\lambda_T}{T} \alpha_{T,\Ac} \odot \text{sgn}(\hat{\theta}_{\Ac}))| & \leq & |Q_T \Vb^{-1/2}_{T,\Ac \Ac} | |\Hb^{-1}_{T,\Ac \Ac} | \lambda_T T^{-1/2} \underset{k \in \Sc, i \in \Ac_k}{\max} \, \alpha_{T,i,\Ac} \\
& \leq & |Q_T \Vb^{-1/2}_{T,\Ac \Ac}| \lambda^{-1}_{\min}(\Hb_{T,\Ac \Ac}) \lambda_T T^{-1/2} \{\underset{k \in \Sc, i \in \Ac_k}{\min} |\tilde{\tilde{\theta}}^{(k)}_i|\}^{-\eta} \\
& \leq & |Q_T \Vb^{-1/2}_{T,\Ac \Ac}| \lambda^{-1}_{\min}(\Hb_{T,\Ac \Ac}) \lambda_T T^{\kappa \eta - \frac{1}{2}}.
\end{array}
\eeqw
If $\lambda_T T^{\kappa \eta} \rightarrow 0$, then $N_{1,T} = o_p(1)$.

\mds

As for the $l^1/l^2$ penalty, it can be upper bounded as
\beqw
\begin{array}{llll}
N_{2,T} := |\sqrt{T} Q_T \Vb^{-1/2}_{T,\Ac \Ac} \Hb^{-1}_{T,\Ac \Ac} \cfrac{\gamma_T}{T} \eta_T| & \leq & |Q_T \Vb^{-1/2}_{T,\Ac \Ac}||\Hb^{-1}_{T,\Ac \Ac}| \gamma_T T^{-1/2} \|\eta_T\|_2 \\
& \leq & |Q_T \Vb^{-1/2}_{T,\Ac \Ac}||\Hb^{-1}_{T,\Ac \Ac}| \gamma_T T^{-1/2} \sqrt{\underset{k \in \Sc}{\sum} \|\tilde{\tilde{\theta}}^{(k)}\|^{-2\mu}_2} \\
& \leq & |Q_T \Vb^{-1/2}_{T,\Ac \Ac}| \lambda^{-1}_{\min}(\Hb_{T,\Ac \Ac}) \gamma_T T^{-1/2} d^{1/2}_T \{\underset{k \in \Sc}{\min}  \|\tilde{\tilde{\theta}}^{(k)}\|_2\}^{-\mu} \\
& \leq & |Q_T \Vb^{-1/2}_{T,\Ac \Ac}| \lambda^{-1}_{\min}(\Hb_{T,\Ac \Ac}) \gamma_T T^{-1/2} d^{1/2}_T T^{\kappa \mu}.
\end{array}
\eeqw
Using $d_T = O(T^c)$, if $\gamma_T T^{\frac{c-1}{2}+\kappa \mu} \rightarrow 0$, then $N_{2,T} = o_p(1)$. Consequently, we have $N_{1,T}+N_{2,T} = o_p(1)$.

\mds

We now turn to the hessian quantity of the Taylor expansion and prove the discrepancy $\Pc(\theta_0)$ converges uniformly to zero in probability. For any $\eps > 0$, by the Markov's inequality, we have
\beqw
\begin{array}{llll}
\Pb(\|\ddot{\Gb}_T l(\theta_0)_{\Ac \Ac} - \Hb_{T,\Ac \Ac}\|^2_2 > (\eps/d_T)^2) & \leq & \cfrac{d^2_T}{\eps^2 T^2} \Eb[ \underset{(k,l)\in \Ac}{\sum} \{\partial^2_{\theta_k \theta_l} l(\eps_t;\theta_0) - \Eb[\nabla^2_{\theta_k \theta_l} l(\eps_t;\theta_0)]\}^2 ] \\
& \leq & \cfrac{d^4_T}{\eps^2 T^2} \lambda_{\max}^2(\Hb_{T,\Ac \Ac}).
\end{array}
\eeqw

\mds

As for the third order term, by the Cauchy-Schwartz inequality
\beqw
\begin{array}{llll}
\|\nabla'\{(\hat{\theta}_{\Ac} - \theta_{0,\Ac})' \ddot{\Gb}_T l(\bar{\theta})_{\Ac \Ac} (\hat{\theta}_{\Ac} - \theta_{0,\Ac})\}\|^2_2 & \leq & \cfrac{1}{T^2} \overset{T}{\underset{t=1}{\sum}} \{ \underset{(k,l,m)\in \Ac}{\sum} \partial^3_{\theta_k \theta_l \theta_m} l^2_T(\eps_t;\bar{\theta}) \} \|\hat{\theta}_{\Ac} - \theta_{0,\Ac}\|^4_2 \\
& \leq & \cfrac{1}{T^2} \overset{T}{\underset{t=1}{\sum}}  \{ \underset{(k,l,m)\in \Ac}{\sum} \psi^2_T(\eps_t) \} \|\hat{\theta}_{\Ac} - \theta_{0,\Ac}\|^4_2 \\
& = & O_p(\cfrac{d^5_T}{T^2}) = o_p(\cfrac{1}{T}).
\end{array}
\eeqw

\mds

We now prove $X_{T,t} =  \sqrt{T} Q_T \Vb^{-1/2}_{T,\Ac \Ac} \Hb^{-1}_{T,\Ac \Ac} \dot{\Gb}_T l_t(\theta_0)_{\Ac}, t = 1,\cdots,T,$ is asymptotically normal by checking the Lindeberg-Feller's condition for applying Shiryaev's Theorem \ref{Shishi}. We remind that $\dot{\Gb}_T l_{T,t}(\theta_0)$ is the $t$-th point of the score of the empirical criterion. Let $\beta > 0$, and to use Shiryaev's Theorem, we need to prove that for any $\eps > 0$, we have
\beqw
\Pb(\overset{T}{\underset{t=0}{\sum}} \Eb[\|X_{T,t}\|^2_2 \mathbf{1}_{\|X_{T,t}\|_2 > \beta}| \Fc^T_{t-1} ]  > \eps) \underset{T \rightarrow \infty}{\longrightarrow} 0.
\eeqw
By the Markov inequality, we obtain
\beqw
\begin{array}{llll}
\Pb(\overset{T}{\underset{t=0}{\sum}} \Eb[\|X_{T,t}\|^2_2 \mathbf{1}_{\|X_{T,t}\|_2 > \beta}| \Fc^T_{t-1} ]  > \eps) & \leq & \frac{1}{\eps} \overset{T}{\underset{t=0}{\sum}} \Eb[\|X_{T,t}\|^2_2 \mathbf{1}_{\|X_{T,t}\|_2 > \beta}| \Fc^T_{t-1} ] \\
& \leq & \frac{1}{\eps} \overset{T}{\underset{t=0}{\sum}} \Eb[ \Eb[\|X_{T,t}\|^4_2| \Fc^T_{t-1}]^{1/2} \Pb(\|X_{T,t}\|_2 > \beta | \Fc^T_{t-1})^{1/2}] \\
& \leq & \frac{1}{\eps} \overset{T}{\underset{t=0}{\sum}} \Eb[ \{\frac{C_{st}}{T^2}\Eb[\|\nabla l(\eps_t;\theta_0) \nabla' l(\eps_t;\theta_0)\|^2_2| \Fc^T_{t-1}]\}^{1/2} \\
& & . \frac{1}{\beta}\Eb[\|\sqrt{T} Q_T \Vb^{-1/2}_{T,\Ac \Ac} \Hb^{-1}_{T,\Ac \Ac} \dot{\Gb}_T l_t(\theta_0)_{\Ac}\|^2_2 | \Fc^T_{t-1}]^{1/2}],
\end{array}
\eeqw
with $C_{st} > 0$. First, let $\Kb_T = Q_T \Vb^{-1/2}_{T,\Ac \Ac} \Hb^{-1}_{T,\Ac \Ac}$, we have
\beqw
\begin{array}{llll}
\Eb[\|\sqrt{T} \Kb_T \dot{\Gb}_T l_t(\theta_0)_{\Ac}\|^2_2 | \Fc^T_{t-1}] & = & \cfrac{1}{T} \Eb[\nabla' l(\eps_t;\theta_0) \Kb'_T  \Kb_T \nabla l(\eps_t;\theta_0)| \Fc^T_{t-1}] \\
& = & \cfrac{1}{T} \Eb[\text{Trace}(\nabla' l(\eps_t;\theta_0) \Kb'_T  \Kb_T \nabla l(\eps_t;\theta_0))| \Fc^T_{t-1}] \\
& = & \cfrac{1}{T} \text{Trace}(\Eb[\nabla l(\eps_t;\theta_0) \nabla' l(\eps_t;\theta_0)| \Fc^T_{t-1}] \Kb'_T  \Kb_T ) \\
& \leq & \cfrac{1}{T}\lambda_{\max}(\Hb^T_{t-1}) \tilde{C}_{st},
\end{array}
\eeqw
where $\tilde{C}_{st} > 0$. Furthermore, we have
\beqw
\begin{array}{llll}
\Eb[\|\nabla l(\eps_t;\theta_0) \nabla' l(\eps_t;\theta_0)\|^2_2| \Fc^T_{t-1}] & = & \Eb[\overset{d_T}{\underset{i,j=0}{\sum}} \{\partial_{\theta_i} l(\eps_t;\theta_0) \partial_{\theta_j} l(\eps_t;\theta_0)\}^2| \Fc^T_{t-1}] \\
& \leq & d^2_T \underset{i,j = 1,\cdots,d_T}{\sup} \Eb[\{\partial_{\theta_i} l(\eps_t;\theta_0) \partial_{\theta_j} l(\eps_t;\theta_0)\}^2 | \Fc^T_{t-1} ].
\end{array}
\eeqw
By assumption \ref{AN}, we have
\beqw
\begin{array}{llll}
\Pb(\overset{T}{\underset{t=0}{\sum}} \Eb[\|X_{T,t}\|^2_2 \mathbf{1}_{\|X_{T,t}\|_2 > \beta}| \Fc^T_{t-1} ]  > \eps) && \\
\leq \cfrac{C^{\frac{1}{2}}_{st} \tilde{C}^{\frac{1}{2}}_{st} d_T}{T^{\frac{3}{2}}} \overset{T}{\underset{t=0}{\sum}} \Eb[  \underset{i,j = 1,\cdots,d_T}{\sup} \Eb[\{\partial_{\theta_i} l(\eps_t;\theta_0) \partial_{\theta_j} l(\eps_t;\theta_0)\}^2 | \Fc^T_{t-1} ] \lambda_{\max}(\Hb^T_{t-1}) ] \leq \cfrac{C^{\frac{1}{2}}_{st} \tilde{C}^{\frac{1}{2}}_{st} \bar{B} T d_T}{T^{\frac{3}{2}}} .&&
\end{array}
\eeqw
Consequently, we obtain
\beqw
\overset{T}{\underset{t=0}{\sum}} \Eb[\|X_{T,t}\|^2_2 \mathbf{1}_{\|X_{T,t}\|_2 > \beta} | \Fc^T_{t-1}] = o_p(1).
\eeqw
We deduce that $X_{T,t}$ satisfies the Lindeberg-Feller condition, and by Theorem \ref{Shishi}, $\sqrt{T} Q_T \Vb^{-1/2}_{T,\Ac \Ac} \Hb^{-1}_{T,\Ac \Ac} \dot{\Gb}_T l(\theta_0)_{\Ac}$ is asymptotically normally distributed. The asymptotic distribution of Theorem \ref{oracle_double} follows.

\end{proof}
\newpage
\begin{section}{Simulation Experiments} \label{Sim}

In this section, we carry out a simulation study to explore the finite sample performance of the adaptive Sparse Group Lasso. We first focus on the calibration of the adaptive weights entering the penalties. The regularization parameters must satisfy conditions to achieve the oracle property in the double asymptotic case. To do so, we suppose $\lambda_T = T^{\beta}$ and $\gamma_T = T^{\alpha}$, where $\beta$ and $\alpha$ are both strictly positive constant. Regarding assumption \ref{H13}, we obtain the conditions
\beqw
\left\{\begin{array}{llll}
\alpha + \frac{c}{2} + \kappa \mu - \frac{1}{2} < 0, &&\\
\alpha - \frac{1}{2} + \frac{1}{2}[(1+\mu)(1-c)-1] > 0, &&\\
\beta + \kappa \eta - \frac{1}{2} < 0, && \\
\beta - \frac{1}{2} + \frac{1}{2}[(1+\eta)(1-c)-1] > 0, &&\\
(1+\mu)[1-\frac{c}{2} - \kappa \eta - \beta] + \alpha - 1 > 0. &&
\end{array}\right.
\eeqw
This system allows for flexibility when choosing $\mu$ and $\eta$ once $\kappa,c,\alpha$ and $\beta$ are fixed. For instance, for $c=1/6$, $\kappa = 0.05$, $\alpha = 1/10$ and $\beta = 1/10$, then $\mu \in [0.4,6.3]$ and $\eta \in [0.6,7.9]$. If $\alpha = \beta = 1/5$ and for $c=1/6$ and $\kappa = 0.05$, then $\mu \in [0.4,4.3]$ and $\eta \in [0.3,5.9]$.

\mds

We consider 6 methods in the experiment: the Lasso (L), the Adaptive Lasso (AL), the Group Lasso (GL), the Adaptive Group Lasso (AGL), the Sparse Group Lasso (SGL) and the Adaptive Sparse Group Lasso (ASGL).

\mds

There are several methods to numerically solve the non-differentiable statistical problem (\ref{DA_SGL}). Fan and Li (2001) proposed a local quadratic approximation (LQA) of the first order derivative of the penalty function and a Newton-Raphson type algorithm. To circumvent numerical instability, they suggest to shrunk to zero coefficients that are close to zero, that is a coefficient $|\theta_j|< \eps$, with $\eps >0$ to be calibrated. The drawback is that once it is set to zero, it will be excluded at any step of the LQA algorithm. Hunter and Li (2005) proposed a more sophisticated version of the LQA algorithm to avoid the drawback of the stepwise selection and numerical instability. They also studied the convergence properties of the LQA method. Zou and Li (2008) proposed a local linear approximation (LLA) of the penalty function such that the estimated coefficients have naturally a sparse representation, under the condition that the penalty function enjoys the continuity condition. Zou (2006) or Zou and Zhang (2009) use the LQA algorithm for their empirical study. Other approaches are also possible such as gradient descent methods.

\mds

When one consider the OLS loss function, closed form algorithm can be applied to our problem. B\"uhlmann and van de Geer (2011) compiled these methodologies for solving the Lasso and the Group Lasso using gradient descent methods for general penalized convex empirical function. We used these algorithms in our study for solving the group LASSO. As for the LASSO, we applied the shooting algorithm developed by Fu (1998), which is a particular case of the gradient descent method. Simon and al. (2013) proposed an algorithm for solving the SGL that can accommodate likelihood criteria. This is a "two-step" method, where we first check whether the group is active, and then, if active, check if the coefficient within this group is active. In this simulation study, we used the alternative direction method of multipliers provided by Li and al. (2014).

\mds

We used a cross-validation procedure to select both parameters $\lambda_T$ and $\gamma_T$ such that both terms are defined by $\lambda_T = T^{\beta}$ and $\gamma_T = T^{\alpha}$, and $ \beta = \alpha = 1/8$. The adaptive weights are computed as follows: we first compute an OLS estimator $\tilde{\theta}$ such that the adaptive weights entering the penalties correspond to $\tilde{\tilde{\theta}} = \tilde{\theta} + T^{-\kappa}$, with $\kappa = 0.2$. As for the adaptive weights, they are chosen such that the above system is satisfied: we set $\eta = 3.5$ and $\mu = 2.5$.

\mds

We report the variable selection performance through the number of zero coefficients correctly estimated, denoted as $C$ and, the number of nonzero coefficients incorrectly estimated, denoted $IC$. Besides, the mean squared error is reported as an estimation accuracy measure.

\emph{Simulated experiment.} We consider a data generating process
\beqw
y = \underset{l}{\sum}\beta^{(l)}_0 \mathbf{X}^{(l)} + \sigma \eta,
\eeqw
where $\eta$ is a strong white noise, normally distributed, centered with unit variance and $\sigma = 0.3$. The matrices $\mathbf{X}^{(l)}$ follow $\cc_l$- dimensional multivariate normal distributions, centered and with variance covariance $\Sigma^{(l)}$ such that the entries are defined as $\Sigma^{(l)}_{ij} = \rho^{|i - j|}, 1 \leq j, i \leq \cc_l$. The correlation parameter $\rho$ is randomly chosen among $\{0.5,0.8,0.9\}$. Moreover, the dimension $d_T = [ x \times T^{1/6} ]$ with $T = 500, 2000, 4000$ and $x = 10,30,50$ respectively for the values of $T$. As $d_T = O(T^c)$ with $c=1/6$, we can multiply by $x$ to consider more realistic settings. The number of groups is defined as $N_g = 4$ (resp. $N_g = 8$, resp.  $N_g = 18$) for $n = 500$ (resp. for $n = 2000$, resp. for $n = 4000$) and the size of each of them is randomly chosen among $\{5,\cdots,30\}$. The number of active groups is defined as $|\Sc| = 2 a_T$ with $a_T = [N_g/3]$. Moreover, zero coefficients are randomly chosen among the whole vector $\beta$ for active groups, such that the total number of zeros -both the zero subvectors for inactive groups and zero components for active groups - matches the total number of inactive indices. The total number of active indices is defined as $|\Ac| = 3 b_T$ with $b_T=[ d_T/9 ]$. Finally, we generate the active indices among a uniform law $\Uc([0.1,0.99])$. Zou and Zhang (2009) experiment influenced our framework.

\begin{table}[h]\centering\caption{\emph{Model selection and precision accuracy based on 100 replications}}
\begin{tabular}{c c c c c c c c c}\hline \hline $T$ & $d_T$ & $N_g$ & $|\Sc|$ & $|\Ac|$ & Model & MSE & C & IC  \\
\hline
500 & 28 & 4 & 2 & 9 & Truth &  & 19 & 0 \\
    &   &    &    &   &  Lasso & 0.0178   &  13.13 &  0 \\
    &   &    &    &   &  aLasso & 0.0118  &  17.98  & 0  \\
    &   &    &    &   &  GLasso & 0.0146 & 12.77  &  0 \\
    &   &    &    &   &  AGLasso & 0.0129 &  13.57  & 0  \\
    &   &    &    &   &  SGL & 0.0183 &  12.97 &  0 \\
    &   &    &    &   &  ASGL & 0.0101 & 18.83  &  0 \\
    &   &    &    &   &      &    &   & \\

2000 & 106 & 8 & 4 & 33 & Truth &  & 73 & 0 \\
    &   &    &    &   &  Lasso & 0.0118 &  49.65  & 0  \\
    &   &    &    &   &  aLasso & 0.0103 & 70.95   & 0  \\
    &   &    &    &   &  GLasso & 0.0150 & 57.48   & 0  \\
    &   &    &    &   &  AGLasso & 0.0160 & 60.78  &  0 \\
    &   &    &    &   &  SGL & 0.0125 &  58.88  &  0 \\
    &   &    &    &   &  ASGL & 0.0095 & 72.70  & 0  \\
    &   &    &    &   &      &    &   & \\

4000 & 199 & 18 & 12 & 66 & Truth &  & 133 & 0 \\
    &   &    &    &   &  Lasso & 0.0105  &  87.17  &  0 \\
    &   &    &    &   &  aLasso & 0.0093 & 131.33   & 0  \\
    &   &    &    &   &  GLasso & 0.0140 &  113.42  &  0 \\
    &   &    &    &   &  AGLasso & 0.0150 & 113.17   & 0  \\
    &   &    &    &   &  SGL & 0.0102 & 98.92   &  0 \\
    &   &    &    &   &  ASGL & 0.0094 & 133  & 0  \\

\hline
\hline\end{tabular}
\end{table}

We can highlight some interesting remarks from this simulation study. First, the adaptive versions of the Lasso, the Group Lasso or the SGL outperfom their non adaptive versions. The difference is significant for the adaptive Lasso and the adaptive SGL. This is in line with the asymptotic theory. The adaptive SGL performs well as it can discard inactive groups and inactive indices among active groups and outperform other adaptive penalization methods.

\end{section}

\newpage
{\bf Acknowledgements.} I would like to thank Jean-Michel Zako\"ian and Christian Francq for all the theoretical references they provided. And I thank warmly Jean-David Fermanian for his significant help and helpful comments. I finally gratefully acknowledge the Ecodec Laboratory.

\end{document}